\documentclass[12pt,reqno]{amsart}
 
\usepackage[left=1in, right=1in, top=1.1in,bottom=1.1in]{geometry}
\usepackage[dvipsnames]{xcolor}
\usepackage[mathscr]{eucal}
\usepackage{latexsym,bm,amsfonts,amssymb,pifont,  bm,  stmaryrd}  
\usepackage{graphicx}   
\usepackage{color}
\usepackage{hyperref}   
\hypersetup{   
     colorlinks   = true, 
     citecolor    = blue, 
     linkcolor    = blue
}

\usepackage{pdfsync}
\usepackage[font={scriptsize}]{caption}

\usepackage{enumerate}
\usepackage{pdfsync}
\newcommand{\ep}{\varepsilon}

\newtheorem{theorem}{Theorem}[section]
\newtheorem{Def}[theorem]{Definition}
\newtheorem{notation}[theorem]{Notation}
\newtheorem{thm}[theorem]{Theorem}
\newtheorem{prop}[theorem]{Proposition}
\newtheorem{cor}[theorem]{Corollary}
\newtheorem{lemma}[theorem]{Lemma}

\newtheorem{hyp}[theorem]{Hypothesis}
\theoremstyle{remark}
\newtheorem{remark}[theorem]{Remark}
\newtheorem{example}[theorem]{Example}
\numberwithin{equation}{section}

\newcommand{\id}{\text{Id}}

\def\RR{\mathbb{R}}

\def\NN{\mathbb{N}}
\def\mE{\mathbb{E}}

\def\ZZ{\mathbb{Z}}

\newcommand{\cc}{{\mathcal C}}
\newcommand{\cd}{{\mathcal D}}

\newcommand{\ch}{{\mathcal H}}
\newcommand{\ci}{{\mathcal I}}
\newcommand{\cj}{{\mathcal J}}

\newcommand{\cl}{{\mathcal L}}

\newcommand{\cp}{{\mathcal P}}

\newcommand{\cs}{{\mathcal S}}

\newcommand{\cv}{{\mathcal V}}

\newcommand{\cx}{{\mathcal X}}

\def\si{\sigma}

\def\al{{\alpha}}

\def\be{{\beta}}

\def\ga{{\gamma}}

\newcommand{\lp}{\left(}
\newcommand{\rp}{\right)}

\def \eref#1{\hbox{(\ref{#1})}}

\def\lll{\llbracket}
\def\rr{\rrbracket}

\def \eref#1{\hbox{(\ref{#1})}}

\makeatletter
\def\l@subsection{\@tocline{2}{0pt}{2.5pc}{5pc}{}}
\makeatother

\begin{document}
\title[Discrete rough path]
{Discrete rough paths and limit theorems}
\date{}   

\author[Y. Liu \and S. Tindel]{Yanghui Liu \and Samy Tindel} 

 
 \address{Yanghui Liu, Samy Tindel: Department of Mathematics,
Purdue University,
150 N. University Street,
W. Lafayette, IN 47907,
USA.}
\email{liu2048@purdue.edu, stindel@purdue.edu}

\thanks{S. Tindel is supported by the NSF grant  DMS-1613163}

 \begin{abstract}
 In this article, we consider   limit theorems for some weighted type random sums (or discrete rough integrals). We   introduce a general transfer principle from limit theorems  for unweighted sums to limit theorems for weighted sums  via rough path techniques. As a  by-product, we provide  a natural explanation  of the various new asymptotic behaviors in contrast with the classical unweighted random sum case. We apply our principle to derive some weighted type Breuer-Major theorems, which generalize  previous results to random sums that do not have to be in a finite sum of chaos.  In this context, a  Breuer-Major type criterion in notion of Hermite rank is obtained. We also  consider some applications to   realized power variations and to     It\^o's formulas in law. In the end, we study the asymptotic behavior of   weighted quadratic variations for  some   multi-dimensional Gaussian processes.  
\end{abstract}

\maketitle 
\vspace{-.6cm}
{
\hypersetup{linkcolor=black}
 \tableofcontents 
}

\section{Introduction} \label{section1}
 Let $\cd_{n}: 0=t_{0}<t_{1}<\cdots<t_{n}=1$ be  a      partition on $[0,1]$.  Take a ``1-increment'' process  $h^{n}_{st}$ defined for $s,t\in \cd_{n} $ such that $s\leq t $ and  a ``weight'' process $y_{t}$ defined for $t\in \cup_{n\in\NN}\cd_{n}$. We consider a ``discrete integral'' as a Riemann sum of the form: 
 \begin{eqnarray}\label{eq:def-j-y-hn}
\cj_{s}^{t}(y; h^{n}):= \sum_{s\leq t_{k}<t} y_{t_{k}} h^{n}_{t_{k}t_{k+1}}. 
\end{eqnarray}
  Recall that a classical limit theorem for such a process is a statement of the type: 
 \begin{eqnarray}\label{eq:ex-clt}
\frac{1}{a_{n}}\cj(1; h^{n})=\frac{h^{n}}{a_{n}} \xrightarrow{ \ \ \ } \omega ,
\quad \text{as } n\to \infty.
\end{eqnarray}
  Here  $a_{n}$ is an increasing sequence such that $\lim_{n\to \infty}a_{n}=\infty$,  $\omega $ is a non-zero  continuous process and the limit is   usually understood as a finite dimensional distribution limit. A typical example of \eqref{eq:ex-clt} is the convergence of a renormalized random walk to Brownian motion (Donsker's theorem, see \cite{JS}), but a wide range of more complex situations can occur.
 Indeed, it is well-known that the rate of growth of $a_{n}$ and  the nature of the limit process $\omega$ are determined by both the marginal tails of $h^{n}$ and its dependence structure; see e.g. \cite{DR, DM, R, T} . 
The limit process $\omega$ is necessarily      self-similar; see \cite{L}. 

In this paper we are interested in the following   related problem: 

\noindent
{\bf Problem 1.} Given that $h^{n}$ converges to some ``1-increment'' process, say, the increment of a Wiener process, what is the asymptotic behavior of the discrete integral $\cj(y; h^{n})$ for a general weight $y$, and when would (or would not) the asymptotic behavior of $\cj(y; h^{n})$ be similar to that of $h^{n}$? 

This   problem has   drawn a lot of attention in recent articles  due to its  essential role in   topics such as normal approximations (e.g. \cite{Nourdin, NN, NNT}), time-discretization based numerical approximations (e.g. \cite{GN, HLN1, LT}),  parameter estimations (e.g. \cite{BCP, CNW, LL, Norvaisa}),  and  the so-called \emph{It\^o's formula in law}  (e.g. \cite{BNN, BS, GNRV, GRV,  HN, HN2, HN3, NR, NRS}).
Let us, however, point out  several limitations in the existing results: 
(1) Each process $h^{n}$ is usually a functional  of a   Gaussian process $x$ with stationary increments, living in a fixed   finite sum of chaos; 
(2)~The underlying Gaussian process $x$ is   one-dimensional; 
(3)  Only the special case $y_{t}=f(x_{t})$, $t\geq 0$ is  considered for the weight function. 
(4) To the best of our knowledge,   there is no theoretical   explanation for  the  various ``unexpected'' asymptotic behaviors of the discrete integral observed in e.g. \cite{BS, GRV, NR, N} so far. 
(5)  Satisfactory   general criteria
of convergence for sequences of discrete weighted integrals are still rare.
This is in sharp contrast with the simple Breuer-Major type conditions  in the unweighted case.

The aim of the current paper is
thus to give an account on limit theorems for discrete integrals thanks to   rough paths techniques combined with Gaussian analysis. 
In our setting, we will consider a general $1$-increment process $h^{n}$ and a general weight process $(y, y', \dots, y^{(\ell-1)})$ with $y_{0}=0$ 
which is controlled by the increments of some rough path $x=(1, x^{1},\dots, x^{\ell-1})$. Here $\ell$ is some constant in $\NN$.  
Notice that we will define the notion of controlled process later in the paper,  see Definition \ref{def2.2} below, but we can observe that this class of paths includes functions of the form $y=f(x)$ or solutions of differential equations driven by $x$. 
Let us   label the following   hypothesis:
\begin{hyp}\label{hyp1.1}
Take $i=0,1,\dots, \ell-1$. 
For any partition $ 0\leq s_{0} <s_{1}<\cdots<s_{m }\leq 1 $   of $[0,1]$ such that $m\ll n$ and 
$ |s_{j+1} - s_{j}| \leq \frac{1}{m}$, 
 we have 
\begin{eqnarray} \label{e1.3i}
\lim_{m\to \infty} \limsup_{n\to \infty} \Big| \sum_{j=0}^{m-1}     \cj_{s_{j}}^{s_{j+1}} ( x^{i}; h^{n} ) \Big| &=&0,
\end{eqnarray}
  where $\cj(x^{i}; h^{n})$ is defined by \eqref{eq:def-j-y-hn} and the limit is understood as a limit in probability. 

\end{hyp}

  We will be able to prove the following limit theorems   (see Theorem \ref{thm5.9} for a more precise statement):

 \begin{theorem} \label{thm1.1}
Consider an underlying rough path $x=(1, x^{1}, x^{2},\dots, x^{ \ell-1 })$, where $\ell$ is some constant in $\NN$ (depending on the regularity of $x$).  Let   $(y, y', \dots, y^{(\ell-1)})$  be a process on $[0,1]$ whose increments are \emph{controlled} by $x$, and assume that  $h^{n}$ satisfies some proper regularity hypothesis.
 Suppose that the following assumptions are fulfilled:
 
\noindent (i) We have the convergence    $ (x ,    h^{n}        )
  \xrightarrow{f.d.d.} ( x ,  W    )$ as $n\to \infty$. Here and in the following 
  $W$ is a standard Brownian motion independent of $x$, and f.d.d. stands for the  finite dimensional distribution limit.

\noindent (ii) Hypothesis \ref{hyp1.1} holds true  for $i=1,\dots, \ell-1$. 

Then we have the following convergence in distribution for the process $y$:
\begin{eqnarray} \label{e1.3}
(x, \cj  (y ; h^{n}) ) &\xrightarrow{f.d.d.}& (x, v), 
\end{eqnarray}
where the integral $v_{st}=\int_{s}^{t} y_{u}   dW_{u}$ has to be understood as a conditional Wiener integral.
 \end{theorem}

As alluded to above, Theorem \ref{thm1.1} can be seen as a general principle which allows to transfer limit theorems \eref{e1.3i} taken on monomials of the rough path to the corresponding limit theorems involving controlled processes as weights. Therefore, potential applications of this result are numerous (see the aforementioned parameter estimation problem, It\^o's formula in law, or numerical schemes for rough differential equations), and will be detailed throughout the paper. 

As has already been observed in   \cite{GRV, Nourdin, NNT}, the asymptotic behavior of \eref{eq:def-j-y-hn} can be completely different from \eref{e1.3}.  One of the first occurrences of this kind of result is provided by \cite{Nourdin}, where for a one-dimensional fractional Brownian motion $x$ with Hurst parameter $\nu\in (0, \frac14)$
the following limit theorem is obtained: consider the increment $h^{n}_{st} = \sum_{s\leq t_{k}<t}[( n^{\nu} \delta x_{t_{k}t_{k+1}} )^{2}-1]$, where $\delta x_{t_{k}t_{k+1}} = x_{t_{k+1}} - x_{t_{k}}$ and $(t_{k})_{k=0,\dots,n}$ stands for the  uniform partition of $[0,1]$. Let $f$ be a continuous function with proper regularity. Then, as  $n\to \infty$, we have:
\begin{eqnarray}\label{e1.4}
n^{2\nu -1}\cj_{0}^{1} ( f(x); h^{n} ) \xrightarrow{L_{2}} \frac14 \int_{0}^{1} f'' (x_{s}) ds. 
\end{eqnarray}
   Our   approach   allows  to generalize Theorem \ref{thm1.1} to
   handle limits such as \eref{e1.4}, weighted by controlled processes. In addition, our results   provide  an    explanation  of the appearance of $f''$ in the right-hand side of \eref{e1.4},  based on the structural understanding of the discrete integral from  the rough path theory. Indeed, our next theorem shows that  the limit  $\frac14 \int_{0}^{1}f''(x_{s})ds$  is the result of a ``speed match'' between different levels $(1, x^{1},  \dots, x^{\ell-1} )$ of the rough path $x$ and the fact that $f(x)$, $f'(x)$, \dots, $f^{(\ell-1)}(x)$ are the corresponding weight processes. Specifically, we shall get the following limit theorem (see Theorem \ref{cor6.3} for a multidimensional version).

\begin{thm} \label{thm1.2} 
Let   $y$,  $x$ and   $h^{n}$ be processes defined as in Theorem \ref{thm1.1}, and recall that $\cj( y ; h^{n})$ is the increment  defined by~\eref{eq:def-j-y-hn}. Suppose that $x$ and $h^{n}$ verify the following assumptions:

\noindent
(i) There is   some $\tau \in \{1,\dots, \ell-1\}$ such that $   \cj_{s}^{t} (x^{\tau}; h^{n}) \rightarrow     (t-s) \varrho $ in probability  for all $s,t\in \cd_{n}$ such that $s<t$, where $\varrho  $ is some constant, and  
$
  \cj_{s}^{t}(x^{i};  h^{n}) 
   \to  0 
$
 in probability for all~$i<\tau $. 
 
\noindent (ii) Hypothesis \ref{hyp1.1} holds true for $i=\tau+1,\dots, \ell-1$.

Then   the following convergence holds true in probability:
 \begin{eqnarray*} 
\lim_{n\to \infty}   \cj (y ; h^{n}) =    \Big( \int  y^{(\tau)}_{t}d {t}\Big)  \varrho\,.
\end{eqnarray*}
 \end{thm}

A more   complicated situation of asymptotic behavior  is observed in \cite{BS, NN, NR}. 
This usually corresponds to a transition in terms of roughness for the underlying rough path $x$. 
 For example, in the critical cases   when $\nu=\frac14$ in \eref{e1.4},   and for the same $f$ and $h^{n}$ as in \eref{e1.4}, one obtains the convergence:
\begin{eqnarray}\label{eq:intro-double-limit}
n^{-1/2}\cj_{0}^{1} ( f(x); h^{n} ) \xrightarrow{ \ d \ } \si \int_{0}^{1} f(x_{s}) dW_{s} + \frac14 \int_{0}^{1} f'' (x_{s}) ds,
\end{eqnarray}
where $\si$ is some constant and recall that $W$ is a standard Brownian motion independent of~$x$. 
An explanation of the above asymptotic behavior according to the technique of rough path is that the two levels $1$ and $x^{2}$
give contributions of the same order in the limit theorem. This is then reflected into the fact that the components $f(x)$ and $f''(x)$ (respectively, $0$th and $2$nd derivatives of $f(x)$ as a controlled process) give contributions of the same order. 
  Our generalization of \eqref{eq:intro-double-limit} is thus the following ``double'' limit theorem.

  \begin{theorem}\label{thm1.3}
  Let   $y$, $x$ and $h^{n}$  be processes defined as in Theorem \ref{thm1.1}. 
  Furthermore, we assume that $x$ and $h^{n}$ fulfill the following conditions:
  
 \noindent (i) 
We have the convergence    $ (x ,    h^{n}        )
  \xrightarrow{f.d.d.} ( x ,  W    )$ as $n\to \infty$.

\noindent (ii)
There exists a constant   $\varrho  $ and some $  \tau : 0<\tau<\ell$ such that 
 for any partition $ 0\leq s_{0} <s_{1}<\cdots<s_{m }\leq 1 $   on $[0,1]$ satisfying $m\ll n$, 
$  |s_{i+1} - s_{i}| \leq \frac{1}{m}$, and $s_{0}= s$, $s_{m}= t$, 
 we have the convergence in probability:
\begin{eqnarray}\label{e1.7}
\lim_{m\to \infty} \lim_{n\to \infty}   \sum_{j=0}^{m-1}     \cj_{s_{j}}^{s_{j+1}} ( x^{\tau}; h^{n} )  =   (t-s)\varrho.
\end{eqnarray}
 
\noindent
(iii)
  Hypothesis \ref{hyp1.1} holds true for $i\in \{ 1,\dots, \ell-1 \}\setminus \{\tau\}$.

 Then  we obtain the following limit  for the increment $\cj (y; h^{n})$:
 \begin{eqnarray*}
  (x, \cj  (y ; h^{n}) ) \xrightarrow{f.d.d.} \lp x,  \int  y_{t }    dW_{t} +  \Big(\int    y^{(\tau)}_{t}d {t}\Big)   \varrho \rp.
\end{eqnarray*}
  \end{theorem}

As mentioned previously, Theorems \ref{thm1.1}, \ref{thm1.2} and \ref{thm1.3} are abstract transfer principles from monomials of a rough path to a controlled process for limit theorems of the form \eqref{eq:ex-clt}.
 For sake of illustration, let us mention an important application of this transfer principle  we will encounter in the article, namely a weighted type Breuer-Major theorem. 
 
 Recall that the Hermite polynomial of order $q$ is defined as $H_{q}(t)= (-1)^{q} e^{\frac{t^{2}}{2}} \frac{d^{q}}{dt^{q}}  e^{-\frac{t^{2}}{2}}$, and we denote by $\ga$ the standard normal distribution. We consider the following Breuer-Major type criterion:
 \begin{hyp}\label{hyp2}
 Take $\ell\in \NN$. Let $f\in L_{2}(\ga)$ be a function such that we have the expansion $f = \sum_{q=d}^{\infty} a_{q}H_{q} $ for a given $d\geq 1$ and $a_{d}\neq 0$. We suppose that the coefficients $a_{q}$ satisfy:
\begin{eqnarray}\label{e1.1}
\sum_{q=d}^{\infty}  a_{q}^{2}q! q^{2(\ell-1)}<\infty.
\end{eqnarray} 

 \end{hyp}
Following  is our weighted type Breuer-Major theorem: 
  \begin{thm}\label{thm:BM-with-weight}
Let $x$ be a one-dimensional fBm with Hurst parameter $\nu < \frac12$. Suppose that Hypothesis \ref{hyp2} holds true for $f \in L_{2}(\ga)$ with  some $\ell \in \NN$ and $d \geq 1$.  
 Let $(y,y',\dots, y^{\ell-1})$ be a   process controlled by $x$.   We define 
 a sequence $\{h^{n}; n\geq 1\}$ of
   increments by  $h^{n}_{st} := n^{-1/2} \sum_{s\leq t_{k}<t} f(n^{\nu} \delta x_{t_{k}t_{k+1}})$ for $s,t$ in the partition  $\cd_{n}$, such that $s<t$. 
 
\noindent (i) When $d>\frac{1}{2\nu }$, and $\ell$ is the smallest integer such that $  \ell >\frac{1}{2\nu}$, we have the   convergence:  
\begin{eqnarray*}
(x, \cj(y; h^{n}))  
 \xrightarrow{f.d.d.}    \lp x, \si_{\nu, d} \int  y_{t} dW_{t} \rp \,,\quad\quad  \text{as \ }   n\to \infty,
\end{eqnarray*}
where $\si_{\nu, d}  $ is a   constant which can 
be computed explicitly
 and where we recall that $\cj(y; h^{n})$ is defined by \eqref{eq:def-j-y-hn}.

 \noindent (ii) When $d=\frac{1}{2\nu}$ and     $\ell=d+1$,   the following convergence holds true:
\begin{eqnarray*} 
(x, \cj(y; h^{n}) )
 \xrightarrow{f.d.d.}  \Big(x,  \si_{\nu, d}    \int  y_{t} dW_{t} +\Big(-\frac12\Big)^{d}a_{d} \int y^{(d)}_{u} du \Big), \quad\quad  \text{as \ }   n\to \infty.
\end{eqnarray*}

\noindent (iii) When $d<\frac{1}{2\nu}$ and     $\ell=d+1$,  we have the    convergence in probability:
\begin{eqnarray*}
 n^{-( \frac12-\nu d)} \cj_{s}^{t}(y; h^{n}) 
 \xrightarrow{ \   \ }   \Big(-\frac12\Big)^{d} a_{d}\int_{s}^{t}y^{(d)}_{u} du, \quad\quad  \text{as \ }   n\to \infty.
\end{eqnarray*}

  \end{thm}
 
 The proof of Theorem \ref{thm:BM-with-weight} is based on our general Theorem \ref{thm1.1}--\ref{thm1.3}.  
 Let us observe that Theorem~\ref{thm:BM-with-weight} improves on the references on weighted Breuer-Major theorems quoted above in the following ways:
 
 \noindent
 (i) The function $f$ is not assumed to be in a finite sum of chaos. In fact a convenient sufficient condition for \eref{e1.1} to be fulfilled is that the function
  $f$ is an element of  $C^{2\ell-2}$. 
  
 \noindent
  (ii) Multidimensional versions of Theorem  \ref{thm:BM-with-weight} (based on \cite{A}) are easily conceived,  where $f(n^{\nu} \delta x_{t_{k}t_{k+1}})$ in the definition of $h^{n}$ is replaced by $f(n^{\nu} \delta x^{1}_{t_{k}t_{k+1}},\ldots,n^{\nu} \delta x^{d}_{t_{k}t_{k+1}})$, for a $d$-dimensional Gaussian process $(x^{1},\ldots,x^{d})$.   

 \noindent
(iii) The weight $y$ in Theorem \ref{thm:BM-with-weight} is obviously a controlled process instead of a mere function of $x$. It is worth noting again that the class of controlled processes includes solutions of differential systems driven by $x$.

 \noindent
(iv) As mentioned above,    the single and double limiting phenomenons   
in Theorem \ref{thm:BM-with-weight} can be explained 
 in terms of speed match on different levels of the rough path above $x$.

 \noindent
(v)  The solution to Problem 1 above is expressed easily in terms of the Hurst parameter $\nu$ of $x$ and the \emph{Hermite rank} $d$ of $f$.

 \noindent
Throughout the paper we will give an account on other applications of our general Theorems~\ref{thm1.1}--\ref{thm1.3}, such as realized power variations, convergence of trapezoidal Riemann sums and quadratic variations of multidimensional Gaussian processes. As the reader might see, the improvements (i)-(v) mentioned above will be a constant of our rough paths method.
  
Let us briefly explain the general methodology we have followed for our proofs, separating the general principle from the applications.

\noindent
(a)   The proof of Theorems~\ref{thm1.1}--\ref{thm1.3} is mostly based on rough path type expansions for the weight process $y$ and a more classical coarse graining argument (also called big block/small block in the literature). By handling the remainder terms thanks to  rough paths techniques,  the convergence of $\cj(y; h^{n})$ is reduced to those of $\cj(y; \zeta^{1})$, $\cj(y'; \zeta^{2})$, \dots, $\cj(y^{(\ell-1)}; \zeta^{\ell})$, 
where each $\zeta^{i}$ is a discrete process of the form $\zeta^{i}_{j} = \cj_{s_{j}}^{s_{j+1}} (x^{i}; h^{n})$. The convergence of these quantities are   further reduced to those of $\cj(1; \zeta^{1})$, $\cj(1; \zeta^{2})$, \dots, $\cj(1; \zeta^{\ell})$, such as those in Hypothesis \ref{hyp1.1} and relation \eref{e1.7}. The random processes $\cj(1; \zeta^{1})$, $\cj(1; \zeta^{2})$, \dots, $\cj(1; \zeta^{\ell})$ will be the elementary bricks for our limiting procedures.

\noindent
(b) Our applications, such as the 
weighted type Breuer-Major Theorem \ref{thm:BM-with-weight},  heavily rely on the criteria developed in Theorem \ref{thm1.1}--\ref{thm1.3}. This ingredient is combined with some Malliavin calculus techniques in order to handle the building bricks $\cj(1; \zeta^{i})$. More specifically, in case of the weighted  Breuer-Major theorem \ref{thm:BM-with-weight}, we shall invoke integration by parts on the Wiener space. This step is similar to what is done in \cite{NN}. However, due to our rough path reduction of the problem, we only have to consider integration by parts to compute moments of the elementary bricks $x^{i}_{t_{k}} H_{q}( n^{\nu } \delta x_{t_{k}t_{k+1}})$ (as opposed to $g(x_{t_{k}}) f( n^{\nu } \delta x_{t_{k}t_{k+1}} )$ for a general nonlinear function $g$). This reduction to computations in finite chaos is one of the crucial steps which allow to derive the Breuer-Major type criteria~\eref{e1.1} for a general function $f$. 

The paper is organized as follows. In Section \ref{section2} we introduce the  concept of discrete   rough paths and discrete rough integrals and recall some basic results of the rough paths theory. In Section  \ref{sectionLT1}, we prove our general limit theorems including Theorem \ref{thm1.1}, Theorem \ref{thm1.2} and Theorem \ref{thm1.3}. In Section \ref{sectionBM}, we apply them to the one-dimensional fractional Brownian motion, which allows us to derive a weighted type Breuer-Major theorem. We also consider   applications of the weighted type Breuer-Major theorem to parameter estimation and  It\^o's formula in law.  In Section \ref{sectionGuansian}, we consider the limit theorem of a weighted quadratic variation in the multi-dimensional Gaussian setting.

 \noindent
{\bf Notations:} 
For simplicity, we consider uniform partitions, that is, we  denote $t_{k} = \frac{k}{n}$ for each $k, n \in \NN$.   Take $s,t\in [0,1]$.   We denote by $\cs_{k}(s,t)$ the simplex $\{ (t_{1},\dots, t_{k}) \in  [0,1]^{k} ;\, s\leq t_{1}\leq \cdots\leq t_{k} <t \}$, and for simplicity we will write $\cs_{k} $ for $ \cs_{k}(0,1)$. In contrast, whenever we deal with a discrete interval $[s,t)\cap \cd_{n}$, we set $\cs_{k}'(s,t)=
\{ (t_{1},\dots, t_{k}) \in \cd_{n}^{k} ;\, s\leq t_{1}< \cdots< t_{k}<t  \}$, and similarly, when $s=0$ and $t=1$ we simply write  $\cs_{k}'  $.

Throughout the paper we work on a  probability space $(\Omega, \mathscr{F}, P)$. If $X$ is a random variable, we denote by $| X |_{L_{p}}  $ the $L_{p}$-norm of $X$.
The letter $K$ stands for a constant which can change from line to line. The letter $G$ denotes a generic a.s. finite random variable. We denote by $\lfloor a \rfloor $   the integer part of   $a$.

\section{Discrete rough paths} \label{section2}
In this section, we introduce the concept of discrete rough paths and discrete rough integrals, and recall some basic results of the rough paths theory. Then we derive our main estimates on   discrete rough integrals. 

\subsection{Definition and algebraic properties}
This subsection is devoted to introduce the main rough paths notations which will be used in the sequel.  

Let $\cv$ be a finite dimensional vector space.
We denote by $\cc_{k}(\cv)$ the set of functions $g : \cs_{k} \to \cv$ such that $g_{t_{1}\cdots t_{k}} = 0$ whenever $t_{i} = t_{i+1}$ for $i\leq k -1$. Such a function will be called a $(k-1)$-\emph{increment}. We define the operator $\delta$ as follows:
\begin{eqnarray*}
\delta : \cc_{k}(\cv) \to \cc_{k+1}(\cv), \quad\quad (\delta g)_{t_{1}\cdots t_{k+1}} = \sum_{ i=1}^{k+1} (-1)^{i} g_{t_{1}\cdots \hat{t}_{i} \cdots t_{k+1}}\,,
\end{eqnarray*}
where $\hat{t}_{i}$ means that this particular argument is omitted. 
For example, for $f\in \cc_{1}(\cv)$ and $g\in \cc_{2}(\cv)$ we have 
\begin{eqnarray}\label{e2.1}
\delta f_{st} = f_{t}-f_{s}
\quad
\text{ and }
\quad
\delta g_{sut} = g_{st}-g_{su}-g_{ut}.
\end{eqnarray} 
 A fundamental property of $\delta$, which is easily verified, is  that $\delta \delta =0$, where $\delta \delta $ is considered as an operator from $\cc_{k}(\cv) $ to $\cc_{k+2} (\cv)$.

Let us now introduce the notion of rough path which will be used throughout the paper.
\begin{Def}\label{def:rough-path}
Consider $\nu \in (0,1)$, $\ell \in \NN $ such that $\ell\leq \lfloor \frac{1}{\nu} \rfloor$ and $p>1$. 
Let 
  $x= (  x^{1},\dots, x^{\ell})$ be a continuous path on   $\cs_{2}$ and with values  in $   \oplus_{k=1}^{\ell} (\RR^{d})^{\otimes k }  $. For $p>0$   set 
\begin{equation}\label{eq:def-holder-seminorms}
    |x^{k} |_{[s,t], \,p\,, \nu }:=\sup_{(u,v)\in\cs_{2}([s,t])}\frac{ |  x^{k}_{uv} |_{L_{p}}^{1/k} }{|v-u|^{\nu  }}  ,
\end{equation}
and   define a $\nu$-H\"older semi-norm as follows:
\begin{equation}\label{eq:def-norm-rp}
|x|_{p\,,\nu} := |x^{1}|_{p\,,\nu}+ \cdots+ |x^{\ell}|_{p\,,\nu}\,. 
\end{equation}

We call $x$ a \emph{$(L_{p}, \nu, \ell)$-rough path} (or simply a rough path) if  the following properties holds true:

\noindent (1) the semi-norms $|x^{k}|_{[s,t], p , \nu}$ in \eqref{eq:def-holder-seminorms} are finite. In this case  we say that $x^{k}$, $k=1,\dots, \ell$   are respectively in $C^{  \nu} (\cs_{2} , (\RR^{m})^{\otimes k})$.
For convenience, we denote   $ |x^{k} |_{p\,, \nu }:=  |x^{k} |_{[0,1],\, p\,, \nu } $.

\noindent (2) For all $k\in \{ 2,\dots, \ell \}$,  $x^{k}$ satisfies the identity 
\begin{eqnarray}\label{e2.3}
\delta x^{k}_{sut} &=& \sum_{j=1}^{k-1} x^{k-j}_{su}\otimes x^{ j}_{ut}.
\end{eqnarray}
 
\end{Def}

\begin{remark}
Our definition of rough path differs slightly from the usual one in several aspects:

\noindent (i) We don't impose $\ell = \lfloor \frac{1}{\nu} \rfloor$, so that the order of our rough path might be lower than in the standard theory. In the sequel we will introduce another parameter $\al\in (0,1)$ such that $\nu \ell+\al>1$. 

\noindent (ii) We consider a rough path with values in $L_{p}$, and measure its regularity by looking at increments of the form $|x^{k}_{st}|_{L_{p}}$ for $(s,t) \in \cs_{2}$.  
\end{remark}

 In this paper, we are mostly concerned with discrete sums. Recall that we are considering discrete simplexes related to partitions of $[0,1]$, which are denoted by $\cs_{2}'$. We now introduce a general notion of discrete controlled process.

   \begin{Def}\label{def2.2}
  Fix $\al>0$ and  let $\ell$ be the smallest integer such that $  \nu \ell +\al>1 $.   Let $\cv $ be some finite dimensional vector space.
    Let $y,y',y'',\dots, y^{(\ell-1)}$  be continuous processes on $[ 0,1] $ such that $y_{0}=y^{(0)}_{0} = 0$. For convenience, we will also write: $y^{(0)} = y$, $y^{(1)} = y'$, $y^{(2)}= y''$,\dots. 
   Suppose that 
     $y  $ takes values in $\cv$, and $y^{(k)}  $ takes values in $\cl ( (\RR^{d})^{\otimes k},  \cv )$ for all   $k=1,\dots,\ell-1$. For $(s,t) \in \cs_{2} $ and  $k=0,1,\dots, \ell-2$ we denote 
   \begin{eqnarray}\label{e2.4}
 r^{(k)}_{st} &=& \delta y^{(k)}_{st} - y^{(k+1)}_{s} x^{1}_{st}-\cdots  -y^{(\ell-1)}_{s} x^{\ell-k-1}_{st},
\end{eqnarray}
 and $r^{(\ell-1)}_{st} = \delta y^{(\ell-1)}_{st} $.
We call $(y^{(0)},\dots, y^{(\ell-1)})$      \emph{a discrete $\cv$-valued rough path in $L_{p}$ controlled by $(x,\al)$ } if  $|  r^{(k)}_{s t} |_{L_{p}} \leq K (t-s)^{(\ell-k)\nu }$ for all $k=0,1,\dots, \ell-1$. The discrete path $(y^{(0)},\dots, y^{(\ell-1)})$ is \emph{controlled by $(x,\al)$ almost surely}   if $|r^{(k)}_{st}|\leq G_{y} (t-s)^{(\ell-k)\nu}  $, $k=0,\dots, \ell-1$ for some finite random variable $G_{y}$. 
\end{Def}
  
  \begin{remark}
  In some of our computations below we will rephrase \eref{e2.4} for $k=0$ as the following identity for $(s,t)\in \cs_{2}$:
  \begin{eqnarray}\label{e2.6j}
y_{t} &=& \sum_{i=0}^{\ell-1} y^{(i)}_{s} x^{i}_{st} +r^{(0)}_{st},
\end{eqnarray}
where we take $x^{0}\equiv 1$ by convention. 
  \end{remark}
  
  We first label a simple algebraic property relating the remainders $r^{(k)}$. 
  
  \begin{lemma}\label{lem2.3}
  Let $y=(y^{(0)},\dots, y^{(\ell-1)})$ be a discrete rough path  in $L_{p}$ controlled by $(x, \al)$ for all $p >1$. Then the following identity holds true for all $(s,u,t) \in \cs_{3}'$:
  \begin{eqnarray}\label{e3.5}
  \delta r^{(0)}_{sut}
  &=&
\sum_{i=1}^{\ell-1} r^{(i)}_{su} x^{i}_{ut} .
\end{eqnarray}
In particular, we have the following estimate for $p>1$:
  \begin{eqnarray}\label{e2.6i}
|\delta r^{(0)}_{sut}|_{L_{p}}&\leq& K (t-s)^{\nu \ell}.
\end{eqnarray}
\end{lemma}
\begin{proof}
By the definition of $r^{(0)}$ in \eref{e2.4}
and the expression \eref{e2.1} of $\delta g$ for $g\in \cc_{2}(\cv)$, some elementary computations yield:
\begin{eqnarray}\label{e2.6ii}
\delta r^{(0)}_{sut} =
 - \sum_{i=1}^{\ell-1}y^{(i)}_{s}  x^{i}_{st} 
+\sum_{i=1}^{\ell-1}y^{(i)}_{u}  x^{i}_{ut} 
+\sum_{i=1}^{\ell-1}y^{(i)}_{s}  x^{i}_{su} 
= 
\sum_{i=1}^{\ell-1} \delta y^{(i)}_{su} x^{i}_{ut} - \sum_{i=2}^{\ell-1} y^{(i)}_{s} \delta x^{i}_{sut}
,
\end{eqnarray}
where we have used the fact that $\delta x^{1}_{sut} =0$. 
 Therefore,   invoking   \eref{e2.4} and \eref{e2.3} again  we obtain
\begin{eqnarray*}
\delta r^{(0)}_{sut} &=& 
\sum_{i=1}^{\ell-1} r^{(i)}_{su} x^{i}_{ut} 
+
\sum_{i=1}^{\ell-2} \sum_{j=i+1}^{\ell-1} y^{(j)}_{s} x^{j-i}_{su}\otimes x^{i}_{ut} 
- 
\sum_{i=2}^{\ell-1} y^{(i)}_{s} \sum_{j=1}^{i-1} x^{i-j}_{su}\otimes x^{ j}_{ut}
\\
&=& \sum_{i=1}^{\ell-1} r^{(i)}_{su} x^{i}_{ut} .
\end{eqnarray*}
This concludes the identity \eref{e3.5}. The inequality \eref{e2.6i} follows by taking $L_{p}$-norm on both sides of \eref{e3.5} and taking into account the assumption that $|x^{i}_{st}|_{L_{p}}\leq K(t-s)^{\nu i}$ and  $|  r^{(i)}_{s t} |_{L_{p}} \leq K (t-s)^{(\ell-i)\nu }$.
  \end{proof}

An essential technical tool used in the sequel is the discrete sewing lemma. It is recalled below, the reader being referred to \cite{LT} for a proof. Let us begin with the definition of discrete $1$-increments.

\begin{Def}\label{def:discrete-increments}
Let $\pi$ be a partition on $[0,1]$. 
We denote by  $\mathcal{C}_{2} ( \pi   ,\mathcal{X})$  the collection of increments    $R$ defined  on $\cs_{2}'$ with values in a Banach space $(\mathcal{X}, |\cdot|)$ such that $R_{t_{k} t_{k+1}} = 0$ for $k=0,1,\dots, n-1$.   Similarly to the continuous case (relations \eref{e2.1} and \eqref{eq:def-holder-seminorms}), we define the operator $\delta$ and some H\"older semi-norms on $ \mathcal{C}_{2} ( \pi   ,\mathcal{X}) $   as follows:
\begin{equation*} 
\delta R_{sut} = R_{st}-R_{su}-R_{ut},
\quad\text{and}\quad
|R|_{\mu} = \sup_{ (u, v) \in \cs_{2} '}  \frac{|R_{uv}|}{|u-v|^{\mu}}\quad \text{for } \mu>0\,.
\end{equation*}
For $R\in \mathcal{C}_{2} ( \pi   ,\mathcal{X})$ and $\mu>0$, we also  set
\begin{equation}\label{e1}
|\delta R|_{\mu} = \sup_{(s,u,t) \in  \cs_{3}'} 
\frac{|\delta R_{sut}|}{|t-s|^{\mu}}\, .
\end{equation}
The space of functions $R\in \mathcal{C}_{2} ( \pi   ,\mathcal{X})$ such that $|\delta R|_{\mu}<\infty$ is denoted by $\mathcal{C}_{2}^{\mu} ( \pi   ,\mathcal{X})$.
\end{Def}

   The sewing lemma for elements of $\mathcal{C}_{2}^{\mu} (\pi, \cx)$ can be stated as follows. 
 
 \begin{lemma}\label{lem2.4}

 For a Banach space $\cx$, an exponent $\mu>1$ and $R \in \mathcal{C}_{2}^{\mu} ( \pi   ,\mathcal{X})$ as in Definition~\ref{def:discrete-increments}, the following relation holds true:
 \begin{equation*}
|R|_{\mu} \leq K_{\mu} |\delta R|_{\mu}\,,
\quad\text{where}\quad
K_{\mu} = 2^{\mu} \, \sum_{l=1}^{\infty} l^{-\mu}.
\end{equation*}
\end{lemma}

\subsection{Discrete rough integrals}
In this subsection, we   derive     upper-bound estimates for some   ``discrete'' integrals defined as Riemann type sums. 
 Namely, let $f$ and $g$ be functions     on   $\cs_{2}'$.
 For a generic partition $\cd_{n}=\{ 0=t_{0}<\cdots<t_{n}=1 \}$ of $[0,1]$, we set
 \begin{eqnarray}\label{e2.10}
\ep(t)=t_{k} \quad\quad \text{ for } \quad t\in (t_{k-1}, t_{k}].
\end{eqnarray}
  We define  the    discrete  integral of $f$ with respect to $g$ as:
  \begin{eqnarray}\label{e2.8}
\cj_{s}^{t} (f;g) 
&:=&
\sum_{s\leq t_{k}  <t }   
  f_{\ep(s) t_{k}} \otimes g_{t_{k}t_{k+1}} , \quad\quad  (s,t) \in \cs_{2}  .
\end{eqnarray}
 Similarly, if  $f  $ is a path on the grid $0=t_{0}<\cdots<t_{n}=1$, then we define the  discrete  integral of $f$ with respect to $g$ as:
\begin{eqnarray}\label{e2.12}
\cj_{s}^{t} (f;g) 
&:=&
\sum_{s\leq t_{k}  <t } 
 \delta f_{\ep(s) t_{k}} \otimes g_{t_{k}t_{k+1}}  , \quad\quad  (s,t) \in \cs_{2}.
\end{eqnarray}
 
 \begin{remark}\label{remark2.8}
 Notice that in \eref{e2.10}, $\ep (t)$ is the upper endpoint of the partition when $t\in (t_{k-1}, t_{k}]$. As a result, the first term of the Riemann sum  \eref{e2.12} is always vanishing. In addition, we also have $\cj_{t_{k}}^{t_{k+1}} (f; g) =0$ for all $(t_{k} , t_{k+1}) \in \cs_{2}'$.  
 \end{remark}
 The next proposition gives a basic estimate for discrete integrals. 
  In the following,  $\cv$ and $\cv'$ stand  for some finite dimensional vector spaces. 
 \begin{prop} \label{prop3.6}
Let $y= (y^{(0)},\dots, y^{(\ell-1)})$  be a discrete rough path on $[0,1]$, controlled by $(x,\al)$ in $L_{2}$, and let $h$ be a  $1$-increment defined on $\cs_{2}'$ with values in $\cv'$.   Suppose that $h$ satisfies
\begin{eqnarray}\label{e3.8i}
| \cj_{s}^{t} (x^{i}; h) |_{L_{2}} \leq K (t-s)^{\al +\nu i},  
\end{eqnarray}
for $i=0,1,\dots,\ell-1$ and $(s,t)\in\cs_{2}'$, where we recall that $\ell$ is an integer such that $\al+\nu\ell>1$.  
Then we have  the estimate
\begin{eqnarray}\label{eq3.36}
| \cj_{s}^{t} ( {r}^{(0)}; h) |_{L_{1}} &\leq & K   (t-s)^{ \nu \ell+\al},
\end{eqnarray}
which is valid for $(s,t)\in\cs_{2}' $.
\end{prop}
\begin{proof} 
In order to bound the increment $R_{st}:=\cj_{s}^{t} (r^{(0)}; h)$, we first note that $R_{t_{k}t_{k+1}} = 0$, due to the fact that $r^{(0)}_{t_{k}t_{k}}=0$. Let us now calculate $\delta R$: for $(s,u,t)\in\cs_{3}'$, it is readily checked that
\begin{eqnarray*}
\delta R_{sut} &=&   \cj_{s}^{t} ( {r}^{(0)}; h) -\cj_{s}^{u} (r^{(0)}; h)-\cj_{u}^{t} (r^{(0)}; h)  
\nonumber 
\\
&=& \sum_{u\leq t_{k}<t} ( r^{(0)}_{st_{k}} - r^{(0)}_{ut_{k}} ) h_{t_{k}t_{k+1}}
 .                                 
\end{eqnarray*}
Writing $r^{(0)}_{st_{k}} -r^{(0)}_{ut_{k}} = \delta r^{(0)}_{sut_{k}} +r^{(0)}_{su} $ and invoking relation \eref{e3.5}, we thus obtain 
\begin{eqnarray}\label{e3.6i}
\delta R_{sut}&=&  \sum_{i=0}^{\ell-1} r^{(i)}_{su} \cj_{u}^{t}(   x^{i}  ; h ).
\end{eqnarray}
 Now take the $L_{1}$-norm on both sides of \eref{e3.6i}, take into account condition  \eref{e3.8i} and the hypothesis $\nu \ell+\al>1$, and then apply Lemma \ref{lem2.4}. This easily yields  the desired estimate~\eref{eq3.36}. 
\end{proof}
  
 \section{Limit theorems}\label{sectionLT1}
 In this section, we first prove a general limit theorem for   discrete integrals.  Then we will handle two more specific situations which arise often in applications. 

\subsection{General limit theorem}

Recall that the discrete integral  $\cj_{s}^{t} (y; h)$ is defined in \eref{e2.8}. In this subsection, we prove a   general limit theorem for  $\cj_{s}^{t} (y; h)$. 

\begin{thm} \label{thm4.1i}
Let $\cv$ and $\cv'$ be two finite-dimensional vector spaces. 
Let $  (y^{(0)},\dots, y^{(\ell-1)})$  be a discrete $\cv$-valued rough path on $[0,1]$ controlled by $(x,\al)$  in $L_{2}$ or almost surely  (see Definition \ref{def2.2}), and 
$h^{n}$ be a  $1$-increment  which satisfies \eref{e3.8i} uniformly in $n$. 
Consider the family $\cj_{s}^{t} (x^{i}; h^{n}) $ defined by \eref{e2.8}, and suppose that 
\begin{eqnarray}\label{e4.1}
   \Big(x,\, \cj (x^{i};  h^{n}) \,,     \,\, i\in  \ci  \Big)
  \xrightarrow{ f.d.d. } ( x, \, {\omega}^{i}  \,,   \,\, i\in \ci  ),
\end{eqnarray} 
as $n\to \infty$, where $( {\omega}^{i},\, i\in \ci)$ is  a $1$-increment independent of $x$,
  $\ci:=\{ 0,1,\dots, \ell-1\}$,  and
   $\xrightarrow{f.d.d.}$ stands for convergence of finite dimensional distributions.  Suppose further that, if $\cj (y^{(i)}; \omega^{i})$ is given by \eref{e2.8}, we have   
\begin{eqnarray}\label{e3.2}
\Big(x, \cj (y^{(i)}, {\omega}^{i}) \,, \,\,  i\in  \ci \Big)  \xrightarrow{f.d.d.}   (x, \,v^{i}  \,, \,\,  i\in  \ci  ),
\end{eqnarray}
where $v^{i}$, $i\in \ci$ are $\cv\otimes \cv'$-valued $1$-increment. 
 Then the following convergence holds true as $  n\to \infty$:
 \begin{eqnarray}\label{e4.4}
  (x, \cj  (y ; h^{n})  )  & \xrightarrow{f.d.d.} & (x, v^{0}+ v^{1}+\cdots+v^{\ell-1}  ).
\end{eqnarray}
\end{thm}
\begin{remark}
If we particularize our limit theorem to the level $i=0$ of $\cj (x^{i}; h^{n})$, we just get that $h^{n} \xrightarrow{f.d.d.} \omega^{0}$ as part of the standing assumption. In return, we obtain that $v^{0} = \int_{0}^{1} y_{s} d \omega^{0}_{s}$ in relation \eref{e4.4}. 
\end{remark}
\begin{remark}
As the reader might have observed, Theorem \ref{thm4.1i} gives a general transfer principle from limit theorems for unweighted sums to limit theorems for weighted sums, within a rough paths framework. 
\end{remark}
\begin{remark}
Condition \eref{e4.1} is more demanding than condition \eref{e3.2} in Theorem \ref{thm4.1i}. Indeed, condition \eref{e3.2} is usually reduced to the convergence of a Riemann sum to an It\^o or Riemann type integral. 
\end{remark}
 \begin{proof}
 For sake of conciseness we will only show the convergence of $(x_{st}, \cj_{0}^{1}(y;h^{n}))$ for   some $(s,t)\in \cs_{2}$. The convergence of the finite dimensional distributions of $(x, \cj(y;h^{n}))$ can be shown in a similar way. 
  The proof is divided into several steps.
  
\noindent{\it Step 1: A decomposition of $\cj_{0}^{1} (y ; h^{n}) $.} 
 Take two uniform partitions  on $[0,1]$: $t_{k} =  {k}/{n}$   for $k, n \in \NN$ and $u_{j}= {j}/{m}$  for $j,m\in \NN$, and   $m\ll n$. Set:
\begin{eqnarray}\label{eq9.7}
D_{j}= \{t_{k}: u_{j+1}> t_{k}\geq u_{j}   \}   \quad \text{and} \quad \bar{u}_{j} = \ep(u_{j}), 
\end{eqnarray}
  where the function $\ep$ has been introduced in \eref{e2.10}.  
  By definition \eref{e2.12} we have 
  \begin{eqnarray*}
\cj_{0}^{1} (y; h^{n}) = \sum_{k=0}^{n-1} \delta y_{0t_{k}} \otimes h^{n}_{t_{k}t_{k+1}} = \sum_{k=0}^{n-1} y_{t_{k}} \otimes h^{n}_{t_{k}t_{k+1}}, 
\end{eqnarray*}
where the second identity is due to the fact that we have assumed $y_{0}=0$ in Definition \ref{def2.2}. Next we decompose the Riemann sum thanks to the sets $D_{j}$. We get
\begin{eqnarray*}
\cj_{0}^{1} (y; h^{n}) = \sum_{j=0}^{m-1} \sum_{t_{k}\in D_{j}} y_{t_{k}} h^{n}_{t_{k}t_{k+1}}. 
\end{eqnarray*}
Now we invoke relation \eref{e2.6j} with $s=\bar{u}_{j}$ and $t=t_{k}$ whenever $t_{k} \in D_{j}$. This yields: 
 \begin{eqnarray}\label{e3.5iii}
\cj_{0}^{1} (y ; h^{n})   &=&   \varphi^{0} +\cdots+\varphi^{\ell-1}+R^{\varphi} ,
\end{eqnarray}
where 
\begin{eqnarray}\label{e3.5ii}
 && \varphi^{i}  =   \sum_{j=0}^{m-1} \sum_{t_{k}\in D_{j}  } y^{(i)}_{\bar{u}_{j} }    x^{i}_{\bar{u}_{j} t_{k}}    \otimes h^{n}_{t_{k}t_{k+1}} =  \sum_{j=0}^{m-1}y^{(i)}_{\bar{u}_{j} }   \cj_{u_{j}}^{u_{j+1}} (x^{i}; h^{n}),
 \nonumber
 \\
  &&  R^{\varphi} =  \sum_{j=0}^{m-1} 
    \sum_{t_{k}\in D_{j}  }  
   r^{(0)}_{\bar{u}_{j} t_{k}} \otimes h^{n}_{t_{k}t_{k+1}} = \sum_{j=0}^{m-1}  \cj_{u_{j}}^{u_{j+1}} ( r^{(0)}; h^{n} ) \,,
\end{eqnarray}
and where we have set $x^{0}_{st} =1$ by convention.  
 Let us further decompose $\varphi^{i}$ as follows:
 \begin{eqnarray*}
\varphi^{i} &=&  \sum_{j=0}^{m-1} y^{(i)}_{ {u}_{j} } \cj_{u_{j}}^{u_{j+1}} ( x^{i}; h^{n} )
+
  \sum_{j=0}^{m-1} ( y^{(i)}_{ \bar{u}_{j} }  - y^{(i)}_{ {u}_{j} } ) \cj_{u_{j}}^{u_{j+1}} ( x^{i}; h^{n} )
 \\
 &:=&  {\varphi}_{1}^{i}+ {\varphi}_{2}^{i}.
\end{eqnarray*}
 We now study the convergence of $\varphi_{1}^{i}$ and $\varphi_{2}^{i}$ separately. 

\noindent{\it Step 2: Convergences of $  {\varphi}_{2}^{i}$.} 
In this step we show that 
for $i=1,\dots, \ell-1$, the random variable  
$   {\varphi}_{2}^{i} $ converges to zero   in probability as $n \to \infty$. To this aim, it suffices to consider the case when:
\begin{eqnarray}\label{e4.6i}
 (y^{(0)},\dots, y^{(\ell-1)}) \text{ is controlled by } (x,\al) \text{ in } L_{p}\,, \text{ for an arbitrary } p>1. 
\end{eqnarray} 
Indeed, for   $\ep>0$, we can find a constant $K$ such that $P(G_{y}>K)\leq \ep$ (see Definition \ref{def2.2} for the definition of $G_{y}$). Define $ (\bar{y}^{(0)}, \dots, \bar{y}^{(\ell-1)}) $ such that $\bar{y}^{i}= y^{i}$ on $\{G_{y}\leq K\}$ and $ \bar{y}^{i}\equiv 0 $ on $\{G_{y}>K\}$.  Then  $ (\bar{y}^{(0)}, \dots, \bar{y}^{(\ell-1)}) $ satisfies the condition \eref{e4.6i}, and we can write
\begin{eqnarray*}
P(|  {\varphi}_{2}^{i} |>\ep) &=&  P(|  {\varphi}_{2}^{i} |>\ep , G_y\leq K)  + P(|  {\varphi}_{2}^{i} |>\ep, G_y>K)
\\
&\leq &   P(| \tilde{\varphi}_{2}^{i} |>\ep  )  +\ep,
\end{eqnarray*}
where $\tilde{\varphi}_{2}^{i} = \varphi_{2}^{i}$ when $G_y\leq K$ and $ \tilde{\varphi}_{2}^{i} =  0 $ when $G_y> K$. 
So if we can show that $\tilde{\varphi}_{2}^{i}\to 0$ in probability, then the same convergence   holds for ${\varphi}_{2}^{i}$.

Assume now that \eref{e4.6i} is true. In this case   we have
\begin{eqnarray}\label{e3.5i}
|\varphi^{i}_{2}| &\leq& \sum_{j=0}^{m-1} | y^{(i)}_{ \bar{u}_{j} }  - y^{(i)}_{ {u}_{j} } | \cdot | \cj_{u_{j}}^{u_{j+1}} ( x^{i}; h^{n} )|.
\end{eqnarray}
Taking the $L_{1}$-norm on both sides of the inequality \eref{e3.5i}, invoking 
the fact that $h^{n}$ satisfies relation \eref{e3.8i} uniformly in $n$, and using the continuity of $y^{(i)}$ given by \eref{e4.6i},   we easily obtain the following convergence in probability:
\begin{eqnarray*}
\lim_{n\to \infty}  |\varphi^{i}_{2}|  &\to & 0.
\end{eqnarray*}

\noindent{\it Step 3: Convergences of $ {\varphi}_{1}^{i}$.} 
The convergences of $ \sum_{i=0}^{\ell-1} {\varphi}_{1}^{i}$ follows immediately from the assumptions of the theorem. Indeed, fixing $m$ and sending $n$ to $\infty$, our assumption \eref{e4.1} directly yields: 
\begin{eqnarray}\label{e3.6ii}
\Big(x_{st}, \sum_{i=0}^{\ell-1} \varphi_{1}^{i} \Big)\xrightarrow{ \ d \ } \Big(x_{st}, \sum_{i=0}^{\ell-1}  y^{i}_{ {u}_{j}}  {\omega}^{i}_{u_{j}u_{j+1}} \Big) . 
\end{eqnarray}
 We now send   $m\rightarrow \infty$ in \eref{e3.6ii}  and recall  the convergence \eref{e3.2}. This yields    the weak convergence:
\begin{eqnarray*}
 \Big(x_{st},  \,\sum_{i= 0}^{\ell-1}  {\varphi}_{1}^{i} \Big)
 & \xrightarrow{ \ d \ } & \Big(x_{st},\,  \sum_{i=0}^{\ell-1}v^{i} \Big) ,
\end{eqnarray*}
as $n\to \infty$ and $m\to \infty$.

\noindent{\it Step 4: Convergences of the remainder term $R^{\varphi}$.}  
Going back to equation \eref{e3.5iii} and summarizing our computations,   our claim \eref{e4.4} is now  reduced  to show that we have
\begin{eqnarray}\label{e9.4iii}
\lim_{m\rightarrow \infty} \limsup_{n\rightarrow \infty}  
    |
  {R}^{\varphi}    
  |   
  =0,
\end{eqnarray}
in probability. Moreover, as  in Step 2  it suffices to show the convergence \eref{e9.4iii} under   condition~\eref{e4.6i}. 
Eventually,   applying Proposition \ref{prop3.6} to \eref{e3.5ii} we obtain:
\begin{eqnarray}\label{e5.9}
| {R}^{\varphi} |_{L_{1}} \leq K \sum_{j=0}^{m-1} m^{-\nu\ell -\al} \leq Km^{1-\nu\ell -\al}.
\end{eqnarray}
The convergence \eref{e9.4iii} then follows   from   \eref{e5.9} and the fact that ${\nu\ell+\al>1}$.
 \end{proof}

 Let us state a more practical version of Theorem \ref{thm4.1i}, for which we take advantage of certain cancellations.

\begin{theorem}\label{thm4.1}
Let $y= (y^{(0)},\dots, y^{(\ell-1)})$ be a discrete rough path controlled by $(x, \al)$  in $L_{2}$ or almost surely, and assume that $h^{n}$ satisfies the inequality \eref{e3.8i} uniformly in $n$. Suppose   that  the following weak convergence holds true as $n\to \infty$:
\begin{eqnarray}\label{eq9.6}
   \Big(x,\, \cj (x^{i};  h^{n}) \,,     \,\, i\in  \ci' \Big)
  \xrightarrow{ f.d.d. } ( x, \, {\omega}^{i}  \,,   \,\, i\in \ci' ),
\end{eqnarray} 
  where $\cj (x^{i}; h^{n})$ is defined by \eref{e2.8}. 
Assume   that  for $i \in \ci''$ we have
\begin{eqnarray}\label{e5.4}
\lim_{m\rightarrow \infty} \limsup_{n\rightarrow \infty}  
   \Big| \sum_{j=0}^{m -1}  y^{(i)}_{u_{j}} \cj_{u_{j}}^{u_{j+1}} ( x^{i}; h^{n} ) \Big|  
  =0
\end{eqnarray}
in probability.
Here \,$\ci'$, $\ci''$    are disjoint subsets of \,$\ci:=\{ 0,1,\dots, \ell-1\}$ such that  \,$\ci'\cup \ci''  = \ci$.  Suppose further that
\begin{eqnarray}\label{e3.13}
\Big(x, \cj (y^{(i)}, {\omega}^{i}) \,, \,\,  i\in  \ci' \Big)  \xrightarrow{f.d.d.}  \Big(x,v^{i}  \,, \,\,  i\in  \ci' \Big),
\end{eqnarray}
where $v^{i}$, $i\in \ci'$ are random variables. 
 Then the following convergence holds true as $  n\to \infty$:
 \begin{eqnarray*} 
 \Big(x,   \cj (y ; h^{n}) \Big)    & \xrightarrow{f.d.d.} & \Big(x, \sum_{i\in \ci' }v^{i} \Big)  .
\end{eqnarray*}
 \end{theorem}
\begin{proof} 
The proof is similar to Theorem \ref{thm4.1i} and is omitted. 
\end{proof}

Theorem \ref{thm4.1} allows us to distinguish two predominant cases in limit   Theorem \ref{thm4.1i}: (i) A usual asymptotic regime, for which only one level $v^{i}$ remains. (ii) A critical case, for which more than one level survive as $n$ goes to $\infty$. 

The following definition captures those different behaviors. 
 \begin{Def}\label{def3.5}
 We will call a limit theorem \emph{single} if $\ci'$ in Theorem \ref{thm4.1} has only one element. Similarly, a limit theorem is called  \emph{double} when $\ci'$ has   two elements.
 \end{Def}

\begin{remark}   If   the convergences in \eref{eq9.6} and \eref{e3.13} hold  true in probability, then in a similar way  one can show that   $\cj_{s}^{t} (y^{(0)}; h^{n})   $ converges to $ \sum_{i\in \ci'}v^{i}_{st}$ in probability.  
\end{remark}
\begin{remark}\label{remark4.12} 
In the case $\nu>\frac12$ and $\al\geq\frac12$, we have $\ell=1$ and $\ci = \{0\}$. Therefore,   conditions   \eref{e3.8i}, \eref{eq9.6}, \eref{e5.4}   are reduced to 
$|h^{n}|_{L_{2}} \leq K(t-s)^{\al}$ and
 $ (h^{n}, x) \xrightarrow{f.d.d.} (\omega,x)  $.  If $h^{n} \to \omega$ in $L_{p}$ for all $p\geq 1$, then $\cj_{0}^{1}(y; h^{n})$ also converges in $L_{q}$ for $q\geq 1$. This situation allows to recover the results in \cite{CNP} and \cite[Proposition 7.1]{HLN1}. A more specific statement will be given in Proposition~\ref{prop8.6} below. 
\end{remark}

 \subsection{Single  limit theorem I}\label{section3.2}

 An important case  in Theorem \ref{thm4.1i} is   when $h^{n}  $ converges  in distribution to a Brownian motion and  the discrete integral $\cj_{0}^{1} (y^{(0)}; h^{n})$ converges to the  Wiener integral $\int_{0}^{1} y_{t}\otimes dW_{t}$. In this subsection we investigate this type of limit theorems.

 \begin{theorem}\label{thm5.9}
 Let $x$ be a $(L_{p}, \nu, \ell)$-rough path for $p=4$, $\nu\in (0,1)$ and $\ell$ such that $\nu \ell +\frac12>1$. 
 Let   $y= (y^{(0)},\dots, y^{(\ell-1)})$  be a process on $[0,1]$ controlled by $(x,\frac12)$ in $L_{2}$ or  almost surely, and assume that  $h^{n}$ satisfies the inequality \eref{e3.8i} uniformly in $n$.
 Suppose that the following assumptions are fulfilled:
 
\noindent (i) We have the convergence    $ (x ,    h^{n}        )
  \xrightarrow{f.d.d.} ( x ,  W    )$ as $n\to \infty$, 
where $W$ is a standard Brownian motion independent of $x$.

\noindent (ii) For any partition $ 0\leq s_{0} <s_{1}<\cdots<s_{m }\leq 1 $   of $[0,1]$ such that $m\ll n$ and 
$ |s_{j+1} - s_{j}| \leq \frac{1}{m}$, 
 we have 
\begin{eqnarray}\label{e6.1}
\lim_{m\to \infty} \limsup_{n\to \infty} \Big| \sum_{j=0}^{m-1}     \cj_{s_{j}}^{s_{j+1}} ( x^{i}; h^{n} ) \Big| &=&0
\end{eqnarray}
in probability for $i=1,\dots, \ell-1$. 

Then we have the following convergence in distribution for the process $y$:
\begin{eqnarray}\label{e5.2}
(x, \cj  (y ; h^{n}) ) &\xrightarrow{f.d.d.}& (x, v), 
\end{eqnarray}
where the integral $v_{st}:=\int_{s}^{t} y_{u} \otimes dW_{u}$ has to be understood as a conditional Wiener integral.
 \end{theorem}
 
\begin{remark}\label{remark5.10}
 Theorem \ref{thm5.9} can be generalized to   some other interesting situations. For example, 
suppose that $ (x,h^{n}) \xrightarrow{f.d.d.} (x, \omega) $, where $\omega$ is a continuous Gaussian  process independent of $x$. Let  $\ch$ be the Hilbert space corresponding to $\omega$ and assume that $ C^{\ga}\subset \ch$ for $\ga<\nu$. We assume that for any $f\in C^{\ga}$,    we have the following convergences for a generic partition $0\leq s_{0}<\cdots<s_{m}\leq 1$:
\begin{eqnarray}\label{e3.29}
\lim_{m\to \infty}
\sum_{j,j'=0}^{m}\langle \delta f_{s_{j},\cdot} \, ,   \delta f_{s_{j'},\cdot} \rangle_{\ch} = 0
\quad \text{and}\quad 
\lim_{m\to \infty}
\sum_{j=0}^{m-1} f_{t_{j}} \mathbf{1}_{[t_{j},t_{j+1})} = f
\end{eqnarray}
where the second limit is a limit in $\ch$.  
Then   following the lines of the proof of Theorem \ref{thm5.9} one can show that 
\begin{eqnarray*}
\cj_{0}^{1} (y; h^{n}) \xrightarrow{ \ d \ }  \int_{0}^{1} y \otimes d\omega . 
\end{eqnarray*}

\end{remark}
 \begin{proof}[Proof of Theorem \ref{thm5.9}] 
 Take $\ci' = \{0\}$ and $\ci'' = \{1,\dots, \ell\}$. 
The theorem will be proved by applying Theorem \ref{thm4.1} and verifying the convergences \eref{eq9.6}, \eref{e5.4} and \eref{e3.13}.  The proof is divided into several steps.
 
 \noindent{\it Step 1: } 
We will show by induction that   
\begin{eqnarray}\label{e5.3i}
  \cj_{s}^{t}  ( x^{i}; h^{n}  ) \xrightarrow{ \ d \ }  \omega^{i} \equiv \int_{s}^{t} x^{i}_{su} \otimes dW_{u}, \quad i=0,1,\dots, \ell-1.
\end{eqnarray}
 
Since $h^{n}\to W$ in $f.d.d.$ sense,    convergence \eref{e5.3i} holds true when $i=0$. Now assume that the convergence holds for   $i=0,1,\dots, \tau-1$ with $\tau<\ell$. 
 Take $m\ll n$ and $u_{j} = j/m$, and set $D_{j}= \{t_{k}: u_{j+1}> t_{k}\geq u_{j}   \} $ as in \eref{eq9.7}.
 Take $j_{1}$ such that $s\in D_{j_{1}}$ and $j_{2}$ such that $t\in D_{j_{2}}$. 
  Then a small variant of \eref{e2.3} shows that for all $t_{k}\in D_{j}$, 
  \begin{eqnarray*}
x^{\tau}_{\ep(s) t_{k}}
= \delta x^{\tau}_{\ep(s), \bar{u}_{j} \vee \ep(s), t_{k}} +x^{\tau}_{\ep(s), \bar{u}_{j}\vee \ep (s)} + x^{\tau}_{\bar{u}_{j} \vee \ep (s), t_{k}} 
 = \sum_{l=0}^{\tau} x^{\tau-l}_{\ep(s) , \bar{u}_{j} \vee \ep(s)  }\otimes x^{l}_{ \bar{u}_{j}   \vee \ep(s) , t_{k} },
\end{eqnarray*}
where recall that the function $\ep$ is defined in \eref{e2.10} and $\bar{u}_{j}$ is given by \eref{eq9.7}. 
Hence it is readily checked that: 
\begin{eqnarray}\label{e5.3}
\cj_{s}^{t} ( x^{\tau}; h^{n} ) = \sum_{l=0}^{\tau}  \sum_{j=j_{1}}^{j_{2}} 
 x^{\tau-l}_{\ep(s) , \bar{u}_{j}  \vee \ep(s)   }\otimes
 \cj_{u_{j}\vee s}^{u_{j+1} \wedge t } (x^{l}; h^{n}  )
  :=  \sum_{l=0}^{\tau} A_{l}.
  \end{eqnarray}
Let us change the name of our variables in order to match the notation of our theorem and use relation \eref{e6.1}. Namely, 
set $s_{0}=s$, $s_{1} = u_{j_{1}+1}$,\dots, $s_{j_{2}-j_{1}} = u_{j_{2}}$, $s_{j_{2}-j_{1}+1}=t$. Then 
it is readily checked that $A_{\tau} = \sum_{j=0}^{j_{2}-j_{1}} \cj_{s_{j}}^{s_{j+1}} (x^{\tau}; h^{n}) $. Thus
 invoking      assumption \eref{e6.1}, we directly have the following convergence in probability:
\begin{eqnarray}\label{e3.15i}
\lim_{m\to \infty} \limsup_{n\to \infty} |A_{\tau}|   &=&0.
\end{eqnarray}
In order to study the convergence of $A_{l} $ for $l<\tau$, we first check that we can replace $ x^{\tau-l}_{\ep(s) , \bar{u}_{j}  \vee \ep(s)   } $ by $ x^{\tau-l}_{s ,  {u}_{j}  \vee s   }$. Indeed, we have the identity:
\begin{eqnarray*}
x^{\tau-l}_{\ep(s), \bar{u}_{j} \vee \ep(s)} - x^{\tau-l}_{s, u_{j}\vee s} = \delta x^{\tau-l}_{s,  u_{j}\vee s, \bar{u}_{j}\vee \ep(s) } - \delta x^{\tau-l}_{s, \ep(s), \bar{u}_{j}\vee \ep(s)} + x^{\tau-l}_{u_{j}\vee s, \bar{u}_{j} \vee \ep(s) } - x^{\tau-l}_{s, \ep(s)}. 
\end{eqnarray*}
 Therefore, if we set:
\begin{eqnarray*}
\tilde{A}_{l}  :=  
\sum_{j=j_{1}}^{j_{2}} 
 x^{\tau-l}_{s ,  {u}_{j}  \vee s   }\otimes
 \cj_{u_{j}\vee s}^{u_{j+1} \wedge t } (x^{l}; h^{n}  )
= \sum_{j=0}^{j_{2}-j_{1}} 
 x^{\tau-l}_{s    s_{j }    }\otimes
 \cj_{s_{j}}^{s_{j+1}  } (x^{l}; h^{n}  ),
\end{eqnarray*}
then it is readily checked that: 
\begin{eqnarray}\label{e3.15}
\lim_{m\to \infty  } \limsup_{n\to \infty} | A_{l} - \tilde{A}_{l} |  =0 \quad \quad \text{in probability}.
\end{eqnarray}
Let us now check the convergence for $\tilde{A}_{l}$. 
 Sending $n \to\infty$ and applying the induction assumption
 \eref{e5.3i} 
   with $l<\tau$,  we get
\begin{eqnarray}\label{e5.4ii}
\tilde{A}_{l} &\xrightarrow{ \ d \ }&    \sum_{j=0}^{j_{2}-j_{1}} 
 x^{\tau-l}_{s   s_{j}     }\otimes 
 \int_{s_{j}}^{s_{j+1}  } x^{l}_{s_{j}    u } \otimes dW_{u}. 
\end{eqnarray}
We now separate the analysis of $\tilde{A}_{l}$ in two cases. 

(a) For $0<l<\tau$ the square of the $L_{2}$-norm of  the right-hand side of \eref{e5.4ii} can be bounded thanks to It\^o's isometry by:  
\begin{eqnarray}\label{e5.6}
K\mE 
\Big[
\sum_{j=0}^{j_{2}-j_{1}} 
 | x^{\tau-l}_{s   s_{j}     }|^{2} 
 \int_{s_{j} }^{s_{j+1}  } | x^{l}_{s_{j}     u } |^{2} d {u}
 \Big]
 ,
\end{eqnarray}
which by property \eref{eq:def-holder-seminorms} applied to $p=4$ and $l\geq 1$ is less than
\begin{eqnarray*}
K \sum_{j=0}^{m-1} (s_{j+1}-s_{j})^{2\nu+1}.
\end{eqnarray*}
Owing to the fact that $2\nu  +1>1$, it is now trivially seen that as   $m\to \infty$, the right-hand side of \eref{e5.4ii} converges to zero.  Thus we get: 
\begin{eqnarray}\label{e3.17}
\lim_{m\to \infty} \limsup_{n\to \infty} |\tilde{A}_{l}|=0
\end{eqnarray}
in probability. 
 In summary of \eref{e3.15} and \eref{e3.17}, we have the convergence 
 \begin{eqnarray}\label{e3.18i}
\lim_{m\to \infty} \limsup_{n\to \infty} |A_{l}|  &=&0, 
\end{eqnarray}
in probability for $0<l<\tau$. 

(b) When $l=0$, the convergence \eref{e5.4ii} implies that,  as $n\to \infty$ and $m\to\infty$  we obtain
\begin{eqnarray}\label{e3.18}
 \tilde{A}_{0}  \xrightarrow{ \ d \ }   \int_{s}^{t} x^{\tau}_{sr} \otimes dW_{r} . 
\end{eqnarray}
Taking into account \eref{e3.15}, the convergence \eref{e3.18} implies that 
\begin{eqnarray}\label{e3.22}
 {A}_{0}  \xrightarrow{ \ d \ }  \int_{s}^{t} x^{\tau}_{sr} \otimes dW_{r}  . 
\end{eqnarray}
Putting together (a), (b) and the case $l=\tau$, we can now propagate our induction hypothesis. Indeed, 
applying \eref{e3.15i}, \eref{e3.18i} and \eref{e3.22} to  \eref{e5.3}, we obtain
\begin{eqnarray*}
  \cj_{s}^{t} ( x^{\tau}; h^{n} ) \xrightarrow{ \ d \ }   \int_{s}^{t} x^{\tau}_{sr}\otimes  dW_{r}. 
\end{eqnarray*}
This completes the proof of \eref{e5.3i} for $i=0,\dots, \ell-1$. In a similar way, we can get a f.d.d. version of \eref{e5.3i}. Namely,  we can show that
\begin{eqnarray}\label{e3.23}
(x, \cj  ( x^{i}; h^{n} ) ) \xrightarrow{ f.d.d. } (x,  {\omega}^{i} )
\end{eqnarray}
for $ {\omega}^{i}_{st} =  \int_{s}^{t} x^{i}_{sr}\otimes   dW_{r}$, $i\in \ci$. Note that this shows that   condition \eref{eq9.6} in Theorem \ref{thm4.1} holds true. 

 \noindent{\it Step 2: }  In this step, we consider the convergence of  $\cj_{0}^{1} (y^{(i)};  {\omega}^{i})$, which will yield condition~\eref{e5.4} in Theorem \ref{thm4.1}. 
 
 We first show that the discrete integral $\cj_{0}^{1} (y^{(i)};  {\omega}^{i})$, $\ell>i>0$ converges to zero in probability. As in the proof of Theorem \ref{thm4.1i}, by a truncation argument,  it suffices to show the convergence when $(y^{(0)},\dots, y^{(\ell-1)})$ is controlled by $(x,\frac12)$ in $L_{p}$ for $p$ large enough. In this case, similarly to \eref{e5.6},   we have
\begin{eqnarray}\label{e3.31}
|\cj_{0}^{1} (y^{(i)};  {\omega}^{i})|_{L_{2}}^{2} &\leq& K
\mE \Big[
 \sum_{j=0}^{m-1} | y^{(i)}_{u_{j}} |^{2} \int_{u_{j}}^{u_{j+1}} |x^{i}_{u_{j}u }   |^{2} du  \Big]
 \nonumber
\\
&\leq& K\sum_{j=0}^{m-1} m^{-2\nu-1}.
\end{eqnarray}
Therefore, we have $\cj_{0}^{1} (y^{(i)};  {\omega}^{i}) \to 0$ in probability. Combining this convergence with \eref{e3.23} for $i=1,\dots, \ell-1$, we obtain the convergence   \eref{e5.4}.

On the other hand, for the quantity $\cj_{0}^{1} (y^{(i)}; \omega^{i})$ with $i=0$,  thanks to the convergences of Riemann sums related to Wiener integrals the following convergence holds in $L^{2}$: 
\begin{eqnarray*}
\cj_{0}^{1} (y ;  {\omega}^{0}) = \cj_{0}^{1} (y ; W) \to \int_{0}^{1} y_{t} \otimes dW_{t}.
\end{eqnarray*}
  So  the condtion \eref{e3.13} holds true with $v^{0} = \int_{0}^{1} y_{t} \otimes dW_{t}$. 

Summarizing our consideration, we can now apply Theorem \ref{thm4.1}  to $\cj_{0}^{1} (y; h^{n})$ and  we obtain the convergence \eref{e5.2}.  
 \end{proof}

 \subsection{Single limit theorem II}
 In Section \ref{section3.2} we have investigated possible limit theorems under the assumption $(x, h^{n})\to (x, W)$, which implies in particular $\cj_{s}^{t} (x^{0}; h^{n}) \to \delta W_{st}$. 
 In the current section we analyze situations for which the convergence of $\cj_{s}^{t}(x^{i}; h^{n})$ is assumed for a more general $i$. Our results are summarized in the following theorem. 
\begin{thm}  \label{cor6.3}
Let   $y= (y^{(0)},\dots, y^{(\ell-1)})$  be a $\cv$-valued rough path on $[0,1]$ controlled by $(x,\al)$  in $L_{2}$ or almost surely, and consider $h^{n}$ satisfying the inequality \eref{e3.8i}. Recall that the increment  $\cj( y ; h^{n}) = \{ \cj_{s}^{t}( y ; h^{n});\,\, (s,t) \in \cs_{2} \}$ is defined by \eref{e2.8}. Suppose that $x$ and $h^{n}$ verify the following assumptions:

\noindent
(i) There is   some $\tau \in \ci$ such that $   \cj_{s}^{t} (x^{\tau}; h^{n}) \rightarrow     (t-s) \varrho $ in probability  for all $(s,t) \in \cs_{2}$, where $\varrho \in (\RR^{d})^{\otimes \tau} \otimes \cv' $ is a constant matrix, and  
$
  \cj_{s}^{t}(x^{i};  h^{n}) 
   \to  0 
$
 in probability for all $i<\tau $ and  $(s,t) \in \cs_{2}$. 
 
\noindent (ii) For any partition $ 0\leq s_{0} <s_{1}<\cdots<s_{m }\leq 1 $   on $[0,1]$ such that $m\ll n$ and 
$   |s_{i+1} - s_{i}| \leq \frac{1}{m}$, 
 we have
\begin{eqnarray}\label{e5.4i}
\lim_{m\to \infty} \limsup_{n\to \infty} \Big| \sum_{j=0}^{m-1}   \cj_{s_{j}}^{s_{j+1}} ( x^{i}; h^{n} ) \Big|  &=&0,
\end{eqnarray}
in probability
for $i=\tau+1,\dots, \ell-1$. 

 Then   the following convergence holds true for $y$:
 \begin{eqnarray}\label{e5.8ii}
\lim_{n\to \infty}   \cj (y ; h^{n}) \xrightarrow{ \   \ }     \Big( \int  y^{(\tau)}_{t}d {t}\Big)\otimes \varrho
\end{eqnarray}
in probability.
 \end{thm}
 
  \begin{proof} 
  As for Theorem \ref{thm5.9}, we will prove our claim thanks to Theorem \ref{thm4.1}, and we are reduced to check \eref{eq9.6}, \eref{e5.4}, and \eref{e3.13}. The difference with Theorem \ref{thm5.9} is that we now consider   $\ci'=\{\tau\}$ and $\ci'' = \ci\setminus \{\tau\}$. We divide the proof in several steps.  
  
  \noindent \textit{Step 1: Case $i<\tau$.} In this first situation it is immediate from our assumptions that \eref{e5.4} holds true for $i<\tau$. 
  
  \noindent \textit{Step 2: Case $i\geq \tau$.} Similarly to the proof of Theorem \ref{thm5.9} (Step 1), we prove by induction the following convergence in probability for all $\ell>i\geq\tau$:
    \begin{eqnarray}\label{e5.8i}
    \cj  ( x^{i}; h^{n} ) \xrightarrow{ \ \ } {\omega}^{i} \quad
    \text{where }  {\omega}^{i}_{st} =  \Big( \int_{s}^{t} x^{i-\tau}_{su} du \Big) \otimes \varrho.
\end{eqnarray}
  To this aim, notice that \eref{e5.8i} is true for $i=\tau$ by assumption.  Next assume that \eref{e5.8i} holds for $i=\tau,\dots, \tau' -1 $.  
  We decompose  the discrete interval $\lll s ,t\rr$ into the subintervals $D_{j} $ again (see \eref{eq9.7}), with $m\ll n$ and $t_{k} = \frac{k}{n}$, $u_{j} = \frac{j}{m}$.  Let $s_{0},\dots,s_{j_{2}-j_{1}+1}$ be as in Theorem~\ref{thm5.9}~(Step~1). 
  Then an approximation procedure similar to \eref{e3.15} allows to replace each $  x^{\tau'-l}_{\ep(s)t_{k}} $ by an expression of the form:
  \begin{eqnarray*}
\sum_{l=0}^{\tau'} x^{\tau-l}_{ss_{j}} \otimes x^{l}_{s_{j}t_{k}}, 
\end{eqnarray*}
 in the sum defining $\cj_{s}^{t} ( x^{\tau'}; h^{n} ) $. 
  Therefore, we get an equivalent of \eref{e3.15} in our context:
  \begin{eqnarray}\label{e-approx.cj}
\lim_{m\to \infty} \limsup_{n\to\infty} | \cj_{s}^{t} ( x^{\tau'}; h^{n} )-  \sum_{l=0}^{\tau'} \tilde{A_{l}} | =0
\end{eqnarray}
with 
\begin{eqnarray*}
\tilde{A}_{l} &=& \sum_{j= 0}^{j_{2}-j_{1}} x^{\tau'-l}_{ss_{j} } \otimes \cj_{s_{j}}^{s_{j+1}} (x^{l}; h^{n}). 
\end{eqnarray*}

We now handle each $\tilde{A}_{l}$. 
For $l<\tau$, each $
 \cj_{s_{j} }^{s_{j+1}   } (x^{l}; h^{n}  )$ converges to $0$ 
in probability as $n\to \infty$ for all $j$, according to our assumption (i). Hence  
$
\tilde{A}_{l} \to 0 
$
in probability as $n\to \infty$ and $m\to \infty$. 
When $ {\tau}<l<\tau'$,
we proceed along the same lines as for  \eref{e5.4ii} and \eref{e5.6}. Namely, we invoke the fact that $ \lim_{n\to \infty} \cj_{s_{j}}^{s_{j+1}} (x^{l}; h^{n}) = \Big( \int_{s_{j}}^{s_{j+1}} x^{l-\tau}_{s_{j}u} du \Big) \otimes \varrho $ for each $j_{1}\leq j\leq j_{2}$ and then use the extra regularity given by $x^{l}_{s_{j}u}$ on each $[s_{j}, s_{j+1}]$. This yields the following limit in probability: 
 \begin{eqnarray}\label{e-Al}
  \lim_{m\to\infty}\limsup_{n\to \infty} 
| \tilde{A}_{l}| \to 0 .
\end{eqnarray}
Let now $l= {\tau}' $. 
Then 
\begin{eqnarray*}
\tilde{A}_{\tau'} = \sum_{j=0}^{j_{2}-j_{1}} \cj_{s_{j}}^{s_{j+1}} (x^{\tau'}; h^{n}),
\end{eqnarray*}
and it is immediate from  identity \eref{e5.4i} that \eref{e-Al}  holds true for $l=\tau'$.
In summary, we have proved that for all $l\in \{0,\dots, \tau'\}\setminus \{\tau\}$ we have:
\begin{eqnarray}\label{e3.32}
 \lim_{m\to \infty} \limsup_{n\to \infty} |\tilde{A}_{l}|  = 0
\end{eqnarray}
 in probability.

 In the case $l=\tau$, by sending $n\to \infty$, our assumption (i)   allows to write:
\begin{eqnarray}\label{e3.33}
 \tilde{A}_{\tau} \xrightarrow{ \   \ }  \sum_{j=0}^{j_{2} - j_{1 } } x^{ {\tau}' - \tau }_{ss_{j}} ( s_{j+1} - s_{j} )\otimes \varrho,
\end{eqnarray}
in probability,
and thus as $n\to \infty$ and $m\to \infty$, we obtain
\begin{eqnarray}\label{e3.38}
 \tilde{A}_{\tau } \xrightarrow{ \   \ } \Big(\int_{s}^{t} x^{ {\tau}' - \tau }_{s u } du \Big)\otimes \varrho
\end{eqnarray}
in probability.
 Combining \eref{e3.32} and \eref{e3.38} and taking into account   relation \eref{e-approx.cj}, we end up with:  
\begin{eqnarray*}
\cj_{s}^{t} ( x^{ {\tau}'}; h^{n} ) \xrightarrow{ \   \ } \Big(\int_{s}^{t} x^{ {\tau}' - \tau }_{s u } du \Big) \otimes \varrho
\end{eqnarray*}
in probability.
This completes our induction and the proof of \eref{e5.8i}. 
    
    \noindent \textit{Step 3: Convergence of $\cj_{0}^{1} (y^{(i)}; \omega^{i})$}  In a similar way as in the proof of Theorem \ref{thm5.9} (see relation \eref{e3.31}), we can show the convergence
    \begin{eqnarray*}
\cj_{0}^{1} (y^{(i)};  {\omega}^{i} ) \xrightarrow{ \   \ } 0, 
\end{eqnarray*}
in probability
for $i \neq \tau$, so the convergence \eref{e5.4} holds true. On the other hand, it is clear from classical integration theory that: 
  \begin{eqnarray*}
\cj_{0}^{1} (y^{(\tau)};  {\omega}^{\tau} ) \to \Big(\int_{0}^{1} y^{(\tau)}_{ u} du  \Big) \otimes \varrho,
\end{eqnarray*}
which implies the convergence \eref{e3.13}.  
 Summarizing, we have proved \eref{eq9.6}-\eref{e3.13} and   the convergence \eref{e5.8ii}   follows immediately from Theorem \ref{thm4.1}. 
  \end{proof}
  \begin{remark}\label{remark3.14} 
  As in Remark \ref{remark5.10}, one can generalize  Theorem \ref{cor6.3} to some other interesting cases. 
  For example, 
suppose that $ (x,\cj (x^{\tau}; h^{n})) \xrightarrow{f.d.d.} (x, \omega) $, where $\omega$ is a continuous Gaussian  process independent of $x$ and with values in $ (\RR^{d})^{\otimes \tau} \otimes \cv' $. As before, let  $\ch$ be the Hilbert space corresponding to $\omega$ and suppose that $ C^{\ga}\subset \ch$ for all $\ga<\nu$. Suppose that \eref{e3.29}  holds true for any $f\in C^{\ga}$. 
 Then    one can show in a similar way as in Theorem \ref{cor6.3} that 
\begin{eqnarray*}
\cj_{0}^{1} (y; h^{n}) \xrightarrow{ \ d \ }  \int_{0}^{1} y^{(\tau)} d\omega . 
\end{eqnarray*}

  \end{remark}
  
 \subsection{Double limit theorem} 
 In this subsection, we consider the double limit theorem case, which has been introduced in Definition \ref{def3.5}. 
 This usually corresponds to a transition in terms of roughness for the underlying noise $x$. 
  \begin{theorem}\label{thm7.1}
  Let   $y= (y^{(0)},\dots, y^{(\ell-1)})$  be a rough path on $[0,1]$ controlled by $(x,\al)$  in $L_{2}$ or almost surely, and suppose that $h^{n}$ satisfies the inequality \eref{e3.8i} uniformly in $n$. 
  Furthermore, we assume that $x$ and $h^{n}$ fulfill the following conditions:
  
 \noindent (i) 
We have the convergence    $ (x ,    h^{n}        )
  \xrightarrow{f.d.d.} ( x ,  W    )$ as $n\to \infty$, 
where $W$ is a standard Brownian motion independent of $x$.

\noindent (ii)
There exists a constant matrix $\varrho \in (\RR^{d})^{\otimes \tau} \otimes \cv'$ and some $  \tau : 0<\tau<\ell$ such that 
 for any partition $ 0\leq s_{0} <s_{1}<\cdots<s_{m }\leq 1 $   on $[0,1]$ such that $m\ll n$, 
$  |s_{i+1} - s_{i}| \leq \frac{1}{m}$, and $s_{0}= s$, $s_{m}= t$, 
 we have
\begin{eqnarray}\label{e6.1i}
\lim_{m\to \infty} \lim_{n\to \infty}   \sum_{j=0}^{m-1}     \cj_{s_{j}}^{s_{j+1}} ( x^{\tau}; h^{n} )  =   (t-s)\varrho,
\end{eqnarray}
where the limit has to be understood as a limit  in probability  for all $(s,t) \in \cs_{2}$ and where $\cj_{s_{j}}^{s_{j+1}} ( x^{i}; h^{n} ) $ is defined by \eref{e2.8}.

\noindent (iii)
  For any $i\in \{ 1,\dots, \ell-1 \}\setminus \{\tau\}$ and any partition $0\leq s_{0}<\cdots<s_{m}\leq 1$ in (ii) above, we have the following convergence in probability:
\begin{eqnarray}\label{e3.40}
\lim_{m\to \infty} \limsup_{n\to \infty} \Big| \sum_{j=0}^{m-1}    \cj_{s_{j}}^{s_{j+1}} ( x^{i}; h^{n} )  \Big|  &=&0.
\end{eqnarray}

 Then the following limit   holds true for $\cj (y; h^{n})$:
 \begin{eqnarray*}
  (x, \cj  (y ; h^{n}) ) \xrightarrow{f.d.d.} (x,  \int  y_{t } \otimes  dW_{t} +  \Big(\int  \delta y^{(\tau)}_{0t}d {t}\Big) \otimes \varrho).
\end{eqnarray*}
  \end{theorem}

  \begin{proof} 
  As in Theorem \ref{thm5.9} and Theorem \ref{cor6.3}, we apply Theorem \ref{thm4.1} and we are reduced to show relations \eref{eq9.6}, \eref{e5.4}, and \eref{e3.13}. In the current situation, we consider   $\ci'=\{0,\tau\}$ and $\ci'' = \ci\setminus \ci'$. We divide again the proof in several steps. 
 
 \noindent \textit{Step 1: Case $i<\tau$.}    As in the proof of Theorem \ref{thm5.9},   by induction we can show that for $i<\tau$ we have the convergence
  \begin{eqnarray}\label{e6.2i}
\cj_{s}^{t} ( x^{i}; h^{n} ) \xrightarrow{ \ d  \ } 
\omega^{i}_{st} \equiv\int_{s}^{t} x^{i}_{su}\otimes dW_{u}. 
\end{eqnarray}

 \noindent \textit{Step 2: Case $i=\tau$.}
An approximation argument similar to \eref{e3.15}  and \eref{e-approx.cj}   yields: 
   \begin{eqnarray} \label{e-approx.cj2}
\lim_{m\to \infty} \limsup_{n\to\infty} | \cj_{s}^{t} ( x^{i}; h^{n} )-  \sum_{l=0}^{i} \tilde{A}_{l,i} | =0
\end{eqnarray}
with 
\begin{eqnarray*}
\tilde{A}_{l,i} &=& \sum_{j= 0}^{j_{2}-j_{1}} x^{i-l}_{ss_{j} } \otimes \cj_{s_{j}}^{s_{j+1}} (x^{l}; h^{n}). 
\end{eqnarray*}

In the same way as in \eref{e3.17}, and taking into account the convergence \eref{e6.2i},     for   $0<l<\tau$ we have  the convergence  
\begin{eqnarray*}
\lim_{m\to \infty} \lim_{n\to \infty} |\tilde{A}_{l, \tau}| =0
\end{eqnarray*}
in probability. On the other hand, in the same way as for relation  \eref{e3.18}, the following limit holds true  for $l=0$: 
\begin{eqnarray*}
\lim_{m\to \infty}\lim_{n\to \infty}
\tilde{A}_{0, \tau} \stackrel{ \ (d) \ }{=} \int_{s}^{t} x^{\tau}_{su}\otimes dW_{u}. 
\end{eqnarray*}
  In addition, owing to   assumption \eref{e6.1i} we have
\begin{eqnarray*}
\lim_{m\to \infty} \lim_{n\to \infty} \tilde{A}_{\tau, \tau} = 
 (t-s) \varrho,
\end{eqnarray*}
where the limit holds  in probability. In summary of the convergences of $\tilde{A}_{l, \tau}$, $l=0,\dots, \tau$ and taking into account \eref{e-approx.cj2}, we obtain
\begin{eqnarray*}
\cj_{s}^{t} ( x^{\tau}; h^{n} )  \xrightarrow{ \ d \ } \int_{s}^{t} x^{\tau}_{su} \otimes dW_{u}
+  (t-s)\varrho .
\end{eqnarray*}
Notice that we can add up limits in distribution here, since one of the limits is deterministic. 

\noindent \textit{Step 3: Case $i>\tau$.}
   In the following,  we show by induction   the convergence 
\begin{eqnarray}\label{e6.4}
\cj_{s}^{t} ( x^{i}; h^{n} )  \xrightarrow{ \ d \ } \omega^{i}_{st} \equiv\int_{s}^{t} x^{i}_{su} \otimes dW_{u}
+ \Big( \int_{s}^{t}x^{i-\tau}_{su} du\Big) \otimes \varrho ,
\end{eqnarray}
 for   $\ell>i\geq \tau$.  
 Indeed, we
 have shown that   convergence \eref{e6.4} holds when $i=\tau$. Now suppose that the convergence holds for $i=\tau,\dots,  {\tau}'-1$, and we wish to propagate the induction assumption.   
Thanks to the  induction assumption and in a similar way as in \eref{e3.17} we can show   that for $l \in \{1, \dots, \tau'-1\}\setminus \{\tau\}$  we have
\begin{eqnarray}\label{e3.45}
\lim_{m\to \infty} \lim_{n\to \infty} |\tilde{A}_{l, \tau'}| =0, 
\end{eqnarray}
where the limit is understood 
in probability. Moreover, invoking  assumption \eref{e3.40} we also have the following limit in probability:
\begin{eqnarray}\label{e3.46}
\lim_{m\to \infty} \limsup_{n\to \infty} |\tilde{A}_{\tau', \tau'}| =0. 
\end{eqnarray}
  On the other hand, we let the patient reader check that
\begin{eqnarray}\label{e3.47}
 \tilde{A}_{\tau, \tau'} - 
 \Big( \int_{s}^{t} x^{i-\tau}_{su}du \Big) \otimes \varrho \to 0
\end{eqnarray}
in probability, similarly to what has been done in \eref{e3.33} and \eref{e3.38}. 
Taking into account \eref{e3.45}, \eref{e3.46}, \eref{e3.47} and \eref{e-approx.cj2}, it is easily checked that \eref{e6.4} for $i=\tau'$ is reduced to the following convergence:
\begin{eqnarray}\label{e3.48}
\tilde{A}_{0, \tau'} +\Big( \int_{s}^{t} x^{\tau'-\tau}_{su}du \Big)\otimes \varrho \xrightarrow{ \ d \ } \int_{s}^{t} x^{\tau'}_{u}dW_{u} +  \Big( \int_{s}^{t} x^{\tau'-\tau}_{su}du \Big)\otimes \varrho,
\end{eqnarray}
 as $n\to \infty$ and then $m\to \infty$. 
 In order to prove \eref{e3.48}, we first fix $m$ and let $n$ go to $\infty$. Then, owing to the fact that $(x, h^{n}) \xrightarrow{f.d.d.} (x, W) $, we get that 
 \begin{eqnarray*}
\lim_{n\to \infty} \ \tilde{A}_{0, \tau'} + \Big( \int_{s}^{t} x^{\tau'-\tau}_{su}du \Big)\otimes \varrho
\stackrel{\ (d) \ }{=} \sum_{j=0}^{m-1} x^{\tau'}_{ss_{j}} \otimes \delta W_{s_{j}s_{j+1}} + \Big( \int_{s}^{t} x^{\tau'-\tau}_{su}du \Big)\otimes \varrho.
\end{eqnarray*}
Then, conditioning on $x$ and considering limits of Riemann sums for Wiener integrals, we end up with:
 \begin{eqnarray*}
\lim_{m\to \infty} 
\sum_{j=0}^{m-1} x^{\tau'}_{ss_{j}} \otimes \delta W_{s_{j}s_{j+1}} + \Big( \int_{s}^{t} x^{\tau'-\tau}_{su}du \Big)\otimes \varrho
 \stackrel{\ L_{2} \ }{=} 
 \int_{s}^{t} x^{\tau'}_{su}dW_{u} +  \Big( \int_{s}^{t} x^{\tau'-\tau}_{su}du \Big)\otimes \varrho,
\end{eqnarray*}
from which \eref{e3.48}, and thus \eref{e6.4}  for $i=\tau'$ are easily deduced. 
 Therefore, we can conclude by induction that  the convergence \eref{e6.4} holds for all $i=\tau,\dots, \ell-1$. 

\noindent \textit{Step 4: Proof of \eref{e5.4} and \eref{e3.13}.} 
Recall that $\omega^{i}$ is defined by relation  \eref{e6.2i} when $i<\tau$ and by \eref{e6.4} when $i\geq \tau$. 
For $i=1, \dots, \ell-1$, $i\neq 0$ and $i\neq \tau$, as in the proof of Theorem \ref{thm5.9} (see relation \eref{e3.31}) we can show that \eref{e5.4} holds. 
 On the other hand, it is easy to show
 by classical integration arguments
  that
\begin{eqnarray}\label{e3.49}
\cj_{0}^{1} ( y^{(\tau)};  {\omega}^{\tau}) \to  \Big(
\int_{0}^{1} y^{(\tau)}_{u}du \Big) \otimes \varrho
\end{eqnarray}
in probability and 
\begin{eqnarray}\label{e3.50}
\cj_{0}^{1} ( y ;  {\omega}^{0}) \xrightarrow{} \int_{0}^{1} y_{u} \otimes  dW_{u}
\end{eqnarray}
in probability, by convergence of Riemann sums for Wiener integrals. 
Putting together \eref{e3.49} and \eref{e3.50} and invoking the same arguments as in Step 3, we can conclude that \eref{e3.13} is satisfied. 

In conclusion, we have checked conditions \eref{eq9.6}-\eref{e3.13}, and our result follows directly from Theorem \ref{thm4.1}.  
  \end{proof}

  \section{Breuer-Major theorem} \label{sectionBM}
 In this section, we consider generalizations of Breuer-Major's theorem \cite{BM}. Notice that recent contributions (see e.g. \cite{Nourdin, NN, NNT, NR}) to this area involving weighted sums of stationary sequences  mostly consider  sequences of   functionals of   one-dimensional fractional Brownian motions (fBm). 
 This is why we also stick to the one-dimensional fBm case, though multi-dimensional studies for more general Gaussian processes do not seem out of reach in our framework. 
 Also observe that the aforementioned references focus on sequences in a fixed chaos or in a finite sum of chaos. In contrast, we will be able to handle general sequences in $L_{2}$ with respect to a Gaussian measure.  
 
 \subsection{Weighted Breuer-Major theorem I}
 In this subsection, we consider the weighted type Breuer-Major theorem in the context of our     single limit Theorem \ref{thm5.9}.  
 
 Let us first introduce some additional notation. 
  Let $d {\ga} (t) = (2\pi)^{-1/2} e^{-t^{2}/2} dt$ be the standard Gaussian measure on the real line, and let $f \in L_{2}(\ga)$ be such that $\int_{\RR} f(t) d\ga (t) =0$. It is well-known that the function $f$ can be expanded into a series of Hermite polynomials as follows:
  \begin{eqnarray*}
f(t) &=& \sum_{q=d}^{\infty} a_{q} H_{q}(t),
\end{eqnarray*}
where  $d\geq 1$ is some integer and $H_{q} (t) = (-1)^{q} e^{\frac{t^{2}}{2}} \frac{d^{q}}{dt^{q}}  e^{-\frac{t^{2}}{2}}    $ is the Hermite polynomial of order $q$. If $a_{d}\neq 0$, then $d$ is called the \emph{Hermite rank} of the function~$f$. Note that since $f \in L_{2}(\ga)$, we have $\sum_{q=d}^{\infty} |a_{q}|^{2} q! <\infty $. 

Our underlying process $X$ is a one-dimensional Gaussian sequence. For such a  process   the basic tools to measure dependence  are based on correlation functions. Throughout this subsection, we assume that the following hypothesis on correlations holds true. 
\begin{hyp}\label{hyp8.1i}
Let $X_{k}$, $k\in \ZZ$ be a centered stationary Gaussian sequence such     that   $X_{k}$ has unit variance. Denote $\rho (k) = \mE(X_{0}X_{k})$. We   suppose that $\sum_{k\in \ZZ} |\rho(k)|^{d}<\infty$ for some $d\geq 1$. 
\end{hyp}

For sake of conciseness we will not recall the basic notions of Gaussian analysis which will be used in this section. The interested reader is referred to \cite{NP} for further details.

We now recall a classical version of Breuer-Major's theorem.
\begin{theorem}\label{thm4.2}
Let $\{X_{k}, k\in \ZZ\}$ be a centered stationary Gaussian sequence satisfying Hypothesis \ref{hyp8.1i} for $d\geq 1$. Consider $f\in L_{2}(\ga)$ with rank $d$. For $n\geq 1$, let $0=t_{0}<\cdots<t_{n}=1$ be the uniform partition of $[0,1]$ defined in Section \ref{section1}.  we set $h^{n}_{st} = \sum_{s\leq t_{k}<t} f(X_{k}) $ for all $(s,t) \in \cs_{2}$. Then the following central limit theorem   holds true:
   \begin{eqnarray*}
   h^{n} /\sqrt{n}
\xrightarrow{f.d.d.}   \si W  \quad \text{as} \quad n\to \infty, 
\end{eqnarray*}
where the variance $\si^{2}  \in [0,\infty) $ is defined by:
\begin{eqnarray}\label{e-si}
\sigma^{2} = \sum_{q=d}^{\infty} q! {a}^{2}_{q} \sum_{k\in \ZZ}\rho(k)^{q}.
\end{eqnarray}
\end{theorem}

   In this subsection we specialize Theorem \ref{thm4.2} to a situation where $X_{k} = n^{\nu} \delta x_{t_{k}t_{k+1}}$, where $x$ is a fBm with Hurst parameter $\nu$. In this context we are interested in the following questions:   (1) Do we have the convergence of the weighted sum 
   \begin{eqnarray}\label{e4.1i}
\frac{1}{\sqrt{n}} \sum_{k=0}^{n-1} y_{ {k}} f(X_{k}) \quad \text{as \ } n\to \infty,
\end{eqnarray}
for a general weight $y_{k}$?
 (2) Does the central limit theorem for \eref{e4.1i} still hold in general?
 We will give a complete answer to these two questions when the weight process $y $ is a controlled process as introduced in Definition \ref{def2.2}. 
 
 Before we start our discussions, let us recall some basic facts about fBm. 
 (i) If $x$ is a one-dimensional fBm with Hurst parameter $\nu$, then $x$ is almost surely $\ga$-H\"older continuous for all $\ga<\nu$.  
 (ii) For a fBm $x$, the covariance function $\rho$ alluded to in Hypothesis \ref{hyp8.1i} is defined by 
 \begin{eqnarray}\label{e-rho}
\rho (k) = \mE(\delta x_{01}\delta x_{k,k+1}).
\end{eqnarray}
Then, 
 whenever $\nu <\frac12$, we have $\sum_{k\in \ZZ}\rho (k) = 0$. 
 
  We also label the following notation for further use. 
   \begin{notation}\label{notation4.3}
     Let $x$ be a one-dimensional fBm with Hurst parameter $\nu$. We consider $x$ as a $(L_{p}, \nu, \ell)$ rough path according to Definition \ref{def:rough-path}, where $p$ is any real number in $[1,\infty)$ and $\ell$
     is the smallest integer 
      satisfying $\nu \ell +\frac12>1$. In addition, we will choose $x^{i}_{st} = \frac{1}{i!} (\delta x_{st})^{i}$ for all $(s,t) \in \cs_{2}$ and $i=1,\dots, \ell$ for $\ell \in \NN$. 
 \end{notation}
 
Let us recall the following identity of multiple Wiener integrals. The reader is referred to e.g. \cite{H, NP, N} for more details:
\begin{lemma}
Let  $f\in L^{2}([0,1]^{p})$ and $g \in L^{2}([0,1]^{q})$ be symmetric functions. Then we have the identity
\begin{eqnarray}\label{e6.2}
I_{p}(f)I_{q}(g) &=& \sum_{r=0}^{p\wedge q} r! {{p}\choose {r} } {{q}\choose {r} }  I_{p+q-2r} (f\otimes_{r}g ),
\end{eqnarray}
where $I_{p}(f)$ is the $p$th multiple Wiener integral of $f$, and $ {{p}\choose {r} }  = \frac{p!}{r!(p-r)!}$. 
\end{lemma}

Let $\ch$ be the completion of the space of indicator functions with respect to the inner product $\langle \mathbf{1}_{[u,v]},  \mathbf{1}_{[s,t]}  \rangle_{\ch} =\mE(\delta x_{uv} \delta x_{st})$.   Let 
$0=t_{0}<\cdots<t_{n}=1$ be the uniform partition of $[0,1]$ alluded to the above and 
 $0\leq s_{0}<\cdots<s_{m }\leq 1$ be another partition of $[0,1]$ with $m\ll n$. 
 In the following, we take $\ell$ such that $\ell-1 \leq \frac{1}{2\nu } <\ell$ (or equivalently $\ell$ is the smallest integer such that $\nu\ell+\al>1$ with $\al=\frac12$).
  We set        
 \begin{eqnarray}\label{e4.3j}
 h^{n,q}_{st} = \sum_{s\leq t_{k}<t} H_{q}( n^{\nu} \delta x_{ t_{k}, t_{k+1}})
 \quad\quad\text{and}\quad\quad
\zeta^{i,q  }_{j} = \cj_{s_{j}}^{s_{j+1}} ( x^{i}; h^{n,q} )
\end{eqnarray}
 for $k=0,\dots, n-1$, $j=0,\dots, m-1$, $i=0,\dots, \ell-1$ and $q \in \NN$, where $\cj_{s_{j}}^{s_{j+1}} (x^{i}; h^{n,q})$ is given by \eref{e2.8}. We denote by $\vartheta ( q,q',i ) $    the following quantity
 \begin{eqnarray}\label{e4.3}
\vartheta ( q,q',i  ) &:=&
\mE\Big(
\sum_{j,j'=0}^{m -1}   \zeta^{i ,q }_{j}
    \zeta^{i  , q' }_{j'} \Big) .
\end{eqnarray}
  We will need the following auxiliary result.
 \begin{lemma} \label{lem8.3}
Let $x$ be a one-dimensional fBm on $[0,1]$ with Hurst parameter $\nu\leq \frac12$.  Take    ${i=1,\dots, \ell-1}$, where we recall that $\ell$ satisfies $\ell-1\leq\frac{1}{2\nu}<\ell$.
  Then for $q', q\geq \ell$  the following estimate holds true: 
  
  \noindent (i) When $ |q'-q| \leq 2i$, we have
  \begin{eqnarray}\label{e6.3}
\vartheta ( q,q',i   )  &\leq&  K   (  n^{1-2 \nu } +nm^{-2i\nu } +n^{1- \nu  } )   \sum\limits_{r=    \frac{1}{2} (q+q') - i}^{q   } r!  {q \choose r} {q' \choose r} ,
\end{eqnarray}
where $\vartheta (q,q',i)$ is defined by \eref{e4.3} and   $K$ is a positive universal constant.

\noindent (ii) When $ |q'-q| >2i$, we have
\begin{eqnarray}\label{e6.3w}
\vartheta ( q,q',i )  &=&0.
\end{eqnarray}

\noindent (iii) 
 When $|q-q'|\leq 2i$, the following inequality holds true for all $(s,t)\in \cs_{2}$:  
\begin{eqnarray}\label{e4.7}
\mE (  
\cj_{s}^{t} (x^{i}; h^{n,q})  \cj_{s}^{t} (x^{i}; h^{n,q'})     )
 &\leq& Kn (t-s)^{ 2i\nu+1 }   \sum_{r=  \frac{1}{2}(q+q')-i}^{q}  r! {q \choose r} {q' \choose r}    .
\end{eqnarray}

\noindent (iv) 
When $|q-q'|> 2i$, for all $(s,t)\in \cs_{2}$ we have:
\begin{eqnarray*}
\mE (  
\cj_{s}^{t} (x^{i}; h^{n,q})  \cj_{s}^{t} (x^{i}; h^{n,q'})     )
 &=& 0  .
\end{eqnarray*}
\end{lemma}
\begin{remark}
Notice that our assumption imply in particular that $q\wedge q'>\frac{1}{2\nu}$. This is also the condition on the Hermite rank of $f$ which will feature in Theorem \ref{thm8.5} below. 
\end{remark}
\begin{proof}[Proof of Lemma \ref{lem8.3}]
\noindent \textit{Step 1:}
Without   loss of generality let us assume that $q'\geq q$. By the definition of $\zeta^{i ,q }_{j}$  we can write
 \begin{eqnarray}\label{e6.4j}
 \vartheta ( q,q',i  )   &=&
  \sum_{j,j'=0}^{m -1}
\sum_{s_{j}\leq t_{k}<s_{j+1}} 
\sum_{s_{j'}\leq t_{k'}<s_{j'+1}} 
a (j,j',k,k') ,
\end{eqnarray}
where, recalling that $\ep (s_{j})$ is defined by \eref{e2.8}, we have 
\begin{eqnarray*}
a (j,j',k,k') &=& \mE\Big(  
x_{ \ep(s_{j}) t_{k} }^{i}  x_{ \ep(s_{j'}) t_{k'} }^{i}  H_{q}( n^{\nu} \delta x_{ t_{k}, t_{k+1}})
  H_{q'}( n^{\nu} \delta x_{ t_{k'}, t_{k'+1}})
  \Big).
\end{eqnarray*}
Now set 
    $\be_{k} = n^{\nu} \mathbf{1}_{[t_{k}, t_{k+1}]}$. 
    Recalling that $H_{q}(n^{\nu} \delta x_{t_{k}t_{k+1}}) = I_{q}(\be^{\otimes q})$ and invoking identity~\eref{e6.2}, we easily obtain:
\begin{eqnarray*}
 a (j,j',k,k') 
&=&
\sum_{r=0}^{q } r!  {q \choose r} {q' \choose r}
 \mE\Big(  
x_{ \ep(s_{j}) t_{k} }^{i}  x_{ \ep(s_{j'}) t_{k'} }^{i}   I_{q+q'-2r} (\be_{k}^{\otimes q-r} \otimes \be_{k'}^{\otimes q'-r} ) \Big) \langle \be_{k}, \be_{k'}\rangle_{\ch}^{r}.
\end{eqnarray*}
Now observe that $\langle\be_{k},\be_{k'}\rangle_{\ch} = \rho(k-k')$, where the covariance function $\rho$ is defined by \eref{e-rho}. 
Therefore, owing to an application of integration by parts, we end up with the following identity:
\begin{eqnarray} \label{e6.5}
 a (j,j',k,k')  &=& \sum_{r=   0}^{q   }  r! {q \choose r}  {q' \choose r}   b(r) \rho (k-k')^{r} ,
\end{eqnarray}
  where $b(r)$ is the coefficient defined by: 
\begin{eqnarray}\label{e4.8}
b(r) &=& \mE 
\Big\langle 
D^{q+q'-2r}( x_{ \ep(s_{j}) t_{k} }^{i}  x_{ \ep(s_{j'}) t_{k'} }^{i}    ),
 \be_{k}^{\otimes (q-r) } \otimes \be_{k'}^{\otimes (q'-r) } 
 \Big\rangle_{\ch^{\otimes (q+q'-2r)}}
   .
\end{eqnarray}

 \noindent \textit{Step 2:} Consider $q\geq \ell$.    
Due to the fact that $x^{i}$ belongs to the sum of the first $i$ chaos, 
  when  $q'-q>2i  $,    it is easy to see that 
\begin{eqnarray}\label{e8.6j}
D^{q+q'-2r}( x_{ \ep(s_{j}) t_{k} }^{i}  x_{ \ep(s_{j'}) t_{k'} }^{i}     ) =0
\end{eqnarray}
 for all $r=0,\dots, q$. Taking into account  \eref{e6.5}, this implies  that whenever $q'-q>2i$ we have 
\begin{eqnarray}\label{e6.6ii}
a(j,j',k,k') = 0,
\end{eqnarray}
and thus the    estimate in \eref{e6.3w} holds.

In the following, we assume that $0\leq q'-q\leq 2i$ and we focus on inequality \eref{e6.3}. 
Note first that since $q'\geq q$ and $q\geq \ell$, we have $\frac{1}{2} (q+q') - i \geq q- (\ell-1) >0$.

We now recall that $b(r)$ is defined by \eref{e4.8}, and we separate the estimates on $b(r)$ in several cases:

(i) Case $0\leq r<\frac12 (q+q')-i$. In this case, going back to the definition \eref{e4.8}, it is readily checked that we differentiate the product $x^{i}_{\ep(s_{j})t_{k}} x^{i}_{\ep(s_{j'})t_{k'}}$ more than $2i$ times, and hence $b(r)=0$. 

(ii) Case $\frac12 (q+q')-i\leq r\leq q-1$. In this case we still have $q+q'-2r>0$. Then we start from relation \eref{e4.8} again, we take into account the order of differentiation, and resort to the following relations, which are valid for all $k\leq n$ and $(a,b)\in \cs_{2}'$
\begin{eqnarray}\label{e4.13}
\langle \be_{k}, \mathbf{1}_{[a,b]} \rangle_{\ch}
= 
n^{-\nu} \langle \mathbf{1}_{[k,k+1]} , \mathbf{1}_{[na,nb]} \rangle_{\ch}, 
\quad\text{and} \quad |\langle \mathbf{1}_{[k,k+1]} , \mathbf{1}_{[na,nb]} \rangle_{\ch}|\leq 1
\end{eqnarray}
whenever $\nu\leq \frac12$. Then we let the patient reader check that this yields 
\begin{eqnarray}\label{e-b(r)}
|b(r)| \leq K n^{-(q+q'-2r)\nu} 
m^{-( 2i-(q+q'-2r) )\nu}
 \leq K n^{-(q+q'-2r)\nu}\leq Kn^{-2\nu}.
\end{eqnarray}

  (iii) Case   $r=q$. 
  If $q<q'$, similarly to case (ii), we can get $|b(r)|\leq K n^{-\nu}$. 
    If $r=q$ and $q=q'$, then
    $|b(r)|$ becomes $|b(r)| = | \mE[ x^{i}_{\ep(s_{j})t_{k}} x^{i}_{\ep(s_{j'})t_{k'}} ] |$, from which is easily seen that this term is bounded 
      by     $K    m^{- 2 i \nu} $. 
        
Now gathering the estimates obtained in (i)-(iii) and plugging them in \eref{e6.5}, we end up with: 
\begin{eqnarray}\label{e6.4i}
| a (j,j',k,k') |    &\leq&
 K  \Big(
 \sum_{r=  \frac{1}{2} (q+q') - i}^{q-1  }  r!  {q \choose r} {q' \choose r} n^{   -2 \nu}    | \rho (k-k')|^{r}
 \nonumber
\\
&& \quad\quad\quad\quad\quad  +   q!  {q' \choose q}   (   m^{- 2i \nu }+ n^{- \nu }) |  \rho (k-k')|^{q}
 \Big).
\end{eqnarray}
Furthermore, observe that 
\begin{eqnarray}\label{e4.15}
\cd_{n}\cap [s,t) = \bigcup_{j=0}^{m-1} \{ t_{k}; s_{j}\leq t_{k}<s_{j+1} \}.
\end{eqnarray}
Hence, 
substituting \eref{e6.4i} into \eref{e6.4j}  and using the fact that 
$\sum_{k\in \NN}|\rho(k)|<\infty$, 
  we obtain 
\begin{eqnarray}\label{e4.14}
| \vartheta ( q,q',i  )  |
&\leq& 
 K    \sum_{r=  \frac{1}{2} (q+q') - i}^{q-1  } r!  {q \choose r} {q' \choose r}
n^{-2 \nu }
 \sum_{k,k'=0}^{n-1}  |\rho (k-k')|^{r}
 \nonumber
 \\
 && + Kq!   {q' \choose q}   (m^{- 2i \nu }+n^{- \nu }) \sum_{k,k'=0}^{n-1}   |\rho (k-k')|^{q}
 \nonumber
 \\
 &\leq& K  (  n^{1-2 \nu } +nm^{-2i\nu } +n^{1- \nu  } )  \sum_{r=   \frac{1}{2} (q+q') - i}^{q  } r!  {q \choose r} {q' \choose r} 
   .
\end{eqnarray}
  This completes the proof of     inequality   \eref{e6.3}.

\noindent  \textit{Step 3:} 
In this step, we prove the estimates in (iii) and (iv). 
For $(s,t) \in \cs_{2} $ such that $t-s<\frac1n$ and with Remark \ref{remark2.8} in mind, we have $\cj_{s}^{t} (x^{i}; h^{n,q})=0$. Therefore,  in the following we assume that $ n(t-s)  \geq 1$. 
  Suppose that $q'\geq q$. Then
  similarly to \eref{e6.4j} and \eref{e6.5},
   we can derive the following expression: 
\begin{eqnarray}\label{e8.9q}
&&
\mE (  
\cj_{s}^{t} (x^{i}; h^{n,q})  \cj_{s}^{t} (x^{i}; h^{n,q'})     )
 \\
 &&
 =  \sum_{s\leq t_{k},t_{k'}<t}  \sum_{r=0}^{q} r!  {q \choose r} {q' \choose r}  \mE  \Big\langle D^{q'+q-2r}(  
x_{ \ep(s ) t_{k} }^{i}  x_{ \ep(s ) t_{k ' } }^{i}),   \beta_{k}^{\otimes (q-r)} \otimes \beta_{k'}^{\otimes (q'-r)} \Big\rangle_{\ch^{\otimes (q'+q-2r)}}    \rho (k-k')^{r} 
   .
\nonumber  
\end{eqnarray}
As in the previous step, we now separate the case $q'-q\leq 2i$ and $q'-q>2i$. 
Indeed, when  $q'-q\leq 2i$, we have seen that $\frac{1}{2}(q+q')-i>0$. Hence,  thanks to the assumption that $q'\geq q$, $q\geq \ell$ and $i\leq \ell-1$,  we obtain the following estimate along the same lines as in the previous step:
\begin{eqnarray} \label{e4.16}
\big| \mE (  
\cj_{s}^{t} (x^{i}; h^{n,q})  \cj_{s}^{t} (x^{i}; h^{n,q'})     )\big|
 & \leq&
 \sum_{r= \frac{1}{2}(q+q')-i}^{q}  r!  {q \choose r} {q' \choose r}   
 n^{-(q+q'-2r)\nu}
  (t-s )^{(2i-( q+q'-2r))\nu}
  \nonumber
  \\
  &&
\quad\quad\quad \times \sum_{s\leq t_{k},t_{k'}<t}   |\rho (k-k')|^{r}  
  .
\end{eqnarray}
Thus, resorting to the inequality $n(t-s)\geq 1$ and thanks to the fact that $\sum_{k,k'=0}^{n-1} |\rho(k-k')|^{r}$ is of order $n$, we get
\begin{eqnarray*}
\mE (  
\cj_{s}^{t} (x^{i}; h^{n,q})  \cj_{s}^{t} (x^{i}; h^{n,q'})     )
&\leq & 
Kn (t-s)^{ 2i\nu+1 }  \sum_{r=  \frac{1}{2}(q+q')-i}^{q}  r!{q \choose r} {q' \choose r},
\end{eqnarray*}
which proves (iii).

The proof of (iv) is left to the reader. Indeed, it is done exactly as for (ii), taking advantage of the fact that $D^{p} (x^{i}_{\ep(s)t_{k}} x^{i}_{\ep(s)t_{k'}}) = 0$ whenever $p>2i$. 
The proof is now complete.
\end{proof}
  
  We are ready to derive the  first main result of this section, which is a Breuer-Major type central limit theorem.
  \begin{thm}\label{thm8.5}
Let $x$ be a one-dimensional fBm with Hurst parameter $\nu \leq \frac12$. Let 
  $\ell  $ be an integer such that $\nu \ell +\frac12>1$.  
 Let $ (y ,y' , \dots, y^{(\ell-1)})$ be a   process controlled by $(x,\frac12)$ in $L_{2}$ or almost surely.
Let  $f\in L_{2}(\ga)$ with  Hermite rank   strictly bigger than $\frac{1}{2\nu }$. Suppose that one of the following conditions holds true, 

\noindent
(a) We have the expansion $f = \sum_{q=d}^{\infty} a_{q}H_{q} $, and 
\begin{eqnarray}\label{e8.10i}
\sum_{q=d}^{\infty}  a_{q}^{2}q! q^{2(\ell-1)}<\infty.
\end{eqnarray}

\noindent (b) 
The function $f$ sits in the Soblev space $W^{2(\ell-1), 2}( \RR, \ga)$, where recall that $\ga $ denotes the standard Gaussian measure on the real line.  
 
 \noindent (c) The function
  $f $ is an element of $ C^{2\ell-3}$ and $f^{(2\ell-3)}$ is Lipschitz.

\noindent We define a family of increments $\{h^{n}; n\geq 1 \}$ by:
 \begin{eqnarray}\label{e4.18}
h^{n}_{st} := \frac{1}{\sqrt{n}} \sum_{s\leq t_{k}<t} f(n^{\nu} \delta x_{t_{k}t_{k+1}}), \quad\quad (s,t)\in \cs_{2}. 
\end{eqnarray}
Then we have the   convergence:
\begin{eqnarray*}
(x, \cj(y; h^{n}))  
 \xrightarrow{f.d.d.}    \Big(x, \si\int  y_{t} dW_{t}\Big) \,,\quad\quad  \text{as \ }   n\to \infty,
\end{eqnarray*}
where $\si$ is given by \eref{e-si}.
  \end{thm}
  
    \begin{remark}
  As mentioned in the introduction, Theorem \ref{thm8.5} can be seen as a generalization as well as a simplification of  \cite{Nourdin, NNT}. 
  \end{remark}
  
  \begin{proof}[Proof of Theorem \ref{thm8.5}:] 
   We first assume   that     condition (a) is true. 
  We will prove the   theorem   thanks to our   central limit Theorem~\ref{thm5.9} applied to $h^{n}$. 
 
   To this aim, it suffices to verify   
   $(x,h^{n}) \xrightarrow{f.d.d.} (x,\omega)$, plus
    condition \eref{e3.8i}  in Proposition \ref{prop3.6} and the convergence in \eref{e6.1}. We now prove that those conditions are satisfied in separate steps. 
   
 \noindent  \textit{Step 1: Stable convergence of $h^{n}$. } The convergence in law of $h^{n}$ to $W$ is a direct consequence of Theorem \ref{thm4.2}. One can  get the stable convergence by applying a multi-dimensional version of Theorem 6.3.1 in \cite{NP}. 
 
   \noindent  \textit{Step 2: Proof of condition \eref{e3.8i}. } Recall that $f = \sum_{q=d}^{\infty} a_{q} H_{q}$, that $h^{n}$ is defined by \eref{e4.18}, and that $h^{n,q}$ has been introduced in \eref{e4.3j}. 
    Then $|  \cj_{s}^{t} ( x^{i}; h^{n}  )  |_{L_{2}}^{2} $ can be expressed as 
   \begin{eqnarray*}
|  \cj_{s}^{t} ( x^{i}; h^{n}  )  |_{L_{2}}^{2} &=& \frac{1}{n}\sum_{q,q'=d}^{\infty} a_{q}a_{q'} \mE( \cj_{s}^{t} ( x^{i}; h^{n,q}  )  \cj_{s}^{t} ( x^{i}; h^{n,q'}  ) ).
\end{eqnarray*}
We can now apply Lemma   \ref{lem8.3} (iii) and (iv) in order to get:  
\begin{eqnarray}\label{e8.11ii}
|  \cj_{s}^{t} ( x^{i}; h^{n}  )  |_{L_{2}}^{2} &\leq& Kc_{i}  (t-s)^{ 2i\nu+1 }    ,
\end{eqnarray}
where $K$ is defined by \eref{e4.7} and
\begin{eqnarray}\label{e4.20ii}
c_{i} &=&
\sum_{|q-q'|\leq 2i}  a_{q}a_{q'}      \sum_{r=  \frac{1}{2}(q+q')-i}^{q \wedge q'}  r!{q \choose r} {q' \choose r} .
\end{eqnarray}
In addition, we observe that $r\geq (q\wedge q')-\ell+1$ in the sum defining $c_{i}$. Hence, invoking the elementary bounds
\begin{eqnarray*}
{q' \choose r} \leq (q')^{q'-r} \quad \text{and} \quad r! {q \choose r} \leq q!, 
\end{eqnarray*}
plus an application of Cauchy-Schwarz's inequality for the sum over $r$, 
it is readily checked that 
\begin{eqnarray}\label{e4.21i}
c_{i}
&\leq&      \sum_{q=d}^{\infty} a_{q}^{2} q! q^{2(\ell-1)}.
\end{eqnarray}

  Taking square root in both sides of \eref{e8.11ii} and taking into account condition \eref{e8.10i} we obtain the condition \eref{e3.8i} in Proposition \ref{prop3.6}.

   \noindent  \textit{Step 3: Proof of condition \eref{e6.1}. }
   Recall that the increment $h^{n}$ is defined by \eref{e4.18}.  For $(s,t) \in \cs_{2}$, we set 
 \begin{eqnarray*}
\zeta^{i}_{j} = \cj_{s_{j}}^{s_{j+1}} (x^{i}; h^{n}) = 
\frac{1}{\sqrt{n}}
\sum_{q=d}^{\infty} a_{q} \cj_{s_{j}}^{s_{j+1}} (x^{i}; h^{n,q})
=
\frac{1}{\sqrt{n}}
\sum_{q=d}^{\infty} a_{q} \zeta^{i,q}_{j} ,
\end{eqnarray*}
where the last identity is due to our convention \eref{e4.3j}. 
  Then
  according to our notation~\eref{e4.3}, the following relation holds true for $i>0$: 
 \begin{eqnarray*}
 \Big| \sum_{j=0}^{m-1} \cj_{s_{j}}^{s_{j+1}} (x^{i}; h^{n})  \Big|^{2}_{L_{2}}
 =  
\frac{1}{n}
\sum_{q,q'=d}^{\infty} a_{q}a_{q'}   \vartheta ( q,q',i   ) .
\end{eqnarray*}
According to Lemma \ref{lem8.3}, we have $\vartheta ( q,q',i ) =0$ when $|q'-q|>2i$. Combining this with     inequality \eref{e6.3},  we obtain
 \begin{eqnarray*}
\Big|  \sum_{j=0}^{m-1}   \zeta^{i}_{j} \Big|_{L_{2}}^{2}
&=& \frac1n \sum_{|q-q'|\leq 2i} a_{q}a_{q'}   \vartheta ( q,q',i  )
\\
&\leq&   K  ( n^{-2 \nu } +m^{-2\nu } +n^{- \nu  } ) \sum_{|q-q'|\leq 2i} a_{q}a_{q'}  \sum_{r=    \frac{1}{2} (q+q') - i}^{q \wedge q'  } r! {q \choose r} {q' \choose r} 
.
\end{eqnarray*}
We now refer to our definition \eref{e4.20ii} of $c_{i}$, as well as inequality \eref{e4.21i}, which yields:
\begin{eqnarray*}
\Big|  \sum_{j=0}^{m-1}   \zeta^{i}_{j} \Big|_{L_{2}}^{2} 
&\leq& K  ( n^{-2 \nu  } +m^{-2\nu } +n^{- \nu  } ) \sum_{q=d}^{\infty}  a_{q}^{2}q! q^{2(\ell-1)} .
\end{eqnarray*}
 Taking into account the assumption \eref{e8.10i}, this implies  that
\begin{eqnarray}\label{e4.24}
\lim_{m\to \infty} \limsup_{n\to \infty} \Big|   \sum_{j=0}^{m-1}   \zeta^{i}_{j} \Big|_{L_{2}} &=&0.
\end{eqnarray}
This   complete the proof of the theorem under condition (a).

\noindent  \textit{Step 4: Proof under conditions (b) and (c). } 
It is well-known that condition (b) is equivalent to condition (a); see e.g. Page 28 in \cite{N}. On the other hand, one can show that condition (c) implies condition (b). Indeed, by Proposition 1.2.4 in \cite{N}, we obtain that $f^{(2\ell-3)}  \in W^{1,2}(\RR, \ga) $.  It is then easy to show that $f^{(2\ell-4)}\in L_{2}(\ga)$ and $ ( f^{(2\ell-4)})'   = f^{(2\ell-3)}  $, which implies that $f^{(2\ell-4)} \in W^{2,2}(\RR, \ga) $. Repeating this argument, we obtain that $f  \in W^{2\ell-2, 2}(\RR, \ga) $. Our proof is now finished. 
  \end{proof}

We now consider a central limit theorem for weights $y$ which satisfies the Young pairing condition with respect to a Brownian motion $W$ (i.e. $y$ is $\nu'$-H\"older continuous for $\nu'>\frac12$).

  \begin{prop}\label{prop8.6}
  Let   $y$ be a $\nu'$-H\"older continuous path  for some  $ \nu'>\frac12$ and let $x$ be a fBm with Hurst parameter $\nu\in (0,1)$.
   Suppose that $f \in L_{2}(\ga)$ has Hermite rank $d  $ such that $\nu<1-\frac{1}{2d}$. Then 
  the following convergence holds true
  \begin{eqnarray}\label{e8.12}
\frac{1}{\sqrt{n}} \sum_{k=0}^{n-1} y_{t_{k}} f( n^{\nu}x_{t_{k}t_{k+1}})
 \xrightarrow{ \ d \ }  \si  \int_{0}^{1} y_{t} dW_{t} \quad\quad  \text{as \ }   n\to \infty,
\end{eqnarray}
where $\si$ is defined by \eref{e-si}. 
  \end{prop}
 \begin{remark}
Proposition \ref{prop8.6} has been first proved in \cite{CNP} by means of fractional calculus techniques, and has been generalized in \cite{HLN1} to   $L_{2}$-convergence limit theorems. 
 \end{remark}
  \begin{proof}[Proof of Proposition \ref{prop8.6}:]
  As in equation \eref{e4.18}, for $(s,t)\in \cs_{2}$ we set 
    \begin{eqnarray*}
h^{n}_{st} = \frac{1}{\sqrt{n}}  \sum_{s\leq t_{k}<t}  f(n^{\nu}x_{t_{k}t_{k+1}}).
\end{eqnarray*}
  In a similar way as in \eref{e8.11ii},   Lemma \ref{lem8.3} (iii) and (iv) we can show that 
  \begin{eqnarray}\label{e8.13i}
 |  h^{n}_{st}|_{L_{2}}\leq K(t-s)^{\frac12}.
\end{eqnarray}
 Notice that we are working here under the assumption $\nu'+\frac12>1$. Therefore, an application of Theorem \ref{thm5.9} combined with Remark \ref{remark4.12} yield our claim \eref{e8.12}. 
    \end{proof}

  \subsection{Weighted Breuer-Major theorem II}
 In this subsection,  we continue our discussion on the Breuer-Major theorem, handling situations with low order Hermite ranks. We first derive some auxiliary results on    the discrete integral
$ \cj_{s}^{t}(y; h^{n,q} ) $, 
where we recall that $h^{n,q} $ is defined by \eref{e4.3j}.   
 \begin{lemma} \label{lem4.8}
 Let $x$ be a fBm with Hurst parameter $\nu$ considered as a $(L_{p}, \nu, \ell)$ rough path as in Notation \ref{notation4.3}. 
 Take    $i=1,\dots, \ell-1$ and    $q \in \NN$.  
 
 \noindent
(i)  
Let $\vartheta (q,q,i)$ be defined by \eref{e4.3}. Then
 for $q<\frac{1}{2\nu}$ and $q<i$, we have  
\begin{eqnarray}\label{e8.4k}
\vartheta ( q,q ,i  )   &\leq& K ( nm^{-4 \nu}   + n^{2-2q\nu}   m^{-2 \nu}    ).
\end{eqnarray}
  For $q=\frac{1}{2\nu}$ and $0< i<q$, we have 
\begin{eqnarray}\label{e4.23}
\vartheta ( q,q ,i  )   &\leq& K  (n^{1-2\nu} + nm^{-2i\nu}).
\end{eqnarray}

\noindent
(ii) 
Recall that $h^{n,q}_{st} = \sum_{s\leq t_{k}<t} H_{q}(n^{\nu} \delta x_{t_{k}t_{k+1}}) $ is defined by \eref{e4.3j}. Then
for $q<\frac{1}{2\nu}$  and $q>i$, we have
\begin{eqnarray}\label{e8.5u}
\mE ( |
\cj_{s}^{t} (x^{i}; h^{n,q})  |^{2}    ) &\leq & Kn(t-s)^{2i\nu+1}, \quad \text{ for  } (s,t) \in \cs_{2}'.
\end{eqnarray}
  For $  q\geq \frac{1}{2\nu}$, we have 
\begin{eqnarray}\label{e8.5k}
\mE ( |
\cj_{s}^{t} (x^{i}; h^{n,q})  |^{2}    )
 &\leq& Kn(t-s)^{2i\nu+1},  \quad \text{ for } (s,t) \in \cs_{2}' .
\end{eqnarray}
  For $q<\frac{1}{2\nu}$, and $q\leq i$, we have 
\begin{eqnarray}\label{e8.6k}
\mE ( |
\cj_{s}^{t} (x^{i}; h^{n,q})  |^{2}    ) &\leq & Kn^{2-2q\nu} (t-s)^{2+2i\nu-2q\nu}  \quad (s,t) \in \cs_{2}' .
\end{eqnarray}

 \end{lemma}
 \begin{proof} The proof is divided into several steps.
 
 \noindent \textit{Step 1: General estimate for $\vartheta$. }  
 Recall that $\vartheta (q,q,i)$ is given by \eref{e6.4j}. Next we use expression   \eref{e6.5} for $a (j,j',k,k')    $. We bound all the combination numbers by a constant and invoke the fact  that $b(r)$ (defined by \eref{e4.8}) satisfies (similarly to \eref{e-b(r)}): 
 \begin{eqnarray*}
|b(r)|&\leq& K n^{-(2q-2r)\nu} m^{-(2i-(2q-2r))\nu}
\end{eqnarray*}
for all $r\geq q-i$. 
 Therefore, similarly to  \eref{e6.4i} we get
\begin{eqnarray} \label{e8.9}
|a (j,j',k,k') |    &\leq&
 K   \Big(
 \sum_{r=  0\vee( q - i)}^{q   }  n^{  -(2q -2r )\nu}   m^{-(2i-(2q-2r))\nu  } |\rho (k-k')|^{r}
  \Big).
\end{eqnarray}

 \noindent \textit{Step 2: Case $q<\frac{1}{2 \nu }$ and $q<i$.} 
In this situation,  similarly to \eref{e4.14}, 
substituting \eref{e8.9} into \eref{e6.4j}   we obtain 
\begin{eqnarray}\label{e4.31}
|\vartheta ( q,q ,i   )   |
&\leq& 
K   
 \sum_{r=   1 }^{q   } 
n^{ -(2q -2r) \nu } m^{-(2i-(2q-2r)) \nu  }
 \sum_{k,k'=0}^{n-1}  |\rho (k-k')|^{r} 
 \nonumber
 \\
 &&+ Kn^{-2q \nu }  m^{-(2i-2q) \nu }     \sum_{k,k'=0}^{n-1}  |\rho (k-k')|^{0} 
 .
\end{eqnarray}
Therefore, owing to the fact that $\sum_{k,k'=0}^{n-1} |\rho (k-k')|^{r} =O(n) $ whenever $r\geq 1$ and properly bounding the exponents in \eref{e4.31}, we end up with 
\begin{eqnarray*}
|\vartheta ( q,q ,i   )   | 
 &\leq& K ( nm^{-4  \nu }   + n^{2-2q \nu }   m^{-2  \nu }    ).
\end{eqnarray*}
 This completes the proof of \eref{e8.4k}. 
 
  \noindent \textit{Step 3: Case   $q=\frac{1}{2 \nu }$ and  $0< i<q$.} 
If $q=\frac{1}{2\nu}$ and $\ell$ is the smallest integer such that $\nu\ell >\frac12$,   we have $\ell=q+1$. Since   $0< i<q$, then substituting \eref{e8.9} into \eref{e6.4j}   we obtain
the same inequality as \eref{e4.31}, except for the fact that the term with $\rho(k-k')^{0}$ is missing. We get 
 \begin{eqnarray*}
|\vartheta ( q,q ,i  ) |
&\leq&  K  \sum_{r=q-i}^{q } n^{-(2q-2r) \nu } 
 m^{-(2i-(2q-2r)) \nu  } 
\sum_{k,k'=0}^{n-1} |\rho (k-k')|^{r}
\\
&\leq&  K   (n^{1-2 \nu } + nm^{-2i \nu }),
\end{eqnarray*}
where   we have followed the same lines as in the previous step for the second inequality.  
This completes the proof of \eref{e4.23}.

 
  \noindent \textit{Step 4: General estimate for $\cj_{s}^{t} (x^{i}; h^{n,q})$.} 
  By \eref{e8.9q}, we have the expression: 
\begin{eqnarray*}
&&
\mE ( |
\cj_{s}^{t} (x^{i}; h^{n,q})  |^{2}    )
 \\
 &&
 =   \sum_{s\leq t_{k},t_{k'}<t}  \sum_{r=0}^{q} r! {q \choose r}^{2} \mE \Big\langle D^{2q-2r}( 
x_{ \ep(s ) t_{k} }^{i}  x_{ \ep(s ) t_{k ' } }^{i})),   \beta_{k}^{\otimes (q-r)} \otimes \beta_{k'}^{\otimes (q-r)} \Big\rangle_{\ch^{\otimes (2q-2r)}}    \rho (k-k')^{r} 
   .
\end{eqnarray*}
Therefore,
proceeding similarly to Step 1 and \eref{e4.16} and bounding all the combination numbers by a constant $K$, 
 we obtain the estimate
\begin{eqnarray}\label{e8.11k}
\mE ( |
\cj_{s}^{t} (x^{i}; h^{n})  |^{2}    )
 &\leq& K n^{-2i \nu } \sum_{r=0\vee (q-i)}^{q}    
(n(t-s))^{(2i-(2q-2r)) \nu }
  \sum_{s\leq t_{k},t_{k'}<t}  | \rho (k-k') |^{r}  .
\end{eqnarray}

 \noindent \textit{Step 5: Proof of \eref{e8.5u}, \eref{e8.5k} and \eref{e8.6k}.} 
 In order to prove \eref{e8.5k}, 
note that when $ q>\frac{1}{2 \nu }$, since $i\leq \ell-1\leq \frac{1}{2 \nu }$, we have $q>i$, and so the estimate \eref{e8.11k} implies \eref{e8.5k} due to the fact that $\sum_{s\leq t_{k},t_{k'} <t } |\rho(k-k')|^{r}\leq Kn(t-s)$ when $r\geq 1$. 
Similarly, we can show that
    estimate \eref{e8.5k} still  holds   when  $q=\frac{1}{2 \nu }$.  

In a similar way, it stems from the inequality  \eref{e8.11k}  that in the case $q<\frac{1}{2 \nu }$  and $q>i$  we have
  \eref{e8.5u}, and that in the case $q<\frac{1}{2 \nu }$  and $q\leq i$, we have the estimate \eref{e8.6k}.
 \end{proof}

   In case of a low rank $q$, we now derive some deterministic limits for Riemann sums related to $x^{q}$ and $h^{n,q}$. 
\begin{lemma}\label{lem4.9}
Let $n\geq 1$ and let $0=t_{0}<\dots<t_{n}=1$ be the uniform partition of $[0,1]$ of order $n$. Let $x$ be a standard fBm with Hurst parameter $\nu\in (0,\frac12)$.  Recall that $h^{n,q}$ is defined by \eref{e4.3j}. Consider also $n\gg m$, and a partition $0\leq s_{0}<\cdots<s_{m}\leq 1$ 
of   $[0,1]$. We   assume that  
$  |s_{i+1} - s_{i}| \leq \frac{1}{m}$, and $s_{0}= s$, $s_{m}= t$.  
Then the following limits hold true:  

\noindent (i) For $ \nu =\frac{1}{2q}$,
 we have
\begin{eqnarray} \label{e4.20i}
\lim_{m\to \infty} \lim_{n\to \infty}   \frac{1}{\sqrt{n}}   \sum_{j=0}^{m-1} \cj_{s_{j}}^{s_{j+1}} (x^{q}; h^{n,q})  =   \Big(-\frac12\Big)^{ q}  (t-s)
\end{eqnarray}
in $L_{2}$,  for  all $(s,t) \in \cs_{2}$.

\noindent (ii) For $\nu <\frac{1}{2q}$, we have
\begin{eqnarray}\label{e4.36i}
\lim_{n\to \infty}   \frac{1}{ n^{1-q\nu}}     \cj_{s }^{t} (x^{q}; h^{n,q})  =    \Big(-\frac12 \Big)^{ q}(t-s)
\end{eqnarray}
in $L_{2}$, for all $(s,t) \in \cs_{2}$.
\end{lemma} 
 \begin{proof}
 We shall only prove item (i), since item (ii) can be treated along the same lines. Our global strategy is based on identity \eref{e6.4j} and \eref{e6.5}, as in the proof of Lemma \ref{lem4.8}, with a more in-depth analysis of the terms appearing in our decomposition. 
 
 Indeed, formula \eref{e6.4j} together with \eref{e6.5}
 assert that 
\begin{eqnarray*}
\mE\Big[ \Big(\frac{1}{\sqrt{n}} \sum_{j=0}^{m-1} \cj_{s_{j}}^{s_{j+1}} (x^{q}; h^{n,q}) \Big)^{2} \Big]
=  {n^{-1 }} \vartheta (q,q,q)   
= \sum_{r=0}^{q} a(r),
\end{eqnarray*}
where
 \begin{eqnarray} \label{e4.36}
      a(r) &=&
n^{ -1 }   \sum_{j,j'=0}^{m -1}
\sum_{s_{j}\leq t_{k}<s_{j+1}} 
\sum_{s_{j'}\leq t_{k'}<s_{j'+1}} 
r! {q \choose r}^{2} \rho (k-k')^{r} b(r) ,
\end{eqnarray}
 and
where
\begin{eqnarray}\label{e4.37}
b(r) &=& \mE 
\Big\langle 
D^{2q -2r}(  x_{ \ep(s_{j}) t_{k} }^{q}  x_{ \ep(s_{j'}) t_{k'} }^{q}   ),
 \be_{k}^{\otimes (q-r) }  {\otimes} \be_{k'}^{\otimes (q -r) } 
 \Big\rangle_{\ch^{\otimes ( 2q -2r)}}.
\end{eqnarray}
We now split the analysis of the terms $a(r)$ and $b(r)$. 

 \noindent \textit{Step 1: Case $r>0$.} 
 The term $a(r)$ for $r>0$ can be bounded as follows, along the same lines as in the proof of Lemma \ref{lem8.3} and Lemma \ref{lem4.8}. Namely, we bound all the combination numbers by a constant, we use identity \eref{e4.15} and the fact that $\sum_{k,k'=0}^{n-1} |\rho(k-k')|^{r}=O(n)$ in order to get
 \begin{eqnarray*}
|a(r)|&\leq & K |b(r)|. 
\end{eqnarray*}
In order to bound $b(r)$, we resort to identity \eref{e4.37}. Then we observe that each term $D^{2q -2r}(  x_{ \ep(s_{j}) t_{k} }^{q}  x_{ \ep(s_{j'}) t_{k'} }^{q}   )$ is of order $m^{-r\nu}$, while each contribution of the form $\langle \mathbf{1}_{[a,b]}^{\otimes (q-r)}, \be_{k}^{\otimes (q-r)} \rangle_{\ch^{\otimes (q-r)}}$ can be bounded by a constant (similarly to \eref{e4.13}). This yields 
\begin{eqnarray*}
 a(r)     \leq K m^{- 2  \nu } .
\end{eqnarray*}
Therefore, it is readily checked that 
 $\lim_{m\to \infty}\lim_{n\to \infty} a(r)=0 $.

 \noindent \textit{Step 2: Decomposition of $a(0)$ and $b(0)$.} When   $r =0$, formula \eref{e4.36} can be read as:
\begin{eqnarray*}
a(0) &=& n^{-1}  \sum_{j,j'=0}^{m -1}
\sum_{s_{j}\leq t_{k}<s_{j+1}} 
\sum_{s_{j'}\leq t_{k'}<s_{j'+1}}   b(0)  .
\end{eqnarray*}
Notice that $D^{2q}(x^{q}_{\ep(s_{j})t_{k}}x^{q}_{\ep(s_{j'})t_{k'}})$ is a deterministic function of the form $h_{2q} =g_{1,q}\otimes g_{2,q}$, where $h_{2q}$ is a function of $2q$ variables, and each $g_{1,q}$, $g_{2,q}$ is a function of $q$ variables. In addition, the reader can check that $g_{1,q}$ contains $q'$ (resp. $q-q'$) tensor products of indicator functions $\mathbf{1}_{[\ep(s_{j}), t_{k}]}$ (resp. $\mathbf{1}_{[\ep(s_{j'}), t_{k'}]}$), and $g_{2,q}$ contains $q-q'$ tensor products of functions $\mathbf{1}_{[\ep(s_{j}), t_{k}]}$, for some $0\leq q'\leq q$. 
Pairing those functions with $\be_{k}$ and $\be_{k'}$, we get the following identity:
\begin{eqnarray*}
b(0) &=& \sum_{q'=0}^{q} b(0,q'),
\end{eqnarray*}
 where  
\begin{eqnarray*}
  b(0,q') &=&    {q  \choose q'}^{2}
\langle  \mathbf{1}_{[\ep(s_{j}), t_{k}] }, \be_{k}  \rangle_{\ch}^{q-q'}
\langle  \mathbf{1}_{[ \ep(s_{j'}) , t_{k'} ]}, \be_{k'}  \rangle_{\ch}^{q-q'}
  \langle  \mathbf{1}_{[ \ep(s_{j'}), t_{k'} ]}, \be_{k}  \rangle_{\ch}^{q'}
   \langle  \mathbf{1}_{[\ep(s_{j}), t_{k} ]}, \be_{k'}  \rangle_{\ch}^{q'}
 . 
\end{eqnarray*}

 \noindent \textit{Step 3: Study of $b(0,q')$ for $q'>0$.}
 Let us observe that, thanks to the fact that $2q\nu =1$, we have 
 \begin{eqnarray}\label{eqn.c}
b(0,q') &=& {q \choose q'}^{2} n^{-1} \hat{b}(0,q') \tilde{b}(0,q'),
\end{eqnarray}
 where 
 \begin{eqnarray*}
\hat{b} (0,q') &=& 
\langle n^{\nu}  \mathbf{1}_{[\ep(s_{j}), t_{k} ]}, \be_{k}  \rangle_{\ch}^{q-q'}
\langle  n^{\nu}\mathbf{1}_{[ \ep(s_{j'}), t_{k'} ]}, \be_{k'}  \rangle_{\ch}^{q-q'}
\end{eqnarray*}
and
\begin{eqnarray*}
\tilde{b} (0,q') &=& 
  \langle  n^{\nu} \mathbf{1}_{[ \ep(s_{j'}), t_{k'} ]}, \be_{k}  \rangle_{\ch}^{q'}
   \langle  n^{\nu}  \mathbf{1}_{[\ep(s_{j}), t_{k} ]}, \be_{k'}  \rangle_{\ch}^{q'}.
\end{eqnarray*}
In order to bound $\hat{b} (0,q')$ we can proceed as in \eref{e4.13} and we just get 
\begin{eqnarray}\label{eqn.b}
|\hat{b}(0,q')| &\leq & K.
\end{eqnarray}
We now turn to a bound on $\tilde{b} (0,q')$. Some scaling arguments similar to \eref{e4.13} reveal that 
\begin{eqnarray}\label{eqn.a}
|\tilde{b} (0,q')| &\leq& 
\langle \mathbf{1}_{[\lceil ns_{j'} \rceil, k']} , \mathbf{1}_{[k,k+1]}\rangle_{\ch}^{q'}
 \langle \mathbf{1}_{[\lceil ns_{j} \rceil, k']} , \mathbf{1}_{[k',k'+1]}\rangle_{\ch}^{q'}.
\end{eqnarray}
We now obtain uniform bounds on $\tilde{b}(0,q')$ according to the values of $j$, $j'$.

(i) If $|j-j'|\geq 2$, then 
we also have $|k-k'|\geq \frac{n}{m}$ in \eref{eqn.a}. Hence
it is readily checked that
\begin{eqnarray*}
 |\langle \mathbf{1}_{[\lceil ns_{j'} \rceil, k']} , \mathbf{1}_{[k,k+1]}\rangle_{\ch} |
&\leq& K_{\nu} \int_{[\lceil ns_{j'} \rceil, k'] \times [k,k+1]} \frac{dudv}{ |u-v|^{2-2\nu}}
\\
&  \leq&  K_{\nu}(\frac{n}{m})^{2\nu-2} (\frac{n}{m}) = (\frac{m}{n})^{1-2\nu}, 
\end{eqnarray*}
and the same bound holds true for $ \langle \mathbf{1}_{[\lceil ns_{j} \rceil, k']} , \mathbf{1}_{[k',k'+1]}\rangle_{\ch}$. Hence we have
\begin{eqnarray*}
|\tilde{b} (0,q')| &\leq& K_{\nu}  (\frac{m}{n})^{1-2\nu }. 
\end{eqnarray*}

(ii) If $|j-j'|\leq 2$, then we simply bound $\tilde{b} (0,q')$ by a constant, just as in \eref{e4.13} and \eref{eqn.b}. 
Plugging those estimates into \eref{eqn.c},
 it is now readily checked that  
\begin{eqnarray*}
\lim_{n\to \infty} n^{-1}  \sum_{j,j'=0}^{m -1}
\sum_{s_{j}\leq t_{k}<s_{j+1}} 
\sum_{s_{j'}\leq t_{k'}<s_{j'+1}}   b(0,q') 
\leq K_{\nu } \lim_{n\to \infty} n^{-2} \sum_{k,k'=0}^{n-1}  (\frac{m}{n})^{1-2\nu }
 = 0.
\end{eqnarray*}
  Therefore, the limit  of $a(0)$ is equal to the limit of 
\begin{eqnarray*}
\tilde{a}(0) &:=&
  n^{-1}  \sum_{j,j'=0}^{m -1}
\sum_{s_{j}\leq t_{k}<s_{j+1}} 
\sum_{s_{j'}\leq t_{k'}<s_{j'+1}}   b(0,0). 
\end{eqnarray*}

 \noindent \textit{Step 4: Convergence of $a(0)$.}
We use some notation of the previous step: we have $b(0,0) = \hat{b} (0,0) $, and we apply the same scaling arguments as before. We end up with  
\begin{eqnarray}\label{e4.21}
\tilde{a}(0) &=& \frac{1}{n^{2}} 
 \sum_{j,j'=0}^{m -1}
\sum_{s_{j}\leq t_{k}<s_{j+1}} 
\sum_{s_{j'}\leq t_{k'}<s_{j'+1}}  
\langle  \mathbf{1}_{ [\lceil ns_{j}\rceil, {k} ]},  \mathbf{1}_{[k,k+1]}  \rangle_{\ch}^{q}
\langle  \mathbf{1}_{ [\lceil ns_{j'}\rceil, {k'} ]},  \mathbf{1}_{[k',k'+1]}  \rangle_{\ch}^{q}
\nonumber
\\
&=&  \Big( n^{ -1} \sum_{j=0}^{m-1} \sum_{s_{j}\leq t_{k}<s_{j+1}}  \langle 
\mathbf{1}_{[ \lceil ns_{j}\rceil, {k} ]},  \mathbf{1}_{[k,k+1]}
  \rangle_{\ch}^{q}  \Big)^{2}
  \nonumber
  \\
  &=& 
   \Big(  \sum_{j=0}^{m-1} \frac{m_{j}}{n} \frac{1}{m_{j}} \sum_{s_{j}\leq t_{k}<s_{j+1}}  \langle  \mathbf{1}_{[ \lceil ns_{j}\rceil, {k} ]},  \mathbf{1}_{[k,k+1]}  \rangle_{\ch}^{q}  \Big)^{2},
\end{eqnarray}
where $m_{j} = \#\{t_{k}: s_{j}\leq t_{k}<s_{j+1}\}$. 
Note that by the stationarity of increments of $x$ 
and recalling that $\rho (i) =  \mE[ \delta x_{01} \delta x_{i,i+1}]$  
 we have
\begin{eqnarray*} 
\frac{1}{m_{j}} \sum_{s_{j}\leq t_{k}<s_{j+1}}\langle   \mathbf{1}_{[ \lceil ns_{j}\rceil, {k} ]},  \mathbf{1}_{[k,k+1]}  \rangle_{\ch}^{q}  
&=&
 \frac{1}{m_{j}}    \sum_{s_{j}< t_{k}<s_{j+1}} \Big(
 \sum_{i=1}^{k - n\ep(s_{j})   } \rho (i)\Big)^{q} 
 \nonumber
\\
&=& 
 \frac{1}{m_{j}}   \sum_{ 0< k < m_{j}  }  \Big( \sum_{i=1}^{ k } \rho (i)\Big)^{q}  .
\end{eqnarray*}
Taking into account the fact that $\lim_{n\to \infty} \sum_{i=1}^{k}\rho(i) =-\frac12 \rho(0)$ (remember that $\sum_{k\in \ZZ} \rho (k)=0$), a Cesaro mean argument shows that
\begin{eqnarray*}
\lim_{n\to \infty} 
 \frac{1}{m_{j}}   \sum_{ 0< k < m_{j}  }  \Big( \sum_{i=1}^{ k } \rho (i)\Big)^{q} 
 =
(-\frac{\rho (0)}{2})^{q} = (-\frac12)^{q}.
\end{eqnarray*}
Plugging this information back into \eref{e4.21} and taking into account the fact that   $ \lim_{n\to \infty} \frac{m_{j}}{n} = (s_{j+1} - s_{j} ) $, we obtain that 
 \begin{eqnarray*}
\lim_{n\to \infty} \tilde{a}(0) = (\sum_{j=0}^{m-1} (s_{j+1} - s_{j})   (-\frac12)^{q} )^{2}= ((t-s) (-\frac12)^{q})^{2}.
\end{eqnarray*}
We can thus conclude that 
\begin{eqnarray}\label{e4.39}
\lim_{m\to \infty} \lim_{n\to \infty} 
\mE\Big[ \Big(\frac{1}{\sqrt{n}} \sum_{j=0}^{m-1} \cj_{s_{j}}^{s_{j+1}} (x^{q}; h^{n,q}) \Big)^{2} \Big]
=
\lim_{m\to \infty} \lim_{n\to \infty} a(0) = (-\frac12)^{2q} (t-s)^{2}.
\end{eqnarray}

 \noindent \textit{Step 5: Conclusion.} 
 With relation \eref{e4.39} in hand, the convergence \eref{e4.20i} is reduced   to show the convergence of the first moment of $  \cj_{s_{j}}^{s_{j+1}} (x^{q}; h^{n,q})  $. Furthermore, we have 
 \begin{eqnarray*}
\frac{1}{\sqrt{n}}\mE \big[ \cj_{s_{j}}^{s_{j+1}} (x^{q}; h^{n,q}) \big] &=& n^{-\frac12 }   \sum_{s_{j}\leq t_{k}<s_{j+1}} \mE \big[  x_{ \ep(s_{j}) t_{k} }^{q} H_{q}(n^{\nu}x_{t_{k},t_{k+1}}) \big].
\end{eqnarray*}
Rescaling and integrating by parts we get:
\begin{eqnarray*}
\frac{1}{\sqrt{n}}\mE \big[ \cj_{s_{j}}^{s_{j+1}} (x^{q}; h^{n,q}) \big]
&=&  \frac{1}{n}  \sum_{s_{j}\leq t_{k}<s_{j+1}} \langle  \mathbf{1}_{[ n \ep(s_{j}),  {k} ]}, \mathbf{1}_{[k,k+1]} \rangle_{\ch}^{q}.
\end{eqnarray*}
With the same arguments as for \eref{e4.39}, we end up with: \begin{eqnarray*}
\lim_{m\to \infty}\lim_{n\to \infty}
\frac{1}{\sqrt{n}}\mE \sum_{j=0}^{m-1} \cj_{s_{j}}^{s_{j+1}} (x^{q}; h^{n,q})  = (t-s)(-\frac12)^{q}. 
\end{eqnarray*}
The proof of \eref{e4.20i} is now complete.  
\end{proof}

   We are now ready to state a weighted type Breuer-Major theorem which generalizes \cite[Theorem 5.3]{NN} and \cite[Theorem 1.1]{NR} to weights given by a controlled process.
   
      \begin{prop} \label{prop6.4}
      Let $x$ be a fBm with Hurst parameter $\nu\in (0,\frac12)$ considered as a $(L_{p}, \nu, \ell)$ rough path as in Notation \ref{notation4.3}, and consider $q>0$. 
      
      We define $h^{n,q} $ by relation \eref{e4.3j}. Let $y$ be a discrete process controlled by $(x, 1-q\nu)$ in $L_{2}$ or almost surely. 
  Then the following convergences hold  true:

(i) When $ \nu =\frac{1}{2q}$, we have 
\begin{eqnarray*}
(x,  n^{-\frac12} \cj_{0}^{1} (y ; h^{n,q}) )\xrightarrow{f.d.d.} 
\Big(x, 
\si    \int  y_{u} dW_{u}
+
\Big(-\frac{1}{2}\Big)^{q} \int  y^{(q)}_{u}du 
\Big),
\end{eqnarray*}
where $\si = q! \sum_{k\in \ZZ} \rho (k)^{q}$.

(ii) When $ \nu <\frac{1}{2q}$, we get
\begin{eqnarray*}
n^{-(1-q \nu) } \cj_{0}^{1}(y ; h^{n,q}) \xrightarrow{ \   \ } \Big(-\frac12\Big)^{q} \int_{0}^{1} y^{(q)}_{u} du,
\end{eqnarray*}
 where the convergence holds in probability.
  \end{prop}
 \begin{proof}
   In the case $ \nu =\frac{1}{2q}$, 
   the condition $\nu \ell +(1-q\nu) = \nu\ell+\frac12 >1$ can be read as   $\ell =q+1$. 
   We now invoke Theorem \ref{thm7.1}. 
   Indeed, condition \eref{e3.8i} is ensured by \eref{e8.5k}, condition (i) in Theorem \ref{thm7.1} is just Breuer-Major's Theorem \ref{thm4.2}, and condition \eref{e6.1i} has been proved in~\eref{e4.20i}. Moreover, in our situation, condition \eref{e3.40} has to be checked for $i<q$, and is easily shown thanks to inequality \eref{e4.23}. Therefore, a direct application of Theorem \ref{thm7.1} yields our claim (i). 
   
   In order to get item (ii), we apply Theorem \ref{cor6.3}. In this case, condition \eref{e3.8i} is a consequence of \eref{e8.5u} and \eref{e8.6k}. Item (i) in Theorem \ref{cor6.3} is a consequence of \eref{e4.36i}  and~\eref{e8.5u}. Eventually, \eref{e5.4i} is obtained through \eref{e8.4k}. This concludes the proof. 
 \end{proof}
 
 We now go one step further in the generalization, and handle the case of a weighted sum in an infinite number of chaos. 
 \begin{thm}\label{thm4.14}
Let $x$  be 
       a fBm with Hurst parameter $\nu\in (0,\frac12)$ considered as a $(L_{p}, \nu, \ell)$ rough path as in Notation \ref{notation4.3}. 
    Let   $ (y ,y' , \dots, y^{(\ell-1)})$ be a discrete process   controlled by $(x, 1- \nu d)$ as in Proposition \ref{prop6.4} and take $\ell=d+1$. Let   $f = \sum_{q=d}^{\infty} a_{q}H_{q} \in L_{2}(\ga)$ be a function with  Hermite rank    $d>0$ satisfying one of the conditions (a), (b) and (c) of Theorem \ref{thm8.5}. 
   Set
  \begin{eqnarray*}
h^{n}_{st} = n^{d\nu-1}\sum_{s\leq t_{k}<t} f(n^{\nu} \delta x_{t_{k}t_{k+1}}), \quad\quad (s,t)\in \cs_{2}.
\end{eqnarray*}   
 Then the following limits hold true. 
   
\noindent (i) When $d=\frac{1}{2\nu}$   we have the   convergence:  
\begin{eqnarray}\label{e4.42}
(x, \cj(y; h^{n}) )
 \xrightarrow{f.d.d.}  \Big(x,  \si   \int  y_{t} dW_{t} +\Big(-\frac12\Big)^{d}a_{d} \int y^{(d)}_{u} du \Big), \quad\quad  \text{as \ }   n\to \infty,
\end{eqnarray}
where $\si$ is given by \eref{e-si}.

\noindent (ii) When $d<\frac{1}{2\nu}$ we get the following  convergence in probability:
\begin{eqnarray*}
  \cj_{s}^{t}(y; h^{n}) 
 \xrightarrow{ \   \ }   \Big(-\frac12\Big)^{d} a_{d}\int_{s}^{t}y^{(d)}_{u} du, \quad\quad  \text{as \ }   n\to \infty.
\end{eqnarray*}
 \end{thm}
 \begin{proof} 
  \noindent \textit{Step 1: A decomposition of $f$.} 
  In order to prove the convergence \eref{e4.42} we 
  invoke  Theorem \ref{thm7.1}. It remains to verify that conditions in Theorem \ref{thm7.1} are satisfied. To this aim, we define a new function
\begin{eqnarray}\label{e4.43i}
\tilde{f} := f - a_{d} H_{d}= \sum_{q=d+1}^{\infty} a_{q} H_{q},
\end{eqnarray}
and denote
\begin{eqnarray*}
\tilde{h}^{n}_{st} = \frac{1}{\sqrt{n}} \sum_{s\leq t_{k}<t} \tilde{f}(n^{\nu} \delta x_{t_{k}t_{k+1}}), \quad\quad (s,t)\in \cs_{2}.
\end{eqnarray*}

Now 
recalling that $h^{n,d}$ is defined by \eref{e4.3j}, 
we   write
\begin{eqnarray}\label{e4.43}
\cj_{s}^{t} (x^{i}; h^{n}) &=& a_{d} \,n^{-\frac12}  \cj_{s}^{t} (x^{i}; h^{n,d})+ \cj_{s}^{t} (x^{i}; \tilde{h}^{n}).
\end{eqnarray}
This decomposition will be used in order to verify the assumptions of Theorem \ref{thm7.1}. 

 \noindent \textit{Step 2: Proof of condition \eref{e3.8i}.} 
 We first note that relation \eref{e8.5k} implies that the quantity $\cj_{s}^{t} (x^{i}; h^{n,d})$ on the right-hand side of \eref{e4.43} satisfies condition \eref{e3.8i}. On the other hand, 
since $\tilde{f}$ satisfies the conditions of Theorem \ref{thm8.5}, it follows from the proof of Theorem \ref{thm8.5} that $\cj_{s}^{t} (x^{i}; \tilde{h}^{n})$ also satisfies \eref{e3.8i}. Combining these two  observations and applying the 
triangle inequality for the $L_{2}$-norm
   to \eref{e4.43}, we obtain \eref{e3.8i} for $\cj_{s}^{t} (x^{i}; h^{n}) $.

 \noindent \textit{Step 3: Stable convergence of $h^{n}$.}
  The proof of the stable convergence of $h^{n}$ follows the same lines as in Theorem \ref{thm8.5}. It  is omitted for sake of conciseness.

\noindent \textit{Step 4: Proof of \eref{e6.1i}.}  
We have already noticed that 
 $\tilde{f}$   defined in \eref{e4.43i} satisfies the conditions of Theorem \ref{thm8.5}. So it follows from the proof of Theorem \ref{thm8.5} that $\tilde{h}^{n}$ satisfies the relation  \eref{e6.1}. More precisely,   the following convergence for $i=1,\dots, d$ is obtained similarly to \eref{e4.24}:
\begin{eqnarray}\label{e4.45}
\lim_{m\to \infty} \lim_{n\to \infty}
\Big| \sum_{j=0}^{m-1}     \cj_{s_{j}}^{s_{j+1}} ( x^{i}; \tilde{h}^{n} ) \Big|_{L_{2}} =0.
\end{eqnarray}
On the other hand, it follows from \eref{e4.20i} that $n^{-\frac12}  h^{n,d}$    satisfies \eref{e6.1i}.
Putting together \eref{e4.45} and the convergence of   $ \cj_{s_{j}}^{s_{j+1}} ( x^{i};  {h}^{n,d} )$ and  
 taking into account \eref{e4.43} we obtain the convergence \eref{e6.1i} for  $h^{n}$. 

\noindent \textit{Step 5: Proof of \eref{e3.40}.}   
As in the previous step, invoking relation \eref{e4.43}, it suffices to    consider the relation \eref{e3.40} for   $\tilde{h}^{n}$ and $h^{n,d}$ separately. Notice that relation \eref{e3.40} for $\tilde{h}^{n}$ follows directly from \eref{e4.45}. 
On the other hand, relation \eref{e3.40} for $h^{n}=h^{n,q}$ is obtained exactly as in the proof of Proposition \ref{prop6.4}, thanks to \eref{e4.23}.   This completes the proof of \eref{e3.40} for $h^{n}$.

  \noindent \textit{Step 6: Proof of item (ii).} 
  Item (ii) is obtained by applying Theorem \ref{cor6.3}. We have to verify the same kind of conditions as in the previous steps. Resorting to our decomposition \eref{e4.43}, this is done similarly to Step 2-5, applying Proposition   \ref{prop6.4} and Theorem \ref{thm8.5}.  Details are left to the reader. 
   The proof is now complete.
 \end{proof}

  \subsection{Realized power variations and parameter estimations}
  The convergence of realized power variations is closely related to 
   the parameter estimation problem of the volatility process (see e.g. \cite{BCP} and \cite{LL} in a fBm context). Here we shall consider generalizations of realized power variations to rougher situations, and then discuss briefly the parameter estimation problem.

  Let us start by introducing some additional notation. For $p>-1$, we denote   
  \begin{eqnarray}\label{e4.46i}
 c_{p} = \mE(|N|^{p}) = \frac{2^{p/2} }{\sqrt{\pi}} \Gamma \left(\frac{p+1}{2}\right).
\end{eqnarray}
Notice that when $p$ is an even integer we can also write $c_{p}=\mE(N^{p}) = (p-1)(p-3)\cdots 1$. 
  We also consider the function $H :\RR\to\RR$ defined by   $H (x) = {|x|^{p}}  -c_{p}$.

 It is easy to see  that $H \in L_{2}(\ga)$ when $p>-\frac12$ and $H$ has Hermite rank $d= 2$.  
 One can also  verify  that $H$ has   the decomposition $H(x) = \sum_{q=1}^{\infty}  {a}_{2q}H_{2q} (x) $, where the constants $a_{2q}$ are 
 obtained by expanding the function $|x|^{p}-c_{p}$ on the Hermite basis, and are expressed
  in terms of the $c_{p}$'s:
  \begin{eqnarray}\label{e4.46}
a_{2q} &=& \sum_{r=0}^{  q } \frac{(-1)^{r}}{2^{r} r! (2q-2r)!}  (c_{2q-2r+p} - c_{p}c_{2q-2r}  ).
\end{eqnarray} 
For example, we will use the fact that $a_{2} =     p c_{p} /2  $. 

Our first result in this subsection concerns the weighted power variations of $x$ by a controlled process $y$. We focus on the rough situation $\nu\leq \frac12$, since more regular situations are handled in e.g. \cite{BCP, LL} and implied by our Proposition \ref{prop8.6}. 
\begin{thm}\label{thm8.8}
Let $x$  be 
       a fBm with Hurst parameter $\nu\in (0,\frac12)$, considered as a $(L_{p}, \nu, \ell)$ rough path as in Notation \ref{notation4.3}. 
       Let   $(y^{(0)},\dots, y^{(\ell-1)})$ be a discrete process controlled by $(x, \al)$, in $L_{2}$ or almost surely, with $\nu\ell+\al>1$.
Then the following limits for weighted power variations hold true:

(i)  Suppose  that $\frac12\geq \nu >\frac14$ and $\al=\frac12$. Then for $p\geq 2$ we have the convergence:
\begin{eqnarray}\label{e8.14i}
\frac{1}{\sqrt{n}} \sum_{k=0}^{n-1} y_{t_{k}} ({| n^{\nu}\delta x_{t_{k}t_{k+1}}|^{p}}  - c_{p}) \xrightarrow{ \ d \ }  {\sigma} \int_{0}^{1} y_{t} dW_{t},
\end{eqnarray}
where $W$ is a Wiener process independent of $x$ and $\si^{2}$ is defined by (recall that $a_{2q}$ is defined by \eref{e4.46}): 
\begin{eqnarray}\label{e4.48}
 {\sigma}^{2} = \sum_{q=1}^{\infty} (2q)!  {a}^{2}_{2q} \sum_{k\in \ZZ}\rho(k)^{2q}   .
\end{eqnarray}

(ii) Suppose that $\nu = \frac14$ and $\al=\frac12$. Then for $p \geq 4$ we have the convergence:
\begin{eqnarray*}
\frac{1}{\sqrt{n}} \sum_{k=0}^{n-1} y_{t_{k}} ({| n^{\nu}\delta x_{t_{k}t_{k+1}}|^{p}}  - c_{p}) \xrightarrow{ \ d \ }  {\sigma} \int_{0}^{1} y_{t} dW_{t}+ \frac{ {a}_{2}}{4}   \int_{0}^{1} y_{t}''dt  ,
\end{eqnarray*}
where $\si$ is defined by \eref{e4.48} and where we recall that, according to \eref{e4.46}, we have $a_{2} = \frac{pc_{p}}{2}$. 

(iii) Suppose that $\nu < \frac14$ and $\al=1-2\nu $. Then for $p\geq 4$ we have the following  convergence in probability:
\begin{eqnarray*}
n^{2\nu -1} \sum_{k=0}^{n-1} y_{t_{k}} ({| n^{\nu}\delta x_{t_{k}t_{k+1}}|^{p}}  - c_{p}) \xrightarrow{ \   \ }   \frac{ {a}_{2}}{4}  \int_{0}^{1} y_{t}''dt  .
\end{eqnarray*}
\end{thm}
  \begin{proof}
  Recall that we have set $H(x) = |x|^{p}-c_{p}$ and that the Hermite rank of $H$ is $d=2$. Then
   item (i) follows immediately from  Theorem \ref{thm8.5}. Indeed,   since $\nu>\frac14$, we have $\frac{1}{2\nu}<2$, and so the Hermite rank of $|x|^{p}-c_{p}$ is larger than $\frac{1}{2\nu}$. On the other hand, it is easy to see that $\ell =2 $ for the definition of our controlled process $y$ under the condition that $\nu \ell +\al>1$ and $\al=\frac12$. 
    So for $p\geq 2$ we have $H\in C^{2\ell-2}$ and thus $H$ satisfies condition (c) in Theorem \ref{thm8.5}.  
    Therefore, a direct application of Theorem   \ref{thm8.5} 
    yields the convergence   \eref{e8.14i}. Item (ii) and item (iii) follows from  Theorem \ref{thm4.14}. The proof is similar and is omitted. 
  \end{proof}
    
  We now consider a controlled process of order $2$ with respect to $x$, called 
  $(z,z')$. Recall that $(z,z')$ satisfies:
\begin{eqnarray}\label{e4.51i}
  |r^{z}_{st}|\leq G (t-s)^{2\nu}, \quad\quad  \text{with} \quad\quad 
 r^{z}_{st} := \delta  z_{st}-z'_{s} \delta x_{st}  , 
 \end{eqnarray}
where $G$ is some almost surely finite random variable. In the following 
we  prove the convergence 
of   the $p$-variation of $z$ with the help of Theorem \ref{thm8.8}. We will see that,  with a proper normalization,  the $p$-variation of the (first-order) increments   $ \sum_{k=0}^{n-1} |\delta z_{t_{k}t_{k+1}} |^{p} $ converges almost surely to the quantity $c_{p}\int_{0}^{1} |z'_{s}|^{p}ds$  (one can also use ``longer filters'', i.e. replacing the increments $\delta z_{t_{k}t_{k+1}}   $ by  the second-order increments $\delta z_{t_{k}t_{k+1}} - \delta z_{t_{k-1}t_{k}} $ or higher-order increments for instance; see   e.g. \cite{Tu}). 
Observe that our motivation for  this limit result     is the parameter estimation of
  the diffusion coefficient for SDEs; see e.g. \cite{BCP, LL}. 
 Indeed, consider the following equation governed by a fBm with Hurst parameter $\nu \in (0,\frac12]$: 
  \begin{eqnarray}\label{e4.28}
z_{t} &=& \int_{0}^{t} b(z_{s}) ds + \int_{0}^{t} v (z_{s})dx_{s}.
\end{eqnarray}
In equation \eref{e4.28}, the coefficient $b$ and $v$ are assumed to be $C^{2}_{b}$ and $C^{3}_{b}$, respectively. The stochastic integral in \eref{e4.28} is understood thanks to the abstract rough paths theory (see e.g. \cite{FH, FV, G}), by considering the rough path $\{ x^{i}, \, 1\leq i \leq \lfloor \frac{1}{\nu} \rfloor \}$, where $x^{i}$ is given in Notation \ref{notation4.3}. Taking $z'= v(z)$, it is well-known that the pair $(z,z')$ is a process controlled by $x$. Then our limit result for $(z,z')$ implies that the $p$-variation of the solution of \eref{e4.28} converges almost surely to the average of the volatility $c_{p}\int_{0}^{1} |v(z_{s})|^{p} ds$.

\begin{cor}\label{thm4.15}
 Let $x$ be a one-dimensional fBm with Hurst parameter $0<\nu\leq\frac12$, and let $(z,z')$ be a controlled process of $x$ satisfying \eref{e4.51i}. 
 Let $n\geq 1$ and consider the uniform partition $0=t_{0}<\cdots<t_{n}=1$ of $[0,1]$. Then for $p>\frac{1}{2\nu}$,   we have almost surely the convergence of $p$-variations of $z$:
\begin{eqnarray}\label{e4.52i}
\frac{1}{n}\sum_{k=0}^{n-1}| n^{\nu} \delta z_{t_{k}t_{k+1}}|^{p}  \to c_{p}\int_{0}^{1} |z_{t}'|^{p}dt.
\end{eqnarray}
 
\end{cor}
 
\begin{proof} 
Set 
$
\varphi:= n^{p\nu-1} \sum_{k=0}^{n-1} |z_{t_{k}}'|^{p} |\delta x_{t_{k}t_{k+1}}|^{p}$. 
We write 
\begin{eqnarray}\label{e4.52}
\frac{1}{n} \sum_{k=0}^{n-1}  | n^{\nu} \delta z_{t_{k}t_{k+1}}|^{p} &=&\varphi+R_{n}, 
\end{eqnarray}
where $R_{n}$ is simply  $\frac{1}{n} \sum_{k=0}^{n-1}  | n^{\nu} \delta z_{t_{k}t_{k+1}}|^{p}-\varphi$. Using the inequality $||a|^{p}-|b|^{p}|\leq p (|a|^{p-1}+|b|^{p-1}) |a-b|$ for $p>1$ and the regularity of $z$ and $x$, we obtain
\begin{eqnarray*}
|R_{n}| &\leq& p n^{p\nu-1} \sum_{k=0}^{n-1} ( |\delta z_{t_{k}t_{k+1}}|^{p-1} + |z_{t_{k}}'\delta x_{t_{k}t_{k+1}}|^{p-1} )| \delta z_{t_{k}t_{k+1}} - z_{t_{k}}' \delta x_{t_{k}t_{k+1}} |
\\
&\leq& G n^{p\nu-1} \sum_{k=0}^{n-1} n^{-(p-1)\nu} n^{-2\nu}= Gn^{-\nu} . 
\end{eqnarray*}
 In particular, we have  $\lim_{n\to \infty} R_{n} =0$  almost surely as $n\to \infty$. 
 On the other hand, a  direct application of Theorem \ref{thm8.8}
 shows that  $\varphi \to  c_{p} \int_{0}^{1} |z_{t}'|^{p}dt$. Putting together the convergence of $R_{n} $ and $\varphi$ and taking into account \eref{e4.52}, we obtain the desired limit \eref{e4.52i}.  
\end{proof}
 \subsection{Stratonovich integrals}
 In this subsection we are shedding a new light on another problem which has drawn a lot of attention in the recent stochastic analysis literature. Namely, we   are interested in the convergence of the following trapezoidal-rule Riemann sum:
 \begin{eqnarray}\label{e6.7}
 \text{tr-}\cj_{0}^{1} (y ; x):= \sum_{k=0}^{n-1} \frac{ y_{t_{k}} + y_{t_{k+1}} }{2} \delta x_{t_{k}t_{k+1}}.
\end{eqnarray}
This quantity has been considered by many authors (see e.g. \cite{BS, CN, GNRV, HN, HN2, HN3, NNT, NRS}) in the case $y_{s} = f(x_{s})$. 
Thanks to the rough paths technique developed in this paper, we will be able to get shorter proofs than in the aforementioned articles, and obtain results which are valid for a wider class of weight processes $y$.  
We will also see that the limit of \eref{e6.7} can be identified with the rough integral $\int_{0}^{1} y_{s} dx_{s} $   for $\nu>\frac16$ and that it is equal to the same rough integral plus  a ``correction'' term when $\nu=\frac16$.

Let us start by some preliminary   results.

  \begin{lemma}\label{prop6.1}
  Let $q$ be an integer such that $q>1$, and 
  assume that
    $  \nu \in (\frac{1}{2q}, \frac12)$.  Let $x$ and $y$ be as in Theorem \ref{thm8.5}. 
 Then the following   convergence holds as $n\to \infty$:
  \begin{eqnarray}\label{e6.3ii}
n^{-\frac12} \cj_{s}^{t} ( y ; h^{n,q} ) \xrightarrow{\ d \ }
\si  \int_{s}^{t} y_{u} dW_{u},
\end{eqnarray}
where 
$
\si^{2} =    {q!}  \sum_{k\in \ZZ}\rho(k)^{q} 
$.
  \end{lemma}
 \begin{proof}  The lemma follows immediately from Theorem \ref{thm8.5} with $f=H_{q}$. Notice that $\nu>\frac{1}{2q}$ by assumption, thus we also have $q>\frac{1}{2\nu}$, which is one of the assumption in Theorem \ref{thm8.5}. 
 \end{proof}

Our second preliminary result concerns weighted power variations of the fBm $x$.     
 \begin{lemma}\label{cor6.5}
 Let $x$ be a fBm with Hurst parameter $\nu \in (0,\frac12)$ considered as a $(L_{p},\nu, \ell)$ rough path as in Notation \ref{notation4.3}. 
 
 (i) Let $(y,y')$ be a discrete process controlled by $(x, 1-\nu)$. 
 When $q\geq 3$   is odd,
 we have the following convergence in probability:
 \begin{eqnarray}\label{e6.6j}
n^{(q+1) \nu -1} \sum_{k=0}^{n-1} y_{t_{k}}  (\delta x_{t_{k}t_{k+1}})^{q} \xrightarrow{ \ \ } 
-\frac{c_{q+1}}{2} \int_{0}^{1} y_{s}' ds,
\end{eqnarray}
where the constants $c_{p}$ are defined by \eref{e4.46i}. 

(ii) Let $(y, \dots, y^{(\ell-1)})$ be a discrete process controlled by $(x,\al)$, in $L_{2}$ or almost surely, with $\nu\ell+\al>1$, where $\al=\frac12$ for $\nu\in [\frac14, \frac12)$ and $\al=1-2\nu$ for $\nu \in (0,\frac14)$. 
When $q$ is even, we have the   convergence in probability:
\begin{eqnarray}\label{e6.6jj}
n^{q\nu-1}\sum_{k=0}^{n-1} y_{t_{k}} (\delta x_{t_{k}t_{k+1}})^{q} \xrightarrow{ \ \ } c_{q} \int_{0}^{1} y_{s}ds.
\end{eqnarray}

 \end{lemma}
 \begin{proof}
 We first show the convergence \eref{e6.6j} with the help of Theorem \ref{thm4.14}. Note that whenever $q$ is odd the function $f(x)=x^{q}$ has rank $d=1$, so we have $d<\frac{1}{2\nu}$. The convergence \eref{e6.6j} then follows from Theorem \ref{thm4.14} (ii). 
 In order to prove \eref{e6.6jj} we start by observing that an easy consequence of \eref{e8.14i} is that the following limit in probability holds true:
 \begin{eqnarray}\label{e4.58}
\lim_{n\to \infty} \frac{1}{n} \sum_{k=0}^{n-1} y_{t_{k}} ( (n^{\nu} \delta x_{t_{k}t_{k+1}})^{q} -c_{q} )&=&0 . 
\end{eqnarray}
Then observe that $y$ is a continuous process. Therefore, we trivially have the following limit in probability:
\begin{eqnarray}\label{e4.59}
\lim_{n\to \infty} \frac{1}{n} \sum_{k=0}^{n-1} y_{t_{k}} &=& \int_{0}^{1} y_{s} ds. 
\end{eqnarray}
Combining \eref{e4.58} and \eref{e4.59}, the convergence \eref{e6.6jj} is established for $\nu\in (\frac14, \frac12)$. The cases $\nu=\frac14$ and $\nu<\frac14$ are treated in the same way, thanks to (respectively) Theorem \ref{thm8.8} (ii) and~(iii). 
  \end{proof}

We can now state a convergence result for trapezoidal Riemann sums. 
\begin{thm}
Let $x$ be a one-dimensional fBm with Hurst parameter $\nu\in (0, \frac12)$. Let $y$ be an almost sure controlled process of order $\ell=8$ (see Definition \ref{def2.2}). Recall that the trapezoidal sums of $y$ with respect to $x$ are defined by \eref{e6.7}, and we set  
\begin{eqnarray}\label{e6.8ii}
\int_{0}^{1} y_{s} d^{\text{tr}} x_{s} = \lim_{n\to \infty}  \text{tr-}\cj_{0}^{1} (y ; x),
\end{eqnarray}
whenever the limit in the right-hand side is properly defined. 
Then the following assertions hold true:

(i)  When $\nu>\frac16$, the convergence \eref{e6.8ii} holds almost surely and we have the identity:
\begin{eqnarray}\label{e6.8j}
\int_{0}^{1} y_{s} d^{\emph{tr}} x_{s} = \int_{0}^{1} y_{s} dx_{s} ,
\end{eqnarray}
where $\int_{0}^{1} y_{s}dx_{s}$ stands for the rough path integral of $y$ with respect to $x$. 

(ii) When $\nu=\frac16$, the convergence \eref{e6.8ii} holds in distribution and the following relation holds true:
\begin{eqnarray}\label{e4.62}
\int_{0}^{1} y_{s} d^{\emph{tr}} x_{s} = \int_{0}^{1} y_{s} dx_{s} + \frac{\si}{12}    \int_{0}^{1} y_{s}''dW_{s},
\end{eqnarray}
where $\int_{0}^{1}y_{s}dx_{s}$ is understood in the rough path sense and 
 $\si=6\sum_{k\in \ZZ}\rho(k)^{3}$. 
\end{thm}

\begin{proof} 
\noindent \textit{Step 1: Decomposition of tr-$\cj_{0}^{1}(y; x)$.} 
Owing to the Definition \ref{def2.2} of a controlled process, we have
\begin{eqnarray}\label{e6.6}
\delta y_{st} &=&
\sum_{i=1}^{5} \frac{1}{i!} y^{(i)}_{s} (\delta x_{st})^{i} +r^{y}_{st} ,
\end{eqnarray}
where the remainder $r$ satisfies   $|r^{y}_{st}|_{L_{2}}\leq K(t-s)^{6\nu}$. 
 Plugging \eref{e6.6} into \eref{e6.7} we obtain
\begin{eqnarray*}
   \text{tr-}\cj_{0}^{1} (y ; x)
   &=&
     \sum_{k=0}^{n-1}   y_{t_{k}}   \delta x_{t_{k}t_{k+1}}
     +
      \frac12  \sum_{k=0}^{n-1}  
      \delta y_{t_{k}t_{k+1}}   \delta x_{t_{k}t_{k+1}}
   \\
   &=& 
     \sum_{k=0}^{n-1}  
     \Big(
      y_{t_{k}}   \delta x_{t_{k}t_{k+1}}
     +
      \frac12   
      \sum_{i=1}^{5} \frac{1}{i!} y^{(i)}_{t_{k}} (\delta x_{t_{k}t_{k+1}})^{i+1} +  \frac12 r_{t_{k}t_{k+1}}    \delta x_{t_{k}t_{k+1}}  \Big)  , 
\end{eqnarray*}
where
we notice that our rough path type expansion is a natural generalization of the Taylor type expansions of $f(x)$ performed in e.g. \cite{HN, NRS}. We now split the expansion as 
\begin{eqnarray}\label{e4.63}
  \text{tr-} \cj_{0}^{1} (y ; x) = a_{0}+a_{1}+a_{2}+a_{3},
\end{eqnarray}
 where 
\begin{eqnarray}\label{e4.65}
a_{0} &=&  \sum_{k=0}^{n-1} \sum_{i=0}^{5} 
\frac{1}{(i+1)!}
y^{(i)}_{t_{k}} (\delta x_{t_{k}t_{k+1}})^{i+1};
\nonumber
\\
 {a_{1}} &=&   \frac1{12}  \sum_{k=0}^{n-1}      y^{(2)}_{t_{k}} (\delta  x_{t_{k}t_{k+1}})^{3}  +  \frac1 {24}  \sum_{k=0}^{n-1}        y^{(3)}_{t_{k}}  (\delta  x_{t_{k}t_{k+1}})^{4}  
:= a_{11}+a_{12}  ,
\end{eqnarray}
and
\begin{eqnarray}\label{e4.66i}
 {a_{2}} = 
 \frac{1}{80} \sum_{k=0}^{n-1} y^{(4)}_{t_{k}} (\delta x_{t_{k}t_{k+1}})^{5} 
 +
 \frac{1}{360} \sum_{k=0}^{n-1} y^{(5)}_{t_{k}} (\delta x_{t_{k}t_{k+1}})^{6} ;
\quad   a_{3} =  \frac12  \sum_{k=0}^{n-1}   r^{y}_{t_{k}t_{k+1}}    \delta x_{t_{k}t_{k+1}}.
\end{eqnarray}
We now consider these terms separately.

\noindent \textit{Step 2: Terms $a_{0}$ and $a_{3}$.}
The remainder $r^{y}$ has a H\"older regularity of order $6\nu\geq 1$. Therefore, it is readily checked that
$a_{3} \to 0$   \text{ almost surely for } $\nu\geq \frac16$.
In addition, we have the convergence  
$a_{0} \to \int_{0}^{1} y_{u} dx_{u}$   \text{ almost surely whenever } $\nu\geq \frac16$,
  ensured by the abstract rough paths theory  (see e.g. \cite{FH, G}). It is worth noticing   at this point that 
  the convergence of $a_{0}$  is obtained in a much easier way in a rough path context  than by means of integrations by parts as performed in e.g.  \cite{HN, HN2, NNT}.

\noindent \textit{Step 3: Decomposition of $a_{1}$ for $\nu>\frac16$.} 
Among the terms defining $a_{1}$ \eref{e4.65}, we  focus on the lower order term $a_{11}$ (which potentially brings most difficulties). 
Thus we expand $a_{11}$ by writing
$\xi^{3} = H_{3}(\xi)-3H_{1}(\xi)$, where we recall that  $H_{k}$ stands for the Hermite polynomial of order $k$. This yields
$
a_{11} = b_{1}+ {b}_{2}$,
where 
\begin{eqnarray}\label{e4.67i}
b_{1} =  \frac1{12n^{3\nu}}    \sum_{k=0}^{n-1}      y^{(2)}_{t_{k}}   H_{3}(n^{\nu}\delta  x_{t_{k}t_{k+1}}), \quad \text{and} \quad 
 {b}_{2} = \frac{1}{4n^{2\nu}} \sum_{k=0}^{n-1} y^{(2)}_{t_{k}} \delta x_{t_{k}t_{k+1}}. 
\end{eqnarray}
Moreover, thanks to Lemma \ref{prop6.1}, it is readily checked that 
$b_{1}\to 0 $ \text{in probability when } $\nu> \frac16$. 
We now focus on the term $ {b}_{2}$. Since $y^{(2)}$ is itself a controlled process of order $6$, a slight elaboration of \cite[Corollary 10.15]{FV} shows that 
\begin{eqnarray*}
\int_{t_{k}}^{t_{k+1}} y^{(2)}_{s}dx_{s}   - \sum_{i=0}^{3} \frac{1}{(i+1)!}  y^{(i+2)}_{t_{k}} (\delta x_{t_{k}t_{k+1}})^{i+1} 
&:=&  r^{y^{(2)}}_{t_{k}t_{k+1}},
\end{eqnarray*}
  where $    r^{y^{(2)}}_{t_{k}t_{k+1}}$ is a remainder of order $5\nu$: 
\begin{eqnarray}\label{e4.66}
|  r^{y^{(2)}}_{t_{k}t_{k+1}} |_{L_{2} } &\leq& K n^{-5\nu}. 
\end{eqnarray}
Summing this identity over $k$, we thus get
\begin{eqnarray}\label{e4.67}
 {b}_{2} &=& \frac{1}{4n^{2\nu}}   \int_{0}^{1} y^{(2)}_{s}dx_{s} -   \sum_{i=1}^{3} b_{3}^{(i)}
  - \frac{1}{4n^{2\nu}}\sum_{k=0}^{n-1}  r^{y^{(2)}}_{t_{k}t_{k+1}}   ,
\end{eqnarray}
where each $b_{3}^{(i)}$ is defined by 
\begin{eqnarray}\label{e4.70}
b_{3}^{(i)}=\frac{1}{4n^{2\nu}}
\sum_{k=0}^{n-1}   \frac{1}{(i+1)!} y^{(i+2)}_{t_{k}} (\delta x_{t_{k}t_{k+1}})^{i+1}.
\end{eqnarray}
In expression \eref{e4.67}, it is easily seen that, thanks to \eref{e4.66}, we have 
\begin{eqnarray*}
\lim_{n\to \infty} \frac{1}{n^{2\nu}}  \sum_{k=0}^{n-1}  r^{y^{(2)}}_{t_{k}t_{k+1}} =0 \quad \text{and} \quad \lim_{n\to \infty} \frac{1}{n^{2\nu}} \int_{0}^{1} y^{(2)}_{s}dx_{s} =0,
\end{eqnarray*}
where the limits stand for limits in probability. Owing to \eref{e6.6j} and \eref{e6.6jj}, the reader can also check that 
$
\lim_{n\to \infty} b_{3}^{(i)} = 0$
for $i=2,3$. 
In order to analyze the right-hand side of \eref{e4.67} 
we are thus left with the term $ {b}_{3}^{(1)}$ defined by \eref{e4.70}.

\noindent \textit{Step 4: Terms $ {b}_{3}^{(1)} $ and $a_{12}$.}
Comparing $b^{(1)}_{3}$  with the expression \eref{e4.65} for $a_{12}$, we see that  
\begin{eqnarray}\label{e4.71}
a_{12} -  {b}_{3}^{(1)} &=& \frac{1}{24 n^{4\nu}} \sum_{k=0}^{n-1} y^{(3)}_{t_{k}} f(n^{\nu}\delta x_{t_{k}t_{k+1}}), 
\end{eqnarray}
where the function $f$ is given by $f(\xi) = \xi^{4}-3\xi^{2}$. In addition, 
invoking elementary properties of Hermite polynomials, it is easily seen that $f$ has a Hermite rank of $d=2$. Hence, according to the values of $\nu$, we can either apply Theorem \ref{thm8.5} (for $\frac14<\nu<\frac12$), Theorem \ref{thm4.14} (i) (for $\nu=\frac14$) or Theorem \ref{thm4.14} (ii) (for $\frac16<\nu<\frac14$). As an example, when $\frac16<\nu<\frac14$, we get
\begin{eqnarray*}
| a_{12} -  {b}^{(1)}_{3} |_{L_{2}} &\leq & \frac{K}{n^{6\nu-1}},
\end{eqnarray*}
which obviously goes to $0$ as $n$ goes to $\infty$. In summary of the convergence of $a_{12}- {b}_{3}^{(1)}$ and the analysis in Step 3, we obtain the convergence $a_{1}\to 0$ in probability as $n\to \infty$. 

\noindent \textit{Step 5: Terms   $a_{2}$ and conclusion for $\nu>\frac16$.}
The convergence in the case $\nu>\frac16$ is now easily obtained. Indeed, due to Lemma \ref{cor6.5}, it is readily checked that $\lim_{n\to \infty} a_{2}=0$ in probability. Therefore, combining  the convergence of $a_{0}$, $a_{1}$, $a_{2}$, $a_{3}$ and taking into account \eref{e4.63} we  obtain the convergence $  \text{tr-}\cj_{0}^{1} (y ; x)  \to \int_{0}^{1} y^{(0)}_{u}dx_{u}$ in probability, which identifies the two sides of equation \eref{e6.8j}. 

\noindent \textit{Step 6: Case $\nu=\frac16$.}
The proof for the case $\nu=\frac16$ follows the same arguments as for $\nu>\frac16$. However, in the current situation more terms are contributing to the limit. Specifically, the terms $a_{2}$, $b_{1}$, $b_{3}^{(2)}$, $b_{3}^{(2)}$ and $a_{12} - b_{3}^{(1)}$, respectively defined by \eref{e4.66i}, \eref{e4.67i}, \eref{e4.70} and \eref{e4.71}, are now converging to non-zero limits. In order to handle the term $b_{1}$, 
 we apply   Proposition \ref{prop6.4} (i)  with $q=3$. This yields    the   convergence: 
\begin{eqnarray*}
(x, b_{1} )  \xrightarrow{f.d.d.} \Big(x, 
 \frac{1}{12} \si   \int_{0}^{1} y^{(2)}_{s}dW_{s}-   \frac{1}{96} \int_{0}^{1}y^{(5)}_{s}ds 
\Big),
\end{eqnarray*}
where $\si=3!\sum_{k\in \ZZ}\rho(k)^{3}$. 
On the other hand, 
applying \eref{e6.6j}, \eref{e6.6jj} respectively to the two terms of $a_{2}$ in \eref{e4.66i}, we obtain the convergence in probability:
$
\lim_{n\to \infty}a_{2} =  - \frac{5}{96}   \int_{0}^{1} y^{(5)}_{s}ds  $.  
Moreover, owing  respectively to  \eref{e6.6j},  \eref{e6.6jj}  and Theorem \ref{thm4.14} (ii) (with $d=2$ and $\nu=3$)  we obtain the convergence: 
\begin{eqnarray*}
 \lim_{n\to \infty} b_{3}^{(2)} = -\frac{1}{16}   \int_{0}^{1} y^{(5)}_{t}dt ,
 \quad\quad
  \lim_{n\to \infty} b_{3}^{(3)} = \frac{1}{32}   \int_{0}^{1} y^{(5)}_{t}dt ,
 \quad\quad
  \lim_{n\to \infty} (a_{12}-\tilde{b}_{1}) = \frac{1}{32}  \int_{0}^{1} y^{(5)}_{t}dt. 
\end{eqnarray*}
 Putting together those additional convergences, and noticing that the terms involving $y^{(5)}$   cancels,   we end up with relation \eref{e4.62}.  
The proof is now complete. 
\end{proof}

\section{Multi-dimensional Gaussian processes}\label{sectionGuansian}
Our method of analysis for limit theorems has potentially many applications in multidimensional settings. For sake of conciseness, we will restrict ourselves to an application concerning multidimensional quadratic variations. In this way we recover (in a more elementary way) a central limit theorem contained in \cite{NTU} and used in \cite{LT}. We are also able to generalize this central limit theorem to a wide class of Gaussian processes (Section \ref{section5.1}), and obtain a weighted version in Section \ref{section5.2}.  
 
\subsection{Preliminaries on Gaussian rough paths}  
\label{section5.1}

Throughout the section we 
assume that $X =(x^{ 1}, \dots, x^{ d})$ is a centered continuous Gaussian process with i.i.d. \!\!\!components. 
We shall write $X^{1}_{uv} $ for the increments $\delta X_{uv} = X_{v} - X_{u}$ of the process $X$. Then we define the covariance of the increments of $X$ as: 
 \begin{eqnarray*}
\mE(X^{1,i}_{uv}X^{1,i}_{st}) &=& R\Big( \begin{matrix}  u&v\\s&t   \end{matrix} \Big),
\end{eqnarray*}
 where $i$ stands for any of the components of $X$. 
 We now recall some basic facts about the constructions of a rough path lift above $X$, borrowed from \cite{FV}. 
 
 The basic assumption in order to be able to lift $X$ as a rough path is that $R$ admits a two-dimensional $\rho$-variation for $\rho\in [1,2)$. Denote $\nu=\frac{1}{2\rho}$. 
 For sake of simplicity, we will moreover assume that the $\rho$-variation of $R$ satisfies:
 \begin{eqnarray}\label{e5.1}
|R|_{\rho\text{-var}, [s,t]}&\leq& K(t-s)^{2\nu},
\end{eqnarray}
where $|R|_{\rho\text{-var}, [s,t] }$ stands for the $2$-dimensional $\rho$-variation of $R$ in the interval $[s,t]^{2}$:
\begin{eqnarray*}
|R|_{\rho \text{-var}, [s,t] } &=& \sup_{(t_{i}), (t_{j}') \in \cd([s,t])} \left( \sum_{i,j} 
\left|R\Big( \begin{matrix}  t_{i}&t_{i+1}\\t_{j}'&t_{j+1}'   \end{matrix} \Big) \right|^{\rho} \right)^{1/\rho},
\end{eqnarray*}
where $\cd([s,t])$ denotes the set of partitions on  the interval $[s,t]$.
As mentioned in \cite[Remark 2.4]{CHLT}, we can assume \eref{e5.1} holds true without loss of generality, up to a deterministic time change.  Then it is shown in \cite{FV} that there exists
   a unique continuous $G^{3}(\RR^{d})$-valued process $\mathbf{X} =(1,X^{1},X^{2},X^{3})$, where $G^{3}(\RR^{d})$ stands for the free nilpotent Lie group of order $3$,  such that:

\noindent (i) $\mathbf{X} $ ``lifts'' the Gaussian process $X$ in the sense $\pi_{1} (\mathbf{X} ) = X^{1}_{t}-X^{1}_{0}$;

\noindent (ii) There exists $C=C(\nu)$ such that for all $s<t$ in $[0,1]$ and $q \in [1,\infty)$,
\begin{eqnarray*}
|\mathbf{X} _{st}|_{L_{q}} &\leq & C \sqrt{q}|s-t|^{\nu},
\end{eqnarray*}
where $d$ stands for \emph{the Carnot-Caratheodory  distance on $G^{3}(\RR^{d})$}. 

\noindent (iii) For all $\ga<\nu$ there exists $\epsilon =\epsilon(p,\nu,C)>0$ such that
\begin{eqnarray*}
\mE \big( \exp \big(\epsilon |\mathbf{X} |_{\ga}^{2} \big) \big) &<& \infty, 
\end{eqnarray*}
 where $|\mathbf{X} |_{\ga}$ designates the $\ga$-H\"older semi-norm of $\mathbf{X} $.

\subsection{Unweighted   limit theorem}\label{section5.2}
With the construction of Section \ref{section5.1} in hand, let us consider the second level $X^{2}$ of the rough path above $X$ considered as a $\RR^{d\times d}$-valued increment. 
In this subsection, we are interested in the convergence of the following random sum:
\begin{eqnarray}\label{e7.1}
  \sum_{k=0}^{n-1} \Big( n^{2\nu}X^{2}_{t_{k}t_{k+1}} -\frac12 \id \Big) .
\end{eqnarray}
Here $\id$ stands for the identity matrix. 
It is clear that \eref{e7.1} is the generalization of the quadratic variation $  \sum_{k=0}^{n-1} (   |n^{\nu} \delta X_{t_{k}t_{k+1}}|^{2} - 1 ) $ to a multi-dimensional setting. Moreover, the quantity \eref{e7.1} is related to the analysis of numerical schemes for rough SDEs (see e.g. \cite{LT}). 

The following assumption will be used heavily in our future computations.  
\begin{hyp}\label{hyp7.1}
Consider a Gaussian process $X$ whose
 covariance $R$ satisfies \eref{e5.1}. Suppose that $X$ has stationary increments
 in the sense that the variance of its increments is given by 
 \begin{eqnarray*}
\mE(|X^{1}_{st}|^{2}) &=& F(|t-s|) \geq 0,
\end{eqnarray*}
with $F$ continuous, nonnegative and with $F(0)=0$. 
   In addition, the following properties hold true: 

 \noindent (i) Either $F'' \geq 0$ or $F'' \leq 0$, in distributional sense on $(0,T)$. In other word, either $F''$ or $-F''$ is a nonnegative Radon measure on $(0,T)$.

\noindent(ii) There exists a constant $\theta:2-2\nu\geq\theta> \frac12$   such that 
\begin{eqnarray*}
|F''| \leq C    /{ t^{\theta}} 
\end{eqnarray*}
holds true for $t$ large, in distributional sense on $(0,T)$ for some $C>0$. 
\end{hyp}
 \begin{remark}
 Condition (i) in Hypothesis \ref{hyp7.1} says that 
 the Gaussian process $X^{1}$ has either negative or positive correlation, that is, the covariance $R \Big( \begin{matrix} u&v\\s&t \end{matrix} \Big) $ has the same sign for all disjoint intervals $[u,v]$ and $[s,t]$.  
 Condition (ii) implies that the correlation of two disjoint increments $\mE(X^{1}_{s, s+h}X^{1}_{t,t+h})$ decays at a rate of $|t-s|^{-\theta}$, where
      $h,s, t$ are such that $t > s+h$. In terms of the covariance function, Condition (ii)  implies the relation
\begin{eqnarray}\label{e9.1}
   R  \Big( 
\begin{matrix} s&  s+h\\ t&t+h \end{matrix}
\Big) \leq  C   {u^{2}}/{ |t-s|^{\theta}}    
\end{eqnarray}
 for $|t-s|$ large. Examples of Gaussian processes satisfying Hypothesis \ref{hyp7.1} include (sums of)
 multi-dimensional fBms with Hurst parameters $\nu \in ( \frac14 , \frac34) $. The readers are referred to \cite{FGGR} for a discussion on the properties of this type of Gaussian processes. 
 \end{remark} 
 
 Let us now define some parameters  that will appear in the limit   of \eref{e7.1}. Namely,
 denote by $X^{2,kl}_{st}$ the $(k,l)$-th element  of the matrix $X^{2}$  
  for $k,l \in \NN$, and set:
 \begin{eqnarray}
\lambda^{n}_{kl}  = n^{4\nu}  \mE(X^{2,12}_{t_{k}t_{k+1}}X^{2,12}_{t_{l}t_{l+1}} )
 \,, 
\quad\quad
\rho^{n}_{kl}  =n^{4\nu} \mE(X^{2,12}_{t_{k}t_{k+1}}X^{2,21}_{t_{l}t_{l+1}} ).
\label{eq8.4}
\end{eqnarray}
 We will need the following hypothesis: 
\begin{hyp}\label{hyp5.3}
Let $\rho_{kl}$, $\lambda_{kl}$ be the sequences defined by  \eref{eq8.4}. We assume that the following limit holds true: 
\begin{eqnarray}\label{e8.6i}
 \lambda   = \lim_{n \to \infty} \frac{1}{n} \sum_{k,l=0}^{n-1} \lambda^{n}_{kl}\,, 
\quad\quad 
 \rho   = \lim_{n\to \infty} \frac{1}{n} \sum_{k,l=0}^{n-1} \rho^{n}_{kl},
 \end{eqnarray}
where $\rho $ and $\lambda $ are finite constant (This type of assumption also appears in \cite{A}, for instance). 
  \end{hyp}
  
 \begin{remark}
 It is readily checked that Hypothesis \ref{hyp5.3} is satisfied for a 2-dimensional enhanced fractional Brownian motion with Hurst parameter $\nu\in(\frac14,\frac12)$.
 \end{remark}

  With those preliminaries in hand, let us state the main result of this subsection.      
\begin{prop}\label{prop8.3}
 Let $X = (1, X^{1}, X^{2}, X^{3})$ be the enhanced Gaussian process above the $d$-dimensional Gaussian process $X^{1}$. Suppose that Hypothesis \ref{hyp7.1} and Hypothesis \ref{hyp5.3} holds.  Set 
 \begin{eqnarray}\label{e7.6}
h^{n}_{st} &=& \sum_{k=\lceil ns\rceil}^{\lfloor nt\rfloor-1} ( n^{2\nu}X^{2}_{t_{k}t_{k+1}} -\frac12 \id ) 
\end{eqnarray}
   for $t \geq \frac1n $ and $h^{n}_{t} =0$ for $t<\frac1n$. 
  Then the finite dimensional distributions of $(n^{ -\frac12} h^{n} , X )  $ converge  weakly to those of $(W, X)$, where $W= (W^{ij} )$ is an $m\times m$-dimensional Brownian  motion, independent of $X$, such that 
 \begin{eqnarray}\label{e8.6}
\mE[W^{ij}_{t}W^{i'j'}_{s}] &=&   (\lambda\delta_{ii'}\delta_{jj'} + \rho \delta_{ij'}\delta_{ji'} ) (t\wedge s). 
\end{eqnarray}
In formula \eref{e8.6}, we have set $\delta_{ij}=1$ if $i=j$ and $\delta_{ij}=0$ if $i\neq j$. Furthermore, the quantities    $\rho$ and  $\lambda$ are defined by relation \eref{e8.6i}.
 \end{prop}

We will state and prove several intermediate results, and then prove Proposition   \ref{prop8.3} 
at the end of the subsection. The first of these lemmas concerns  covariances of $X^{2}$, for which we introduce some additional notation. Namely, we consider the specific   case when $d=2$, and analyze the weak convergence of the  two processes $z^{n}$ and $\tilde{z}^{n}$ defined  by: 
 \begin{eqnarray}\label{e8.1} 
 z^{n}_{t}  =
  n^{2\nu}   \sum_{k=0}^{ \lfloor  {nt}{ } \rfloor  }   
   X^{2,12}_{t_{k}t_{k+1}},
\quad \text{and} \quad 
\tilde{z}^{n}_{t}  =
   n^{2\nu}  \sum_{k=0}^{ \lfloor  {nt}{ } \rfloor  }   
  X^{2,21}_{t_{k}t_{k+1}}  
   \,.
\end{eqnarray} 
  We first prove a lemma on the moment convergence of   $z^{n}$, $\tilde{z}^{n}$. 
  
 \begin{lemma}\label{lem8.1}
 Let $X^{1,1} $ and $X^{1,2}$ be two independent real-valued (incremental) Gaussian processes. We set the $2$-dimensional process
  $X^{1} =(X^{1,1},X^{1,2})$ and   consider the rough path $\mathbf{X}$ above  $X^{1}$, as in Proposition \ref{prop8.3}.  Let $z^{n}$ and $\tilde{z}^{n}$ be defined in \eref{e8.1}.  Suppose that Hypothesis~\ref{hyp7.1} and Hypothesis \ref{hyp5.3} hold for $X^{1}$.  Then the following limits hold true:
 \begin{eqnarray}\label{eq8.6}
\lim_{n\to\infty} n^{ -1} \mE[|z^{n}_{t}|^{2 }]  
 =     \lambda t  
\quad\quad \text{and} \quad\quad 
\lim_{n\to\infty}n^{ -1} \mE[z^{n}_{t} \tilde{z}^{n}_{t} ]  
 =   \rho t,
\end{eqnarray}
where $\lambda$ and $\rho$ are defined in \eref{e8.6i}. 
 \end{lemma}

 \begin{proof}First,  by the definition of $z^{n}$ and $\lambda^{n}_{kl}$ it is readily checked  
   that:
\begin{eqnarray}\label{eq8.5i}
\mE
(| {z}^{n}_{t} |^2 )
&=&    \sum_{k,l=0}^{ \lfloor  {nt}{ } \rfloor  }  \lambda^{n}_{kl} 
.
\end{eqnarray}
Therefore, we can write
\begin{eqnarray*}
\frac{1}{n} \mE(|z^{n}_{t}|^{2}) &=&t \frac{\lfloor  {nt}{ } \rfloor }{nt} \frac{1}{\lfloor  {nt}{ } \rfloor } \sum_{k,l=0}^{ \lfloor  {nt}{ } \rfloor  }  \lambda^{n}_{kl} .
\end{eqnarray*}
Sending $n\to \infty$ we obtain the first point  in \eref{eq8.6} thanks to relation \eref{e8.6i}.  
  In the same way, we can show the convergence of $
 n^{ -1} \mE[ {z}^{n}_{t}  \tilde{z}^{n}_{t}  ]   
 $. 
The proof is complete.
 \end{proof}

   In order to get our central limit theorem for the process $z^{n}$, we will apply a corollary of  the fourth moment theorem. This relies on Malliavin calculus tools, for which we first introduce some basic notations. 
   
  \begin{notation}\label{notation8.4}
We define  the Hilbert space $\ch$   as the completion of indicator functions with respect to the inner product $\langle \mathbf{1}_{[s,t]}, \mathbf{1}_{[u,v]} \rangle_{\ch} = \mE(X^{1,1}_{st}X^{1,1}_{uv})$, where $s,t,u,v \in [0,1]$. 
Denote by    $\mathfrak{H}  := \{ h: h(\cdot, i) \in \ch,\,\, i=1,2  \}$   the Hilbert space defined by the following inner product:    
\begin{eqnarray}\label{eq8.12}
\langle h, \bar{h} \rangle_{\mathfrak{H}} = \langle h(\cdot, 1), \bar{h}(\cdot, 1) \rangle_{\ch} + \langle h(\cdot, 2), \bar{h}(\cdot, 2) \rangle_{\ch} .
\end{eqnarray}
 Then, it is readily checked that the Gaussian family $\{ W(h)  = \int h(\cdot, 1) \delta X^{1,1} +\int h(\cdot, 2) \delta  X^{1,2}: h   \in \mathfrak{H}\}$  is an isonormal Gaussian process, where we recall that $(X^{1,1}, X^{1,2})$ is our couple of independent Gaussian process and where $\int f \delta X^{1,1}$ stands for the Wiener integral.  
The random variable $W(h)$ is   called the (first-order) Wiener integral of $h$ with respect to $(X^{1,1}, X^{1,2})$ and is also denoted by $I_{1}(h)  $. 

The operator $I_{1}$ can   be generalized to $\mathfrak{H}^{\otimes k}$. Indeed, for $h= \sum_{j=1}^{n} f_{j} \otimes g_{j}$, where $f_{j} \in \mathfrak{H}$ and $g_{j} \in \mathfrak{H}^{\otimes (k-1)}  $, we set
$I_{1}  (h) = \sum_{j=1}^{n} I_{1} (f_{j})   g_{j} $. Since vectors in the form of $h$ are dense in $\mathfrak{H}^{\otimes k}$, we see that $I_{1}$ can be extended to a bounded operator from $ \mathfrak{H}^{\otimes k}$   into $L^{2}(\Omega , \mathfrak{H}^{\otimes (k-1)})$.  The reader is referred to Page 35 in \cite{NP} for details on this construction. 
 
 Denote by $I_{k}$ the $k$th iteration of the integration operator $I_{1}$, namely, $I_{k} = I_{1}\circ\cdots \circ I_{1}$. 
  For $h \in \mathfrak{H}^{\otimes q}$, $I_{q}(h)$ is called the $q$th-order Wiener integral of $h$. 
   \end{notation}
   
   \begin{example}\label{ex5.7}
   Since $X^{1,1}$ and $X^{1,2}$ are independent, for $t\in [0,1]$ the random variable $z^{n}_{t}$ can be represented as a 2nd-order Wiener integral. Indeed, define   $\phi^{n}  \in \mathfrak{H}^{\otimes 2}$ as follows:
   \begin{eqnarray}\label{eq8.10}
\begin{cases}
\phi^{n}  ( (u,2), (s , 1) )  =&  n^{2\nu} \sum\limits_{k=0}^{ \lfloor  {nt}{ } \rfloor  } \mathbf{1}_{t_{k} \leq u \leq s \leq t_{k+1}}
\\
\phi^{n}  ( (u,i), (s , j) )  =&0 \quad \text{for} \quad   (i,j) \neq (2,1)
\end{cases}.
\end{eqnarray}
We also 
      denote 
   by   $\tilde{\phi}^{n}$ the symmetrization of $\phi^{n}$, that is, 
   \begin{eqnarray}\label{eq8.11}
\tilde{\phi}^{n} ( (u,i), (s , j) ) &=& \frac12 \left( \phi^{n} ( (u,i), (s , j) ) + \phi^{n} ( (s , j), (u,i)  ) \right).
\end{eqnarray}
    Then it is easily checked (see e.g. \cite{FU}  and Page 23 in \cite{N}) that 
   \begin{eqnarray}\label{eq8.14}
 {z}^{n}_{t} = I_{2}(\phi^{n} ) = I_{2}(\tilde{\phi}^{n}).
\end{eqnarray}
   \end{example}


Now that we have expressed $ {z}^{n}_{t}$ as a multiple Wiener integral, we can use the $4$th moment theorem in order to study its limiting law. We thus    recall  the following result borrowed from Theorem 5.2.7 in \cite{NP}:
   \begin{prop}\label{prop8.4}
Fix $q\ge 1$.
Let $\{z^{n} = I_{q}(f_{n})=I_{q}(\tilde{f}_{n}); \,n\geq 1\}$  be a sequence of centered random variables belonging to the $q $th  chaos of $X^{1}=(X^{1,1}, X^{1,2})$, where $\tilde{f}_{n}$ denotes the symmetrization of $f_{n} $ in $ \mathfrak{H}^{\otimes q} $. Assume that 
\begin{equation*} 
\lim_{n\to\infty}\mE[|z^{n}|^2]=1.
\end{equation*}
Then $z^{n} $ converges in distribution to a centered Gaussian random variable if and only if  the following condition is met:
\begin{equation*}  
\lim_{n\to\infty} \| \tilde{f}_{n} \otimes_{r} \tilde{f}_{n}  \|_{\mathfrak{H}^{\otimes (2q - 2r)}} = 0, \quad \text{ for all } r=1,\dots, q-1.
\end{equation*}
The reader is referred to \cite[Appendix B.4]{NP} for the definition of the contraction $ \tilde{f}_{n} \otimes_{r} \tilde{f}_{n} $. 
\end{prop}
  In view of Proposition  \ref{prop8.4},  Lemma \ref{lem8.1} and Example \ref{ex5.7}, we are reduced to the analysis of  the contraction $\| \tilde{\phi}   \otimes_{1} \tilde{\phi}    \|_{\mathcal{H}^{\otimes 2}}$ in order to get our central limit theorem for $z^{n}$, where $\tilde{\phi}^{n}$ is defined by \eref{eq8.11}. This is what is done in the next lemma.

   \begin{lemma}\label{lem8.2}
   Let the assumptions of Lemma \ref{lem8.1} prevail, and consider $ {z}^{n}_{t}= I_{2}(\tilde{\phi}^{n})$ defined by \eref{eq8.14}.  Then we have the convergence
   \begin{eqnarray}\label{e3.11}
\lim_{n\rightarrow \infty} n^{ -2}  \| \tilde{\phi}^{n}   \otimes_{1} \tilde{\phi}^{n}    \|_{\mathfrak{H}^{\otimes 2}}^{2} =0 .
\end{eqnarray} 
   \end{lemma}
  \begin{proof} 
 We  will divide the proof in several steps. We   denote $e = \| \tilde{\phi}^{n}   \otimes_{1} \tilde{\phi}^{n}    \|_{\mathfrak{H}^{\otimes 2}}^{2} $.    

  \noindent
{\it Step 1: An expression for   $ e $.}
Owing to   relation~\eref{eq8.12} for the inner product in $\mathfrak{H}$, we have  
    \begin{eqnarray}\label{e8.10}
\tilde{\phi}^{n}   \otimes_{1} \tilde{\phi}^{n}  &=& \varphi_{1}^{n}  + \varphi_{2}^{n} ,
\end{eqnarray}
where
\begin{eqnarray*}
 \varphi_{1}^{n}((c,2),(d, 2)) =
     \tilde{\phi}^{n} ((c,2), (a,1)) \otimes_{1}  \tilde{\phi}^{n} ((d,2), (a,1))   
\end{eqnarray*}
and
    \begin{eqnarray*}
  \varphi_{2}^{n}((c,1),(d,1))
  =     \tilde{\phi}^{n} ((c,1), (a,2)) \otimes_{1}   \tilde{\phi} ((d,1), (a,2))   
.
\end{eqnarray*}
Here the letter  $a$  designates the pairing for our inner product  in $\ch$. 
Moreover, owing to  the definition \eref{eq8.10} of $\phi^{n}$ and \eref{eq8.11} of $\tilde{\phi}^{n}$ one can check that:
\begin{eqnarray}\label{e8.14}
 \varphi_{1}^{n}( (c,2),(d, 2)) 
=\frac14        {\phi}^{n} ((c,2), (a,1)) \otimes_{1}    {\phi}^{n} ((d, 2), (a,1))    
\end{eqnarray}
\begin{eqnarray*} 
  \varphi_{2}((c,1),(d,1) ) = 
   \frac14    {\phi}^{n} ((a,2), (c,1)) \otimes_{1}    {\phi}^{n} ((a,2), (d,1))   .
\end{eqnarray*}
Taking the operation $ \|    \cdot  \|_{\mathcal{H}^{\otimes 2}}^{2}  $ on both sides of \eref{e8.10} and taking  into account the expressions of $\varphi_{1}$ and $\varphi_{2}$ we obtain
\begin{eqnarray}\label{e8.11i}
e &=& \frac{1}{16} \| \varphi_{1}^{n} ( (c, 2), (d, 2) ) \|^{2}_{\ch^{\otimes 2}} + \frac{1}{16} \| \varphi_{2}^{n} ( (c, 1), (d, 1) ) \|^{2}_{\ch^{\otimes 2}}
\nonumber
\\
&=& 
\frac{1}{8} \| \varphi_{1}^{n} ( (c, 2), (d, 2) ) \|^{2}_{\ch^{\otimes 2}}.
\end{eqnarray}
Here the letters  $c$, $d$ designate the pairing for our inner products in   $\ch^{\otimes2}$. 

    We now decompose the term $\varphi_{1}^{n}$ in \eref{e8.11i}. To this aim, denote
     $\phi_{k}^{n}(u,s) = n^{2\nu} \mathbf{1}_{  t_{k}\leq u \leq s \leq t_{k+1} }   $. Then by the definition \eref{eq8.10} of $\phi^{n}$ we have
    \begin{eqnarray*}
\phi^{n}  ( (u,2), (s , 1) ) =     \sum_{k=0}^{ \lfloor  {nt}{ } \rfloor  } \phi_{k}^{n}(u,s).
\end{eqnarray*}
Plugging this formula into the expression \eref{e8.14} of   $\varphi_{1}^{n}$ we obtain:
   \begin{eqnarray}\label{e8.13}
\varphi_{1}^{n} ( (c,2),(d, 2) ) =   
\frac14   \sum_{k , k'=0}^{ \lfloor  {nt}{ }  \rfloor  }
         \phi_{k }^{n}(c, a) \otimes_{1}    \phi_{k' }^{n}(d, a)   ,
\end{eqnarray}
where we recall that $a$ is the letter used for the pairing in $\ch$. 
Next we compute the $\ch^{\otimes 2}$-norm of $\varphi_{1}^{n}$ thanks to relation \eref{e8.13}. Taking into account formula~\eref{e8.11i}, this yields:
   \begin{eqnarray} \label{e8.18}
 e &=&   \frac{1}{128}       \sum_{ ( k_{1},k_{2},k_{3},k_{4}) \in M } c(k_{1},k_{2},k_{3},k_{4})
,
   \end{eqnarray}
   where we denote
   \begin{eqnarray}\label{e8.19}
c(k_{1},k_{2},k_{3},k_{4}) &=& \left\langle
   \phi_{k_{1}}^{n} ( d, a ) \otimes_{1}  \phi_{k_{4}}^{n} (c,a)     , 
     \phi_{k_{3}}^{n} (d,b) \otimes_{1}  \phi_{k_{2}}^{n} (c,b)     
 \right\rangle_{\mathcal{H}^{\otimes 2}} ,
\end{eqnarray}
and where $M $ is the set of indices $M  = \{ 0,1,\dots,    \left\lfloor  {nt}{ } \right\rfloor \}^{4}$.
   
  \noindent
{\it Step 2: Decomposition of $e $.}
We will now split the summation in \eref{e8.18}  according to convenient subsets of $M $. We thus introduce an additional notation, valid for all subsets   $M' \subset M  $:
\begin{eqnarray} \label{e8.8i}
 e (M')&=&   \frac{1}{128}    \sum_{(k_{1},k_{2},k_{3},k_{4}) \in M'} c(k_{1},k_{2},k_{3},k_{4})  .
   \end{eqnarray}
   Next for $i=0,\dots, 4$ we define the following subsets of indices: 
   \begin{eqnarray*}
M_{i} = \left\{ (k_{1},k_{2},k_{3},k_{4}  ) \in M : \text{exactly }  i  \text{ of the pairs }   ( {j},  {j'}) \in \cp \text{ satisfy } |k_{j} - k_{j'}|\leq 2  \right\},
\end{eqnarray*}
where we denote
$
\cp = \{ (1,3), (1,4), (2,3), (2,4) \}.
$
Then we can decompose $e$ as:
\begin{eqnarray}\label{e8.21}
e = \sum_{i=0}^{4} e (M_{i}). 
\end{eqnarray}
 So to prove \eref{e3.11}, we are now reduced  to show that $n^{ -2} e  (M_{i}) $ tends to $ 0$ for $i=0,\dots,4$.

 \noindent
{\it Step 3: Computations for $e(M_{1})$.}
Let us approximate the   functions $n^{2\nu}\mathbf{1}_{t_{k}\leq u<s\leq t_{k+1}}$ in the definition of $\phi_{k}^{n}$ by sums of indicators of rectangles. Namely, for $k \leq \left\lfloor {nt}{ } \right\rfloor$  we set
\begin{eqnarray}
\phi^{n,\ell}_{k} (u,s) &=& n^{2\nu}\sum_{i=0}^{\ell-1} \mathbf{1}_{[ t_{k}   , t_{k}+ \frac{i }{n\ell}]} (u) 
\times   \mathbf{1}_{[t_{k}+\frac{i}{n\ell}, t_{k}+\frac{i+1}{n\ell}]} (s)     
\label{e8.20}
\\
 &=& n^{2\nu}\sum_{i=0}^{\ell-1}  \mathbf{1}_{[t_{k}+\frac{i}{n\ell}, t_{k}+\frac{i+1}{n\ell}]} (u) \times   \mathbf{1}_{[   t_{k}+ \frac{i+1 }{n\ell}, t_{k+1}]} (s) .
 \label{e8.8}
\end{eqnarray}
Note that we have $\phi_{k}^{n} \in \mathcal{H}^{\otimes 2}$ and the following approximation result holds true: 
\begin{eqnarray}\label{e8.4}
\lim_{\ell\to \infty} \| \phi_{k}^{n}-\phi^{n,\ell}_{k} \|_{\mathcal{H}^{\otimes 2}} = 0. 
\end{eqnarray}

We now compute $e(M_{11})$ for a given subset $M_{11} \subset M_{1}$. Namely, 
 denote 
 \begin{eqnarray}\label{e5.27}
I_{ij} = \{ (k_{1}, k_{2}, k_{3},k_{4}) \in M : |k_{i}-k_{j}|>2\}
\end{eqnarray}
  and set $M_{11} = I_{13} \cap I_{14} \cap  I_{23} \cap I_{24}^{c}$. It is clear that $M_{11}  \subset M_{1}$. 
 Now consider $( k_{1} , k_{2}, k_{3},k_{4}) \in M_{11}$. By \eref{e8.4}, the expression \eref{e8.20} of $\phi^{n,\ell}_{k}$ and \eref{e8.19} we have
\begin{eqnarray*}
 \left|
 c(k_{1},k_{2},k_{3},k_{4}) 
 \right|
  &=&  \lim_{\ell\rightarrow \infty} \lim_{\ell'\rightarrow \infty}    
 \left|
 \left\langle
  \phi_{k_{1}}^{n} ( d, a ) \otimes_{1}
  \phi^{n,\ell}_{k_{4}} (c,a)
,
 \phi_{k_{3}}^{n} (d,b) \otimes_{1} \phi^{n,\ell'}_{k_{2}} (c,b) 
     \right\rangle_{\mathcal{H}^{\otimes 2}} 
 \right|
 \\
 &\leq & n^{2\nu}  \lim_{\ell\rightarrow \infty}  \lim_{\ell'\rightarrow \infty} \sum_{i =1}^{\ell-1}\sum_{i' =1}^{\ell'-1}
| \tilde{c}(i,i')  |  
\left|
\left\langle
 \mathbf{1}_{[  t_{ k_{4}} , t_{k_{4}}+\frac{i'}{n\ell}]} (c) , \mathbf{1}_{[ t_{k_{2}} , t_{k_{2}}+  \frac{i }{n\ell'}]} (c)   \right\rangle_{\mathcal{H}}
\right|
 \,,
\end{eqnarray*}
where 
\begin{eqnarray*}
\tilde{c}(i,i') &=&   
n^{2\nu}  \left\langle
     \phi_{k_{1}}^{n} ( d, a ) \otimes_{1}  \mathbf{1}_{[t_{k_{4}}+\frac{i'}{n\ell}, t_{k_{4}}+\frac{i'+1}{n\ell}]} (a)  ,
   \mathbf{1}_{[t_{k_{2}}+ \frac{i}{n\ell'}, t_{k_{2}}+  \frac{i+1}{n\ell'}]} (b)     \otimes_{1} 
    \phi_{k_{3}}^{n} (d,b)      
     \right\rangle_{\mathcal{H} } 
  .
\end{eqnarray*}
 Notice that, thanks to Cauchy-Schwarz inequality, for all $i,i' \leq l-1$ we have 
 \begin{eqnarray}\label{eq8.24}
\left|
\left\langle
n^{2\nu}
 \mathbf{1}_{[   t_{k_{4} }, t_{k_{4}}+\frac{i'}{n\ell}]}   , \mathbf{1}_{[ t_{k_{2}} , t_{k_{2}}+  \frac{i }{n\ell'}]}     \right\rangle_{\mathcal{H}}
\right| \leq 1.
\end{eqnarray}
  We thus get
\begin{eqnarray*}
 \left|
 c(k_{1},k_{2},k_{3},k_{4}) 
 \right| &\leq&    \lim_{\ell\rightarrow \infty}   \lim_{\ell'\rightarrow \infty} \sum_{i=1}^{\ell-1}\sum_{i'=1}^{\ell'-1}
\left|
 \tilde{c}(i,i') 
 \right|.
\end{eqnarray*}

In order to evaluate $\tilde{c} (i,i')$, observe that, thanks to    Hypothesis \ref{hyp7.1} (i) and the fact that $(k_{1},k_{2},k_{3},k_{4}) \in M_{11}$,  the quantities   $  \tilde{c}(i,i')   $ have the same sign for all $i,i'=1,\dots, \ell-1$. Denoting
$\psi_{k} = \mathbf{1}_{[k,k+1]}$
and $\phi_{k} = \mathbf{1}_{k\leq u\leq s\leq k+1}$,
  we thus get
\begin{eqnarray}  
\left|
 c(k_{1},k_{2},k_{3},k_{4}) 
 \right|
 &\leq&
    \lim_{\ell\rightarrow \infty}   \lim_{\ell'\rightarrow \infty} \left|\sum_{i=1}^{\ell-1}\sum_{i'=1}^{\ell'-1}
 \tilde{c}(i,i') 
 \right|
 \nonumber
  \\
  &= &      
\left|
  \left\langle
   \phi_{k_{1}} (d  ,  a )\otimes_{1} \psi_{k_{4}} (a) 
    ,
   \psi_{k_{2}} (b)     \otimes_{1} 
     \phi_{k_{3}}  (d  , b)      
      \right\rangle_{\mathcal{H} } 
 \right| 
 \nonumber
 \\
  &\leq &      
\left|
 \left\langle
  \left\langle 
   \psi_{k_{1}} (d    )   \psi_{k_{1}} (a), \psi_{k_{4}} (a) 
    \right\rangle_{\mathcal{H} },
 \left\langle 
  \psi_{k_{2}} (b)     , 
     \psi_{k_{3}}  (d  )    \psi_{k_{3}}(b)   
      \right\rangle_{\mathcal{H} }
 \right\rangle_{\mathcal{H} } 
 \right| 
 \label{e8.24}
 \\
 &=&       
\left| 
  \left\langle
   \psi_{k_{4}}    ,
   \psi_{k_{1}}      
 \right\rangle_{\mathcal{H} } \right|
  \cdot 
   \left|   \left\langle
    \psi_{k_{3}}       ,   \psi_{k_{1}}     
     \right\rangle_{\mathcal{H} } \right| \cdot
 \left| 
 \left\langle 
  \psi_{k_{2}}        , 
    \psi_{k_{3}}     
      \right\rangle_{\mathcal{H} }
  \right| .
  \nonumber
  \end{eqnarray}
In \eref{e8.24}, notice that we can replace the simplex indicator $\phi_{k}(u,s)   $ by $\psi_{k}^{n}  \otimes \psi_{k}^{n} (s,u) =  \mathbf{1}_{[ {k}, {k+1}]^{2}} (s,u)$ due to the fact that  each of the three  pairs $ (k_{1},k_{4}) $, $(k_{2}, k_{3})$, and $(k_{1}, k_{3})$ are disjoint and also Hypothesis \ref{hyp7.1} (i).  
Furthermore, applying \eref{e9.1}    to relation \eref{e8.24} with $u=1$ we obtain
 \begin{eqnarray}\label{eq8.26}
\left|
 c(k_{1},k_{2},k_{3},k_{4}) 
 \right| &\leq & K   |k_{1} -k_{4}|^{-\theta}  |k_{1} -k_{3}|^{-\theta}  |k_{2} -k_{3}|^{-\theta},
\end{eqnarray}
where $\frac12<\theta\leq 2-2\nu$. 
Applying this estimate to \eref{e8.8i} with $M'=M_{11}$ we obtain 
\begin{eqnarray*}
e (M_{11}) &\leq&  K    \sum_{( k_{1} , k_{2}, k_{3},k_{4}) \in M_{11}}  |k_{1} -k_{4}|^{-\theta}  |k_{1} -k_{3}|^{-\theta}  |k_{2} -k_{3}|^{-\theta}.
\end{eqnarray*}
It is now easy to show from this estimate that
\begin{eqnarray}\label{e8.11}
n^{ -2} e  (M_{11})    \leq Kn^{-\theta} \rightarrow 0 \quad \text{as} \quad n\rightarrow \infty,
\end{eqnarray}
which is our desired estimate for $e(M_{11})$. 

In order to conclude for the term $e(M_{1})$, 
set 
\begin{eqnarray*}
M_{12}=  I_{13} \cap I_{14} \cap  I_{23}^{c} \cap I_{24}   , \quad M_{13} 
=    I_{13} \cap I_{14}^{c} \cap  I_{23} \cap I_{24} 
, \quad M_{14} =  I_{13}^{c} \cap I_{14} \cap  I_{23} \cap I_{24} .
\end{eqnarray*}
 Similarly to what we have done above,   we can show that the convergence \eref{e8.11} still  holds when $M_{11}$ is replaced by $M_{1i}$, for $i=2,3,4$.   Noticing that $M_{1} = \bigcup_{i=1}^{4} M_{1i}$, we conclude that 
\begin{eqnarray*}
\lim_{n\to \infty} n^{ -2} e  (M_{1})  &=&0.
\end{eqnarray*}

\noindent
{\it Step 4: Computations for $e(M_{2})$-Part 1.} 
As in the case of $M_{1}$, we will decompose $e(M_{2}) $ in several terms and analyze them individually. To start with, set
  $M_{21} = I_{13}^{c} \cap I_{14} \cap  I_{23} \cap I_{24}^{c}$, where we recall that $I_{ij}$ is defined by \eref{e5.27}. Along the same lines as  the proof of Step 3,  for $( k_{1} , k_{2}, k_{3},k_{4}) \in M_{21}$ we can show that
\begin{eqnarray}
 \left|
 c(k_{1},k_{2},k_{3},k_{4}) 
 \right|
 \leq K |k_{2} - k_{3}|^{-\theta} |k_{1} - k_{4}|^{-\theta}
  \leq  K |k_{2} - k_{3}|^{-2\theta},
 \label{e8.26}
\end{eqnarray}
where the last relation stems from the fact that $|k_{1}-k_{3}|\leq 2$ and $|k_{2}-k_{4}|\leq 2$.
Let us   highlight the following difference between the term $e(M_{21})$ and $e(M_{11})$: in order to handle $e(M_{21})$, since now both $|k_{1} - k_{3}|$ and $|k_{2} - k_{4}|$ are smaller than $3$, we need  to apply the approximation \eref{e8.20} for each of  $\phi_{k_{1}}^{n} $, $\phi_{k_{2}}^{n} $, $\phi_{k_{3}}^{n} $ and $\phi_{k_{4}}^{n}$.  
Then applying relation \eref{e8.26} to~\eref{e8.8i} with $M=M_{21}$  and invoking the fact that $\# M_{21} = O(n^{2})$ and $\sum_{j\geq 1} |j|^{-2\theta}<\infty$,  we obtain 
\begin{eqnarray*}
n^{ -2} e (M_{21}) \leq K n^{-1}  \rightarrow 0 \quad \text{as} \quad n \rightarrow \infty.
\end{eqnarray*}
In a similar way we can show that this  convergence still holds for $M_{22}:=I_{13} \cap I_{14}^{c} \cap  I_{23}^{c} \cap I_{24} $.

\noindent
{\it Step 5: Computations for $e(M_{2})$-Part 2.} We now deal with a slightly different kind of term involved in $e(M_{2})$. Namely,  set $M_{23} =  I_{13}^{c} \cap I_{14}  \cap  I_{23}^{c} \cap I_{24}$ and  take  $( k_{1} , k_{2}, k_{3},k_{4}) \in M_{23}$. Owing to relation \eref{e8.4} we have
\begin{eqnarray*}
 \left|
 c(k_{1},k_{2},k_{3},k_{4}) 
 \right|
 &=&  \lim_{\ell\rightarrow \infty} \lim_{\ell'\rightarrow \infty}    
\left|
 \left\langle
  \phi^{n,\ell}_{k_{1}} ( d, a )\otimes_{1}
  \phi_{k_{4}}^{n} (c,a)
 ,
  \phi^{n,\ell'}_{k_{2}} (c,b) \otimes_{1}    \phi_{k_{3}}^{n} (d,b)    
 \right\rangle_{\mathcal{H}^{\otimes 2}} 
 \right|.
 \end{eqnarray*}
 We now use expression \eref{e8.20} for $\phi_{k_{1}}^{n,\ell}$ and expression \eref{e8.8} for $\phi_{k_{2}}^{n,\ell'}$. This yields:
 \begin{equation}\label{e8.27}
\left|
 c(k_{1},k_{2},k_{3},k_{4}) 
 \right|  \leq  n^{2\nu} \lim_{\ell\rightarrow \infty}  \lim_{\ell'\rightarrow \infty}    
 \sum_{i=0}^{\ell-1}  \sum_{i' =0}^{\ell'-1}  
   |
 \hat{c}(i,i')
  |
      \cdot 
\left|
  \left\langle 
     \mathbf{1}_{[ t_{k_{1}}   , t_{k_{1}}+ \frac{i '}{n\ell}]} (d)  \mathbf{1}_{[   t_{k_{2}}+ \frac{i+1 }{n\ell'}, t_{k_{2}+1}]} (b)  ,  
       \phi_{k_{3}}^{n} (d,b) 
        \right\rangle_{\mathcal{H}^{\otimes 2}}
        \right|
 ,
\end{equation}
where
\begin{eqnarray*}
\hat{c}(i,i') &=& n^{2\nu} \Big\langle
 \Big\langle 
 \mathbf{1}_{[t_{k_{1}}+\frac{i'}{n\ell}, t_{k_{1}}+\frac{i'+1}{n\ell}]} (a)     
     ,  
        \phi_{k_{4}}^{n} (c,a)  \Big\rangle_{\mathcal{H} }, 
         \mathbf{1}_{[t_{k_{2}}+\frac{i}{n\ell'}, t_{k_{2}} +\frac{i+1}{n\ell'}]} (c)
\Big\rangle_{\mathcal{H} } .
\end{eqnarray*}
We now observe two facts: 

(i) Since $|k_{1} - k_{4}| >2$ and $|k_{2} - k_{4}|>2$, and resorting to Hypothesis \ref{hyp7.1} (i), we have  $\hat{c}(i,i') \geq 0$ for all $i=1,\dots, \ell-1$, $i'=1,\dots, \ell'-1$.

(ii) Similarly to \eref{eq8.24}, we can apply 
  Cauchy-Schwarz inequality in order to get
  \begin{eqnarray*}
\Big|
  \Big\langle 
   n^{2\nu}  \mathbf{1}_{[ t_{k_{1}}   , t_{k_{1}}+ \frac{i '}{n\ell}]} (d)  \mathbf{1}_{[   t_{k_{2}}+ \frac{i+1 }{n\ell}, t_{k_{2}+1}]} (b)  ,  
       \phi_{k_{3}}^{n} (d,b) 
        \Big\rangle_{\mathcal{H}^{\otimes 2}}
        \Big| &\leq& 1.
\end{eqnarray*}

Plugging this information into \eref{e8.27} we obtain  
\begin{eqnarray*}
  \left|
 c(k_{1},k_{2},k_{3},k_{4}) 
 \right|
      \leq  \lim_{\ell\rightarrow \infty} \lim_{\ell'\rightarrow \infty}    
  \Big| \sum_{i =0}^{\ell-1}\sum_{i' =0}^{\ell'-1} 
 \hat{c}(i,i')
  \Big|
 =     \Big|
 \Big\langle
 \left\langle 
   \phi_{k_{1}} (a) 
     ,  
        \phi_{k_{4}} (c,a)  \right\rangle_{\mathcal{H} },
  \phi_{k_{2}} (c)     
 \Big\rangle_{\mathcal{H} } 
 \Big|.
\end{eqnarray*}
We can now proceed
by enlarging the simplex $\{ t_{k}\leq u\leq s \leq t_{k+1} \}$ to a rectangle $[t_{k},t_{k+1}]^{2}$ as in   \eref{e8.24},  and using the bound \eref{e9.1} as in  \eref{eq8.26}. We end up with: 
 \begin{eqnarray*}
  \left|
 c(k_{1},k_{2},k_{3},k_{4}) 
 \right| & \leq& K |k_{2} - k_{4}|^{-\theta}   |k_{1} - k_{4}|^{-\theta} .
\end{eqnarray*}
It is now easy to show by this estimate,   expression \eref{e8.8i}, and the fact that $|k_{i} - k_{j}| \leq 4 $ for $i,j \in \{1,2,3\}$ that 
\begin{eqnarray*}
n^{ -2} e  (M_{23}) \leq Kn^{-1} \rightarrow 0  \quad \text{as} \quad n \rightarrow \infty.
\end{eqnarray*}

We can easily extend the considerations above in order to get a similar convergence for 
$e(M_{2i})$, $i=4,5,6$, where
 \begin{eqnarray*}
M_{24} = I_{13}^{c} \cap I_{14}^{c} \cap  I_{23} \cap I_{24} 
, \quad M_{25} =   I_{13} \cap I_{14}^{c}  \cap  I_{23} \cap I_{24}^{c} , \quad M_{26} =  I_{13}  \cap I_{14}  \cap  I_{23}^{c} \cap I_{24}^{c} .
\end{eqnarray*}

In summary of Step 4 and 5 and noticing that   $M_{2} = \bigcup_{i=1}^{6} M_{2i}$,     we obtain the convergence:
 \begin{eqnarray*}
\lim_{n\to\infty} n^{ -2} e  (  M_{2 })= \sum_{i=1}^{6} \lim_{n\to\infty} n^{ -2} e  ( M_{2i}) =0.
\end{eqnarray*}

\noindent
{\it Step 6: Computations for $e(M_{0})$.} 
Take now $( k_{1} , k_{2}, k_{3},k_{4}) \in  M_{0}$. Then as before, by assumption \eref{e9.1} we obtain
\begin{eqnarray*}
\left|
 c(k_{1},k_{2},k_{3},k_{4}) 
 \right| &\leq & K     |k_{1} -k_{4}|^{-\theta}  |k_{1} -k_{3}|^{-\theta}  |k_{2} -k_{3}|^{-\theta} |k_{2} -k_{4}|^{-\theta}.
\end{eqnarray*}
It is   easy to show from this estimate and expression \eref{e8.8i} that
\begin{eqnarray*}
n^{ -2} e  (M_{0})    \leq Kn^{1-2\theta} \rightarrow 0 \quad \text{as} \quad n\rightarrow \infty.
\end{eqnarray*}

\noindent
{\it Step 7: Computations for $e(M_{3}\cup M_{4})$.} Finally, we consider the case when $( k_{1} , k_{2}, k_{3},k_{4}) \in  M_{3} \cup M_{4}$.  In order to get an estimate for $e(M_{3} \cup M_{4})$, we first note that 
$
\# ( M_{3} \cup M_{4})    \leq 19n$. 
On the other hand, a simple application of Cauchy-Schwarz inequality yields the relation 
$
|  c(k_{1}, k_{2}, k_{3}, k_{4}) |  \leq 1$ 
  for all $(k_{1}, k_{2}, k_{3}, k_{4}) \in M_{3} \cup M_{4}$.
Therefore, we obtain 
\begin{eqnarray*}
n^{ -2} e  (M_{3}\cup M_{4}) \leq Kn^{-1} \rightarrow 0  \quad \text{as} \quad n \rightarrow \infty.
\end{eqnarray*}

Gathering the estimates we have obtained in Steps 3 to 7 and recalling the decomposition~\eref{e8.21}, the proof of our claim \eref{e3.11} is now complete. 
 \end{proof}

\

   \noindent \textit{Proof of Proposition \ref{prop8.3}.}\quad 
   According to the fourth moment method applied to the second chaos $\mathfrak{H}^{\otimes 2}$ introduced in Notation \ref{notation8.4}, we are reduced to show the following facts:
   
   (i) For any $L\geq 1$, the covariance matrix of 
   \begin{eqnarray*}
(n^{ -\frac12} ( h^{n}_{0r_{1}}, \dots, h^{n}_{0r_{L}}), X^{1}_{r_{1}}, \dots, X^{1}_{r_{L}} )
\end{eqnarray*}
 converges to that of 
    \begin{eqnarray*}
( (W_{r_{1}},\dots, W_{r_{L}}), X^{1}_{r_{1}}, \dots, X^{1}_{r_{L}}).
\end{eqnarray*}

   (ii) The following weak convergence holds true for all $i,j=1,\dots, m$, $l=1,\dots, L$:
   \begin{eqnarray*}
 n^{ -\frac12} h^{n,ij}_{0 r_{l}} \Rightarrow  W^{ij}_{r_{l}}.
   \end{eqnarray*}

\noindent Note that we have recalled the fourth moment method for $1$-d sequences of random variables in Proposition \ref{prop8.4}.
We refer to \cite{NP} for more details about the fourth moment method for random vectors in a fixed chaos, and we now focus on the proof of item (i) and (ii). 
   
  The weak convergence (ii) of $h^{n,ij}_{0r_{l}}$ for $i\neq j$ follows immediately from Lemma \ref{lem8.2} and Proposition~\ref{prop8.4}. In the case when $i=j$,   (ii) follows from the classical results in \cite{BM}, see also Section 7.4 in \cite{NP}. In the following, we show the convergence of the covariance $\mE(h^{n,ij}_{0r_{l}} h^{n,i'j'}_{0r_{l'}})$. 

We start by studying $\mE ( h^{n,ij}_{0r_{l}}h^{n,i'j'}_{0r_{l'}} )$ when $r_{l}=r_{l'}$. In this case, whenever
 $(i,j) = (i',j')$ or $(i,j) = (j',i')$, the convergence of   $\mE(h^{n,ij}_{0r_{l}} h^{n,i'j'}_{0r_{l}})$ follows from Lemma \ref{lem8.1}. In the case $(i,j) \neq (i',j')$ and $(i,j) \neq (j',i')$,   the covariance $\mE(h^{n,ij}_{0r_{l}} h^{n,i'j'}_{0r_{l}}) $ is simply equal to $0$. 
 
 Let us now assume that   $r_{l} > r_{l'}$.  Since $\mE(h^{n,ij}_{0r_{l}} h^{n,i'j'}_{0r_{l'}})  =
 \frac12 (\mE(h^{n,ij}_{0r_{l}} h^{n,i'j'}_{0r_{l'}}) +\mE(h^{n,ij}_{0r_{l'}} h^{n,i'j'}_{0r_{l}}) )$,   we   can reduce this case to the previous study by invoking the following identity:
\begin{eqnarray}\label{e8.33}
 \mE(h^{n,ij}_{0r_{l}} h^{n,i'j'}_{0r_{l'}})  =
 \frac12 \left(  \mE[ h^{n,ij}_{0r_{l}}h^{n,i'j'}_{0r_{l}}] +  \mE[  h^{n,ij}_{0r_{l'}}h^{n,i'j'}_{0r_{l'}} ] -  \mE[\delta h^{n,ij}_{r_{l'}r_{l}}  \delta h^{n,i'j'}_{r_{l'}r_{l}}   ]  \right).
\end{eqnarray}
Then thanks to  Lemma \ref{lem8.1}, the first two terms on the right-hand side of \eref{e8.33} converge  to $   (\lambda\delta_{ii'}\delta_{jj'} + \rho \delta_{ij'}\delta_{ji'} ) r_{l}$ and $   (\lambda\delta_{ii'}\delta_{jj'} + \rho \delta_{ij'}\delta_{ji'} ) r_{l'}$. 
In order to treat the term $\mE[\delta h^{n,ij}_{r_{l'}r_{l}}  \delta h^{n,i'j'}_{r_{l'}r_{l}}   ] $, 
note that 
$  \delta h^{n,ij}_{r_{l'}r_{l}}  \delta h^{n,i'j'}_{r_{l'}r_{l}}     $ is equal to $h^{n,ij}_{\lfloor r_{l} \rfloor - \lfloor r_{l'} \rfloor  }h^{n,i'j'}_{\lfloor r_{l} \rfloor - \lfloor r_{l'} \rfloor  }$ in distribution, where recall that $ \lfloor r_{l} \rfloor $ and $\lfloor r_{l'} \rfloor $ denote respectively the integer part of $r_{l}$ and $r_{l'}$. So by Lemma \ref{lem8.1} the third term converges to $ (\lambda\delta_{ii'}\delta_{jj'} + \rho \delta_{ij'}\delta_{ji'} )(r_{l} - r_{l'})$. 
Summarizing our last considerations, we easily get:
 \begin{eqnarray*}
\lim_{n\to \infty}\frac{1}{n} \mE(h^{ij}_{r_{l}} h^{i'j'}_{r_{l'}}) = (\lambda\delta_{ii'}\delta_{jj'} + \rho \delta_{ij'}\delta_{ji'} ) r_{l'} .
\end{eqnarray*}
  The proof is complete. 
\hfill $\square$

\subsection{Weighted limit theorem}
Let $X$ be the enhanced Gaussian process defined as in Section \ref{section5.1}. With the preparation in the previous subsection, we now
  consider the convergence of the discrete integral $n^{-\frac12}\cj_{0}^{1}(y ;  h^{n})$ with
$h^{n}$ defined in \eref{e7.6}, where $(y,y',\cdots,y^{(\ell-1)})$ is a discrete process controlled by $(X,\al)$.

Let us recall some basic facts about the range of our parameters. First, the covariance $R$ satisfies \eref{e5.1}, and we consider a parameter $\nu=\frac{1}{2\rho}$. Since we assume that $\rho = [1,2)$, we also have $\nu\in (\frac14, \frac12]$. Then the coefficient $\al$ is dictated by the regularity type estimate \eref{eq8.6}, namely $\al=\frac12$. Eventually the order $\ell$ of the controlled process $y$ is such that $\nu\ell+\frac12>1$, which yields $\ell=2$ in our setting.   

We start by giving some uniform bounds on $\cj(X^{1}; h^{n})$. 
\begin{lemma}
The following relations holds true:
\begin{eqnarray}\label{e5.34}
n^{-\frac12} \Big|   \cj_{s}^{t} (X^{1}; h^{n} ) \Big|_{L_{2}} \leq  K (t-s)^{\nu+\frac12},
\quad
\quad
n^{-\frac12} \Big| \sum_{j=0}^{m-1} \cj_{s_{j}}^{s_{j+1}} (X^{1}; h^{n} ) \Big|_{L_{2}}\leq  K m^{-\nu}. 
\end{eqnarray}

\end{lemma}
\begin{proof} 
The estimate \eref{e5.34} is obtained in a similar way as in those in Lemma \ref{lem8.3}.
We just observe that the non diagonal terms of the matrix $h^{n}$ will be handled by approximating the indicator function of the simplex by indicator functions of rectangles, similarly to what we did in the proof of   Lemma \ref{lem8.2}. The details are omitted. 
\end{proof}

\begin{theorem}\label{thm5.11}
Let $X$ and $h^{n}$ be as in Proposition \ref{prop8.3}, with $\nu\in (\frac14, \frac12]$. Let $y$ be a controlled process of order $\ell=2$. 
Then the following convergence holds true:
\begin{eqnarray}\label{e5.34i}
(X, n^{-\frac12}\cj(y; h^{n})) \xrightarrow{f.d.d.} (X, \int y_{r}\otimes dW_{r}) ,
\end{eqnarray}
where $W$ is the Wiener process introduced in Proposition \ref{prop8.3}, and where the integral $\int_{s}^{t} y_{r} \otimes dW_{r}$ has to be understood in the Wiener sense. 

\end{theorem}
\begin{proof}
In order to show the convergence \eref{e5.34i} we invoke   Theorem \ref{thm5.9}. We first note that  
inequality \eref{e3.8i}    holds true thanks to   the first relation in \eref{e5.34}.   Furthermore, the convergence of $(X, n^{-\frac12 } h^{n})$ follows from Proposition \ref{prop8.3}. Finally, relation \eref{e6.1} is a consequence of the second relation in \eref{e5.34}. Therefore, applying Theorem \ref{thm5.9} we obtain the desired convergence \eref{e5.34i}.  
\end{proof}





\begin{thebibliography}{99}

\bibitem{A} Arcones, M. (1994). Limit theorems for nonlinear functionals of a stationary Gaussian sequence of vectors. \textit{Ann. Probab.} \textbf{22} no. 4, 2242-2274. 

 
 \bibitem{BCP} Barndorff-Nielsen, O.; Corcuera, J. and Podolskij, M. (2009). Power variation for Gaussian processes with stationary increments. \textit{Stochastic Proc. Appl.} \textbf{119} 1845-1865. 
 
 
 \bibitem{BNN} Binotto, G.; Nourdin, I. and Nualart, D. (2016). Weak symmetric integrals with respect to the fractional Brownian motion. \textit{arXiv: 1606.04046} 
 
\bibitem{BM} Breuer, P. and Major, P. (1983). Central limit theorems for nonlinear functionals of Gaussian fields. \textit{J. Multivariate Anal.} \textbf{13} no. 3, 425-441. 

\bibitem{BS} Burdzy, K. and Swanson, J. (2010). A change of variable formula with It\^o correction term.  \textit{Ann. Probab.} \textbf{38} no. 5, 1817-1869. 

\bibitem{CHLT}   Cass, T;   Hairer, M.;   Litterer, C. and Tindel, S. (2015). Smoothness of the density for solutions to Gaussian Rough Differential Equations. {\it Ann. Probab.} {\bf 43} , no. 1, 188-239.


\bibitem{CN} Cheridito, P. and Nualart, D. (2005). Stochastic integral of divergence type with respect to fractional Brownian motion with Hurst parameter $H \in (0 , \frac12)$. \textit{Ann. Inst. H. Poincar\'e Probab. Statist.} \textbf{41}, no. 6, 1049-1081. 


\bibitem{CNP} Corcuera, J. M.;     Nualart, D.  and   Podolskij, M. (2014). Asymptotics of weighted random sums.    \textit{Commun. Appl. Ind. Math.} \textbf{6}, no. 1, e-486, 11pp.  

\bibitem{CNW} Corcuera, J. M.;     Nualart, D.  and   Woerner, J. (2006).  Power variation of some integral fractional processes. \textit{Bernoulli} \textbf{12} no. 4, 713-735. 

\bibitem{DR}
Davis, R. and Resnick, S. (1985). Limit theory for moving averages of random variables with regularly varying tail probabilities. \textit{Ann. Probab.} \textbf{13} no. 1, 179-195.

\bibitem{DM} Dobrushin, R. L. and Major, P. (1979). Non-central limit theorems for nonlinear functionals of Gaussian fields. \textit{Z. Wahrsch. Verw. Gebiete} \textbf{50} no. 1, 27-52.

 \bibitem{FU} Ferreiro-Castilla, A. and   Utzet, F. (2011).  L\'evy area for Gaussian processes: A double Wiener-It\^o integral approach.  \textit{Statist.   Probab. Lett.} \textbf{81} no. 9,   1380-1391.

\bibitem{FGGR} Friz, P. K.; Gess, B.; Gulisashvili, A. and Riedel, S. (2016). The Jain-Monrad criterion for rough paths and applications to random Fourier series and non-Markovian H\"ormander theory. \textit{Ann. Probab.} \textbf{44} no.~1, 684-738. 

\bibitem{FH} Friz, P. K. and Hairer, M. (2014). \textit{A course on rough paths: with an introduction to regularity structures.} Springer. 


\bibitem{FV} Friz, P. K. and Victoir, N. B. (2010). \textit{Multidimensional stochastic processes as rough paths: theory and applications} Vol. 120.   Cambridge University Press. 



\bibitem{GN} Gradinaru, M. and Nourdin, I. (2009). Milstein's type schemes for fractional SDEs. \textit{Ann. Inst. Henri Poincar\'e Probab. Stat.} \textbf{45} no. 4, 1085-1098. 

 \bibitem{GNRV} Gradinaru, M.; Nourdin, I.; Russo, F. and Vallois, P. (2005). $m$-order integrals and It\^o's formula for non-semimartingale processes; the case of a fractional Brownian motion with any Hurst index. \textit{Ann. Inst. H. Poincar\'e Probab. Statist.} \textbf{41}, 781-806. 
 
 \bibitem{GRV} Gradinaru, M.; Russo, F. and Vallois, P. (2001). Generalized covariations, local time and Stratonovich It\^o's formula for fractional Brownian motion with Hurst index $H\geq \frac14$. {\it Ann. Probab.} \textbf{31}, 1772-1820.
 
\bibitem{G} Gubinelli, M. (2004). Controlling rough paths. \textit{J. Funct. Anal.} \textbf{216}, no. 1, 86-140.  
 
 \bibitem{HN} Harnett, D. and Nualart, D. (2012). Weak convergence of the Stratonovich integral with respect to a class of Gaussian processes. \textit{Stochastic Process. Appl.} \textbf{122}, no. 10, 3460-3505. 
 
  \bibitem{HN2} Harnett, D. and Nualart, D. (2013). Central limit theorem for a Stratonovich integral with Malliavin calculus. \textit{Ann. Probab.} \textbf{41}, no. 4, 2820-2879. 
 
   \bibitem{HN3} Harnett, D. and Nualart, D. (2013). On Simpson's rule and fractional Brownian motion with $H=1/10$. \textit{J. Theoret. Probab.} \textbf{28}, no. 4, 1651-1688.

\bibitem{H} Hu, Y. (2017). \textit{Analysis on Gaussian spaces.} World Scientific Publishing Co. Pte. Ltd., Hackensack, NJ.

\bibitem{HLN1} Hu, Y.; Liu, Y. and Nualart, D. (2016). Rate of convergence and asymptotic error distribution of Euler approximation schemes for fractional diffusions. \textit{Ann. Appl. Probab.} \textbf{26}, no. 2, 1147-1207. 


\bibitem{JS} Jacod, J. and  Shiryaev, A. (2002). \textit{Limit theorems for stochastic processes. Second edition.} Vol. 288. Springer Science \& Business Media.


\bibitem{L} Lamperti, J. (1962). Semi-stable stochastic processes. \textit{Trans. Amer. Math. Soc.} \textbf{104} 62-78.


\bibitem{LL} Le\'on, J. and Lude\~na, C. (2007). Limits for weighted $p$-variations and likewise functionals of fractional diffusions with drift. \textit{Stochastic Process. Appl.} \textbf{117} no. 3, 271-296.

\bibitem{LT} Liu, Y. and Tindel, S. (2017). First-order Euler scheme for SDEs driven by fractional Brownian motions: the rough case. \textit{arXiv: 1703.03625}

\bibitem{NTU} Neuenkirch, A.;    Tindel, S. and   Unterberger, J. (2010).  Discretizing the fractional L\'evy area.  \textit{Stochastic Process.  Appl.} \textbf{120} no. 2, 223-254.

\bibitem{Norvaisa} Norvaisa, R. (2015). Weighted power variation of integrals with respect to a Gaussian process. \textit{Bernoulli} \textbf{21} no. 2, 1260-1288. 

\bibitem{Nourdin} Nourdin, I. (2008). Asymptotic behavior of some weighted quadratic and cubic variations of the fractional Brownian motion. \textit{Ann. Probab.} \textbf{36}   2159-2175.

\bibitem{NN} Nourdin, I. and Nualart, D. (2010). Central limit theorems for multiple Skorohod integrals. \textit{J. Theoret. Probab.} \textbf{23} no. 1, 39-64.

\bibitem{NNT} Nourdin, I.; Nualart, D. and Tudor, C. (2010). Central and non-central limit theorems for weighted power variations of fractional Brownian motion. \textit{Ann. Inst. Henri Poincar\'e Probab. Stat.} \textbf{46} no. 4, 1055-1079. 



 \bibitem{NP} Nourdin, I.  and   Peccati, G. (2012). \textit{Normal approximations with Malliavin calculus: from Stein's method to universality.} Vol. 192. Cambridge University Press.

\bibitem{NR} Nourdin, I. and R\'eveillac, A. (2008). Asymptotic behavior of weighted quadratic variations of fractional Brownian motion: The critical case $H=1/4$. \textit{Ann. Probab.}  \textbf{37} no. 6,   2200-2230.

\bibitem{NRS} Nourdin, I.; R\'eveillac, A. and Swanson, J.(2010). The weak Stratonovich integral with respect to fractional Brownian motion with Hurst parameter $1/6$. \textit{Electron. J. Probab.}  \textbf{15} no. 70,   2117-2162.

 \bibitem{N} Nualart, D. (2006). \textit{The Malliavin calculus and related topics. Second edition.}    Springer-Verlag, Berlin. 

\bibitem{R} Rosenblatt,   M. (1956). A central limit theorem and a strong mixing condition. \textit{Proc. Nat. Acad. Sci. U.S.A.} \textbf{42} 43-47.

\bibitem{T} Taqqu, M. (1979). Convergence of integrated processes of arbitrary Hermite rank. \textit{Z. Wahrsch. verw. Gebiete} \textbf{50}, 53-83.

\bibitem{Tu} Tudor, C. A. (2013). \textit{Analysis of variations for self-similar processes. A stochastic calculus approach.} Probability and its Applications. Springer, Cham.  
 
\end{thebibliography}
\end{document}